\documentclass[12pt]{article}
\usepackage[utf8]{inputenc}

\usepackage[left=1cm, top=1cm, right=1cm, bottom=1cm, includeheadfoot, headheight=0pt, a4paper]{geometry}
\usepackage{fancyhdr}
\usepackage{lastpage}
\usepackage{amsmath}
\usepackage{amssymb}
\usepackage{graphicx}
\usepackage{enumitem}
\usepackage{multicol}
\usepackage{amsmath}
\usepackage{amsfonts}
\usepackage[utf8]{inputenc}
\usepackage{amsthm}
\usepackage{tikz-cd}
\usepackage{geometry}
\usepackage{mathtools}
\usepackage{tikz}
\usepackage{comment}
\usepackage{url}
\usetikzlibrary{patterns}
\usepackage{svg}
\usepackage{xcolor}
\usepackage[T1]{fontenc}
\usepackage[utf8]{inputenc}

\newtheorem{thm}{Theorem}[section]
\newtheorem{lemma}[thm]{Lemma}

\newtheorem{prop}[thm]{Proposition}

\newtheorem{cor}[thm]{Corollary}

\newtheorem{prob}[thm]{Problem}
\theoremstyle{remark}
\newtheorem{remark}[thm]{Remark}
\renewcommand{\thefootnote}{\fnsymbol{footnote}}

\newcommand\blfootnote[1]{%
  \begingroup
  \renewcommand\thefootnote{}\footnote{#1}%
  \addtocounter{footnote}{-1}%
  \endgroup
}

\title{On numerical semigroups with embedding dimension four}

\author{Kazimierz Chomicz}
\date{April 2026}
\begin{document}

\maketitle

\abstract{We develop a geometric procedure for finding the Ap{\'e}ry set of any numerical semigroup with embedding dimension four. Previous methods of comparable strength worked only for embedding dimension three or under very specific conditions. We illustrate our method by finding the Frobenius numbers, genera, Betti elements, minimal presentations, and catenary degrees of numerical semigroups generated by four consecutive squares and by four consecutive triangular numbers.}

\section{Introduction}

\blfootnote{\text{Keywords: numerical semigroup, embedding dimension, Frobenius number, genus, catenary degree, Betti elements. }}
\blfootnote{\text{2020 Mathematics Subject Classification: 11D07, 20M13, 20M14.}}

A subset of $\mathbb{N}$ that contains 0, is closed under addition and is missing only finitely many positive integers is called a \emph{numerical semigroup}. The Frobenius problem, also known as the Coin problem or the Chicken McNugget problem, concerns finding the largest integer not in a numerical semigroup. This integer is called the \emph{Frobenius number} and is denoted $F(S)$, for a numerical semigroup $S$.

For a subset $A$ of $\mathbb{N}$, we let $\langle A \rangle$ denote the set of all linear combinations of elements in $A$  with nonnegative coefficients. It is well known that $\langle A \rangle$ is a numerical semigroup if and only if $\gcd(A)=1$ (for instance, see \cite[Lemma 2.1]{Rosales2009}). We define the integer e$(S)$, called the \emph{embedding dimension}, as the cardinality of the unique set of minimal generators of a numerical semigroup $S$. 

An explicit formula for $F(S)$ is well known in the case of a numerical semigroup with two generators (i.e. e$(S) =2$):
$$ F(\langle a,b\rangle) = ab-a-b,$$
obtained by Sylvester in 1883 (see \cite{SylvesterSharp1883}). However, for e$(S) \geq 3$, the Frobenius problem remains open. Indeed, Curtis showed in \cite{Curtis1990} that it is impossible to find a polynomial formula (that is, a finite set of polynomials) to compute the Frobenius number if e$(S) = 3$. Moreover,  Ram{\'i}rez-Alfons{\'i}n \cite{Ramrez1996} proved that the question is NP-hard. Despite this, for the case of three generators, a lot of useful results were proved, with Johnson \cite{Johnson1960}, and Brauer and Shockley \cite{Brauer1962} being among the fundamental ones (for more information we refer the reader to \cite{RamirezAlfonsin2005}).

Many formulas for special types of numerical semigroups have also been established, such as those generated by arithmetic \cite{Brauer1942,Roberts1956} or almost arithmetic sequences \cite{Lewin1975,Selmer1977}, by geometric sequences \cite{Ong2008}, by three consecutive squares or cubes \cite{Felsquaresandcubes,Lepilov2015}, by three Fibonacci numbers \cite{Marin2007}, by squares of three consecutive Fibonacci numbers \cite{Fibsquares2025}, by triangular and tetrahedral numbers \cite{triangular}, and many others.

Numerical semigroups generated by infinite sequences were also considered, including infinite sequences of squares \cite{Moscariello2015} and of other $k$th powers \cite{Arias2025}, and by infinite quadratic sequences \cite{quadraticseq}. See also the references in \cite{Arias2025} and  \cite{quadraticseq}.

In this paper, we will be interested in finding the Frobenius numbers of numerical semigroups generated by four consecutive squares and four consecutive triangular numbers. The corresponding problems for three generators have been solved in \cite{Felsquaresandcubes,Lepilov2015} and \cite{triangular}.  

In Section \ref{geometricprocedure}, we develop a method for finding the Ap{\'e}ry set of any numerical semigroup with embedding dimension four, from which the Frobenius number can be obtained. Our approach directly generalizes the classical Brauer-Shockley framework \cite{Brauer1962}, and is the first method to achieve results of comparable strength in embedding dimension four. This is the main result of this paper. 

In Section \ref{bettielements}, we build on the results of the previous section to find the catenary degree. It plays a crucial role in factorization theory, and recently in \cite{Geroldinger2025}, it was shown that there is no polynomial formula for the catenary degree, which is valid for all numerical semigroups (similarly to the Frobenius number \cite{Curtis1990}). In the process of finding the catenary degree, we will find the set of Betti elements and a minimal presentation.
For more information about this topic, we refer the reader to \cite[Chapter 6]{AssiDAnnaGarciaSanchez2020}. 

In Sections \ref{findingtheFnumber} and \ref{triangularnumbers}, mostly as a showcase, we apply the tools developed throughout the previous sections to numerical semigroups generated by four consecutive squares and four consecutive triangular numbers. Similar results for a different family of numerical semigroups with embedding dimension four were obtained in \cite{triangular}. In particular, the authors studied the Frobenius numbers, Ap{\'e}ry sets, Betti elements, and minimal presentations of numerical semigroups generated by four consecutive tetrahedral numbers. However, the proofs rely heavily on the property --- defined there --- of being telescopic, and cannot be applied for most numerical semigroups. Specifically, they cannot be applied for four consecutive triangular numbers, which is possible using our methods.

We finish the paper with remarks which may stimulate further research.

To obtain many identities that play a crucial role in Sections \ref{findingtheFnumber} and \ref{triangularnumbers}, we used \texttt{Mathematica} \cite{Mathematica}. To verify the resulting formulas for semigroup invariants, we used \texttt{numericalsgps} \cite{NumericalSgps} \texttt{GAP} \cite{GAP4} package.

\section{Geometric procedure}\label{geometricprocedure}
\subsection{Preliminaries}
For a numerical semigroup $S$ and an element $a \in S$, we define the \emph{Ap{\'e}ry set} of $S$ with respect to $a$ as
$$ \text{Ap}(S,a) = \{ s \in S \mid s-a \notin S\}$$
(see \cite{Apery1946}). We have $\# \text{Ap}(S,a) = a$, because the set contains exactly one element from each residue class modulo $a$. This set is of interest to us because from the Ap{\'e}ry set one can obtain several semigroup invariants. Specifically, $F(S) = \max \{\text{Ap}(S,a)\} - a$, and $G(S) = \frac{1}{a} \left( \sum_{s \in \text{Ap}(S,a)} s \right) - \frac{a-1}{2}$, where $G(S)$ is the cardinality of the set $\mathbb{N} \setminus S$ and is called the \emph{genus} (see \cite{Selmer1977}).

A procedure for finding the Ap{\'e}ry set of numerical semigroups with embedding dimension three has been well established (see \cite{Brauer1962}). The procedure is geometric in nature and works as follows. For a numerical semigroup $\langle d_1,d_2,d_3 \rangle$ with $\gcd(d_i,d_j) = 1$ for $i \neq j$, we find the smallest positive integers $a_{11},a_{22},a_{33}$ such that: 
\begin{align*}
    a_{11} d_1 &= a_{12} d_2 + a_{13} d_3,  \\
    a_{22} d_2 &= a_{21} d_1 + a_{23} d_3, \\
    a_{33} d_3 &= a_{31} d_1 + a_{32} d_2,
\end{align*}
for some nonnegative integers $a_{12},a_{13},a_{21},a_{23},a_{31},a_{32}$.
The above equations are called the \emph{minimal relations} of $\langle d_1,d_2,d_3 \rangle$.

Next, we consider a collection of unit squares in the first quadrant, placed against the $x$- and $y$-axes. 
We label the square, whose vertex closest to the origin has the coordinates $(i,j)$, with the value $id_2 + jd_3$. Using the minimal relations, we construct the points $(a_{22},0)$, $(0,a_{33})$, $(a_{12},a_{13})$ and remove the regions $\{x > a_{22}, y>0\}$, $\{x > 0, y>a_{33}\}$, $\{x > a_{12}, y>a_{13}\}$ (see Figure \ref{img3gen}). In \cite{Brauer1962}, it has been proven that the labels remaining in the resulting figure form exactly the set $\text{Ap}(\langle d_1,d_2,d_3 \rangle,d_1)$.

Furthermore, the maximal number of the Ap{\'e}ry set clearly labels one of the corner squares of Figure \ref{img3gen} (colored light gray). Now, using $F(S) = \max(\text{Ap}(S,a)) - a$, this yields the following formula for the Frobenius number (see \cite{Brauer1962} for details). 
\begin{equation}\label{3genF}
 F(\langle d_1,d_2,d_3 \rangle) = \max\{ (a_{22}-1)d_2 + (a_{13}-1)d_3,(a_{12}-1)d_2 + (a_{33}-1)d_3 \} - d_1.
\end{equation}

\begin{figure}[h]
    \caption{a diagram finding $\text{Ap}(\langle d_1,d_2,d_3 \rangle,d_1)$}
    \centering
    \label{img3gen}
    \begin{tikzpicture}

\fill[gray!25] (7.5,4) rectangle (6,3.5);
\fill[gray!25] (12,2) rectangle (10.5,1.5);

\draw[line width=1.5pt] (0,0) -- (0,4);
\draw[line width=1.5pt] (7.5,2) -- (7.5,4);
\draw[line width=1.5pt] (12,0) -- (12,2);
\draw[line width=1.5pt] (0,0) -- (12,0);
\draw[line width=1.5pt] (7.5,2) -- (12,2);
\draw[line width=1.5pt] (0,4) -- (7.5,4);
\draw[dashed] (0,4) -- (0,4.5);
\draw[dashed] (1.5,4) -- (1.5,4.5);
\draw[dashed] (9,2) -- (9,2.5);
\draw[dashed] (7.5,2.5) -- (9,2.5);
\draw (1.5,0) -- (1.5,4);
\draw (3,0) -- (3,4);
\draw (4.5,0) -- (4.5,4);
\draw (6,0) -- (6,4);
\draw (7.5,0) -- (7.5,2);
\draw (9,0) -- (9,2);
\draw (10.5,0) -- (10.5,2);
\draw[dashed] (13.5,0) -- (13.5,0.5);
\draw[dashed] (0,4.5) -- (1.5,4.5);
\draw[dashed] (12,0) -- (13.5,0);
\draw[dashed] (12,0.5) -- (13.5,0.5);
\draw (0,0.5) -- (12,0.5);
\draw (0,1) -- (12,1);
\draw (0,1.5) -- (12,1.5);
\draw (0,2) -- (7.5,2);
\draw (0,2.5) -- (7.5,2.5);
\draw (0,3) -- (7.5,3);
\draw (0,3.5) -- (7.5,3.5);

\node at (0.75,0.25) {$0$};
\node at (2.25,0.25) {$d_3$};
\node at (3.75,0.25) {$2d_3$};
\node at (5.25,0.25) {$3d_3$};
\node at (8.25,0.25) {$a_{13}d_3$};
\node at (12.75,0.25) {$a_{33}d_3$};
\node at (0.75,0.75) {$d_2$};
\node at (0.75,1.25) {$2d_2$};
\node at (0.75,2.25) {$a_{12}d_2$};
\node at (0.75,4.25) {$a_{22}d_2$};
\node at (8.25,2.25) {$a_{11}d_1$};
\node at (2.25,0.75) {$d_2+d_3$};
\node at (2.25,1.25) {$2d_2+d_3$};
\node at (3.75,0.75) {$d_2+2d_3$};
\end{tikzpicture}
\end{figure}
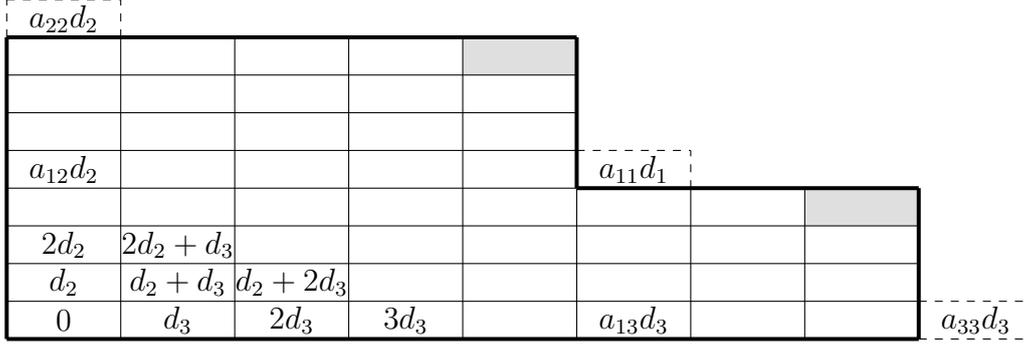

\vspace{10pt}
Killingbergtr{\o}, in \cite{Killingbergtro2000}, considered a possible generalization of the above procedure to numerical semigroups with embedding dimension four. Among other results, he computed the Frobenius number of a randomly chosen numerical semigroup $\langle 103,133,165,228\rangle$, using a geometric procedure, to find the  Ap{\'e}ry set with respect to 103.  His method for a semigroup $\langle d_0,d_1,d_2,d_3\rangle$ involves constructing a three-dimensional figure (Figure \ref{imgnorweg}), consisting of a collection of unit cubes in the first octant, placed against the $xy$-, $yz$-, and $zx$‑planes, where the cube whose vertex closest to the origin has the coordinates $(i,j,k)$, is labeled with the value $id_1 + jd_2 + kd_3$. Similarly to the method for three generators, we then find some special set of points and for each of them having coordinates $(a,b,c)$, we remove the region $\{x > a,y>b,z>c\}$. We do this for the set of labels of the resulting figure to end up being the Ap{\'e}ry set with respect to $d_0$. We will often refer to this set as simply the Ap{\'e}ry set for short.   

The set of points used by Killingbergtr{\o} (and its counterpart for other numerical semigroups) is in general not sufficient to produce the Ap{\'e}ry set. We propose a modified approach that works for any numerical semigroup with embedding dimension four.

\begin{figure}[h]
    \caption{the figure from \cite{Killingbergtro2000}}
    \centering
    \includegraphics[width = 0.58\textwidth]{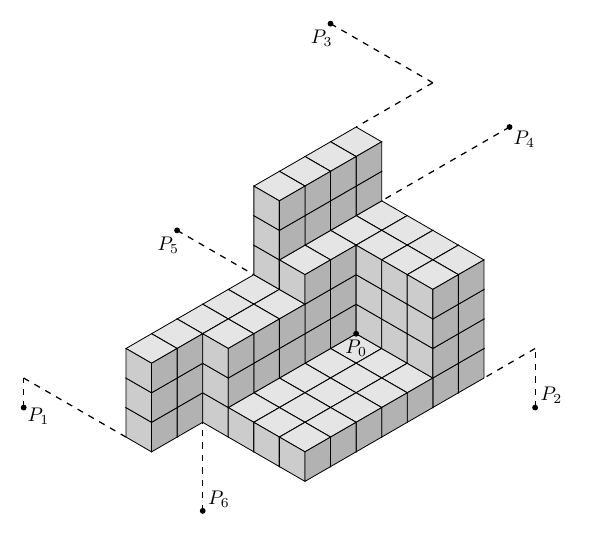}
    \label{imgnorweg}
\end{figure}

\subsection{Procedure}
Let $\langle d_0,d_1,d_2,d_3 \rangle$ be a numerical semigroup with embedding dimension four. 

The remainder of this section develops a method for finding an Ap{\'e}ry set of $\langle d_0,d_1,d_2,d_3 \rangle$. Theorem \ref{theoremprocedure} provides the fundamental result that makes this possible, although it is not directly practical on its own and the idea is not new (see \cite{Kannan1992}). Propositions \ref{proptrick} and \ref{propnottrick} are the key results that make our method innovative and useful in practice. We also include further results in Propositions \ref{propminimalrelations} and \ref{propL-shapes}.

Define the \emph{minimal relations} of $d_0,d_1,d_2,d_3$ respectively:
\begin{align}
    a_{00} d_0 &= a_{01} d_1 + a_{02}d_2 + a_{03} d_3,  \label{minrelation-0} \\
    a_{11} d_1 &= a_{10} d_0+a_{12} d_2 + a_{13} d_3,  \\
    a_{22} d_2 &=a_{20} d_0+ a_{21} d_1 + a_{23} d_3,  \\
    a_{33} d_3 &= a_{30} d_0 +a_{31} d_1 + a_{32} d_2, 
\end{align}
where $a_{00},a_{11}, a_{22}, a_{33}$ are the smallest positive integers such that the above is true for some nonnegative integers $a_{ij}$.

Additionally, define the following relations:
\begin{align}
    b_{11} d_1 &= b_{10} d_0+b_{12} d_2 + b_{13} d_3,  \\
    b_{22} d_2 &= b_{20} d_0+ b_{21} d_1 + b_{23} d_3,  \\
    b_{33} d_3 &= b_{30} d_0 +b_{31} d_1 + b_{32} d_2, 
\end{align}
where $b_{11}, b_{22}, b_{33}$ are the smallest positive integers such that the above is true for some nonnegative integers $b_{ij}$, and specifically with $b_{10}, b_{20}, b_{30}$ being positive. We call these the \emph{$d_0$-positive minimal relations} of $d_1,d_2,d_3$ respectively. Using the above coefficients, we construct the points
\begin{equation}\label{positivepoints}
 P_1=(b_{11},-b_{12},-b_{13}), \ P_2=(-b_{21},b_{22},-b_{23}), \ P_3 = (-b_{31}, -b_{32}, b_{33}),
\end{equation} 
each of which satisfies that $xd_1 + yd_2 + zd_3$ is a positive multiple of $d_0$ --- here, and for the rest of this section, $x,y,z$ will always denote the $x$-, $y$- and $z$- coordinates of a point in question.

Note that in the minimal relations (resp. $d_0$-positive minimal relations) $a_{ii}$ (resp. $b_{ii}$) is unique, however $a_{ij}$ (resp. $b_{ij}$) for $i \neq j$ is not necessarily unique.

Consider unit cubes in $\mathbb{R}^3$. Each unit cube $[i,i+1] \times [j,j+1] \times [k,k+1] \in \mathbb{R}^3$ has natural coordinates $(i,j,k) \in \mathbb{N}^3$ and is labeled with the value $id_1 + jd_2 + kd_3$. We will refer to the collection of all such labeled cubes, as \emph{the initial collection of cubes}. We denote the unit cube $[i,i+1] \times [j,j+1] \times [k,k+1]$ by $[[i,j,k]]$. For a point $P =(a,b,c) \in \mathbb{Z}^3$, we call the region $\{x > a, y>b, z>c\}$, \emph{the region associated to} $P$.

\begin{lemma}\label{1remlemma}
    Deleting from the initial collection of cubes the region associated to a point $P = (\mu_1,\mu_2,\mu_3) \in \mathbb{Z}^3$, satisfying $\mu_1d_1 + \mu_2d_2 + \mu_3d_3 = \mu d_0$ for some positive integer $\mu$, does not remove any cube labeled with an element of the Ap{\'e}ry set with respect to $d_0$. 
\end{lemma}
\begin{proof}
    Note that the Ap{\'e}ry set with respect to $d_0$ is the set of the smallest nonnegative integers from each residue class modulo $d_0$ that are in the numerical semigroup $\langle d_1,d_2,d_3\rangle$. Suppose, to the contrary, that for some residue $r \in \{0,1,\dots,d_0-1\}$, a cube labeled with the smallest number $\lambda d_0 + r$ ($\lambda \in \mathbb{N}$) that is in $\langle d_1,d_2,d_3\rangle$, has been deleted. Let $[[\lambda_1,\lambda_2,\lambda_3]]$ be the deleted cube, which was labeled with $\lambda d_0 + r$. Since it is deleted by the region associated to $P$, we have $\lambda_1 \geq \mu_1$, $\lambda_2 \geq \mu_2$, $\lambda_3 \geq \mu_3$. We get
    \begin{equation}\label{1remlemmaeq}
     (\lambda_1 -\mu_1) d_1 + (\lambda_2-\mu_2) d_2 + (\lambda_3-\mu_3) d_3 = (\lambda-\mu) d_0 + r, 
    \end{equation}
    where all coefficients on the left-hand side are nonnegative. This means that the value \eqref{1remlemmaeq}, is congruent to $r$ modulo $d_0$, is nonnegative, is in $\langle d_1,d_2,d_3\rangle$, and is strictly less than $\lambda d_0 + r$, which contradicts the minimality of $\lambda d_0 + r$.
\end{proof}

\begin{thm}\label{theoremprocedure}
     Delete from the initial collection of cubes the regions associated to the points \eqref{positivepoints}, and for every triple $(a_{01},a_{02},a_{03}) \in \mathbb{N}^3$ satisfying \eqref{minrelation-0} delete the region associated to $(a_{01},a_{02},a_{03})$. Additionally, for all identities of the form
    $$ \lambda_0d_0 + \lambda_1d_1=\lambda_2d_2 + \lambda_3d_3,  $$
    with $\lambda_i \in \mathbb{Z}_{>0}$, $\lambda_2 < b_{22}$, and $\lambda_3 < b_{33}$, remove the region associated to the point $(-\lambda_1,\lambda_2,\lambda_3)$. For all identities of the form
    $$ \lambda_0d_0 + \lambda_2d_2=\lambda_1d_1 + \lambda_3d_3,  $$
    with $\lambda_i \in \mathbb{Z}_{>0}$, $\lambda_1 < b_{11}$, and $\lambda_3 < b_{33}$, remove the region associated to the point $(\lambda_1,-\lambda_2,\lambda_3)$. Finally, for all identities of the form
    $$ \lambda_0d_0 + \lambda_3d_3=\lambda_1d_1 + \lambda_2d_2,  $$
    with $\lambda_i \in \mathbb{Z}_{>0}$, $\lambda_1 < b_{11}$, and $\lambda_2 < b_{22}$, remove the region associated to the point $(\lambda_1,\lambda_2,-\lambda_3)$.
    Call the remaining collection of cubes the figure $R$. Then, the set of all cubes from the initial collection of cubes which are labeled with an element of the Ap{\'e}ry set is the figure $R$.
\end{thm}
\begin{proof}
The condition $\lambda_i < b_{ii}$ ensures that we consider only the identities that matter --- the points \eqref{positivepoints} already delete the regions $\{x > b_{11}\}$, $\{y > b_{22}\}$, $\{z > b_{33}\}$. Lemma \ref{1remlemma} ensures that every cube labeled with an element of the Ap{\'e}ry set is in $R$. We need to prove that only the cubes labeled with elements of the Ap{\'e}ry set are in $R$. Suppose, to the contrary, that for some remainder $r \in \{0,1,\dots,d_0-1\}$, $\lambda d_0 + r$ and $\lambda' d_0 + r$, with $\lambda' > \lambda$, are labels in $R$. Now for some $\lambda_1, \lambda_2, \lambda_3,\lambda_1', \lambda_2', \lambda_3' \in \mathbb{N}$ we have
    \begin{equation}\label{eq1theorem}
        \lambda d_0 + r = \lambda_1d_1 + \lambda_2d_2 + \lambda_3d_3,
    \end{equation}
     \begin{equation}\label{eq2theorem}
        \lambda' d_0 + r = \lambda_1'd_1 + \lambda_2'd_2 + \lambda_3'd_3,
    \end{equation}
    with $[[\lambda_1,\lambda_2,\lambda_3]], [[\lambda_1',\lambda_2',\lambda_3']] \in R$.
    Subtracting \eqref{eq1theorem} from \eqref{eq2theorem} we get
    \begin{equation}\label{eq3theorem}
        (\lambda' - \lambda)d_0 = (\lambda_1' - \lambda_1)d_1 + (\lambda_2' - \lambda_2)d_2 + (\lambda_3' - \lambda_3)d_3,
    \end{equation}
    where $(\lambda' - \lambda)>0$, by assumption. Consider cases of how many coefficients on the right-hand side of \eqref{eq3theorem} are positive. The left-hand side is positive, hence, on the right-hand side, at least one of the coefficients is positive. 
    
If exactly one coefficient is positive, assume without loss of generality that $(\lambda_1' - \lambda_1)$ is the one. Moving the non positive integers to the other side gives
    \begin{equation}\label{eq4theorem}
        (\lambda' - \lambda)d_0  + (\lambda_2 - \lambda_2')d_2 + (\lambda_3 - \lambda_3')d_3= (\lambda_1' - \lambda_1)d_1.
    \end{equation} 
Since all the above coefficients are nonnegative and $(\lambda_1' - \lambda_1)$ and $(\lambda' - \lambda) $ are positive, by the minimality of $b_{11}$, we get $b_{11} \leq (\lambda_1' - \lambda_1) \leq \lambda_1'$. This is a contradiction with $[[\lambda_1',\lambda_2',\lambda_3']] \in R$.

If exactly two coefficients on the right-hand side of \eqref{eq3theorem} are positive, without loss of generality, suppose that $(\lambda_1' - \lambda_1)$ and $(\lambda_2' - \lambda_2)$ are the two. Moving $(\lambda_3' - \lambda_3)$ to the other side gets
    \begin{equation}\label{eq5theorem}
        (\lambda' - \lambda)d_0  + (\lambda_3 - \lambda_3')d_3= (\lambda_1' - \lambda_1)d_1  + (\lambda_2' - \lambda_2)d_2.
    \end{equation}
The above identity satisfies the conditions of the current theorem, thus the region $\{x > (\lambda_1' - \lambda_1), y > (\lambda_2' - \lambda_2)\}$ was deleted. Again, this contradicts $[[\lambda_1',\lambda_2',\lambda_3']] \in R$, as $\lambda'_1 \geq (\lambda_1' - \lambda_1)$ and $\lambda'_2 \geq (\lambda_2' - \lambda_2)$.

Now, only the case where all the coefficients in \eqref{eq3theorem} are positive is left. Let us substitute $\mu =(\lambda' - \lambda)  > 0$ and $\mu_i =(\lambda_i' - \lambda_i)  >0$ for $i = 1,2,3$. We have
    \begin{equation}\label{eq6theorem}
         \mu d_0 = \mu_1d_1 + \mu_2d_2 + \mu_3d_3,
    \end{equation}
and by the process of deleting the regions, the cube $[[\mu_1,\mu_2,\mu_3]]$ is also in $R$, as $\mu_i \leq \lambda'_i$ for $i = 1,2,3$. 
From the minimality of $a_{00}$, we get $a_{00} \leq \mu$. Now we subtract $a_{00}d_0 = a_{01}d_1 + a_{02}d_2 + a_{03}d_3$ from \eqref{eq6theorem} to get
\begin{equation}\label{eq7theorem}
    (\mu - a_{00}) d_0 = (\mu_1 - a_{01})d_1 +  (\mu_2 - a_{02})d_2 + (\mu_3 - a_{03})d_3.
\end{equation}
    
If $(\mu - a_{00}) = 0$, then 
the region associated to $(\mu_1,\mu_2,\mu_3)$ had been deleted, which contradicts $[[\mu_1,\mu_2,\mu_3]] \in R$. 

Finally, when $(\mu - a_{00}) > 0$ in \eqref{eq7theorem}, we repeat the same case by case consideration as we did with \eqref{eq3theorem}, which results in all the coefficients in \eqref{eq7theorem} being positive. However, this implies that the cube $[[\mu_1,\mu_2,\mu_3]]$, has been deleted by the region associated to $(a_{01}, a_{02}, a_{03})$, which is a contradiction. 
\end{proof}

In general, finding the minimal relations and all the identities used in Theorem \ref{theoremprocedure} is a rather difficult task. In order to use Theorem \ref{theoremprocedure} effectively, we need some additional lemmas.

\begin{lemma}\label{lemmaminrelation1}
    Let $R$ be as in Theorem \ref{theoremprocedure}. Delete from $R$ the region $\{x > a_{11}\}$. Then every element of the Ap{\'e}ry set still labels some remaining cube. The same will be true if we instead delete the region $\{y > a_{22}\}$ or $\{z > a_{33}\}$.
\end{lemma}
\begin{proof}
If $a_{11} = b_{11}$, there is nothing to prove. Assume $a_{11} < b_{11}$, which implies $a_{10} = 0$. Consider an element of the Ap{\'e}ry set, which has the form $\lambda_1d_1 + \lambda_2d_2 + \lambda_3d_3$, where $\lambda_i \in \mathbb{N}$ and $\lambda_1 \geq a_{11}$. Then there exists a positive integer $\lambda$ such that $a_{11} > (\lambda_1 - \lambda a_{11}) \geq 0$. Using $a_{11}d_1 = a_{12}d_2 + a_{13}d_3$, we obtain
    \begin{equation}
        \lambda_1d_1 + \lambda_2d_2 + \lambda_3d_3 = (\lambda_1 - \lambda a_{11})d_1 + (\lambda_2+\lambda a_{12})d_2 + (\lambda_3+\lambda a_{13})d_3.
    \end{equation}
As every cube labeled with an element of the Ap{\'e}ry set is in $R$, we have $[[\lambda_1 - \lambda a_{11},\lambda_2+\lambda a_{12},\lambda_3+\lambda a_{13} ]] \in R \cap \{x \leq a_{11}\}$ and the claim follows.   
\end{proof}

\begin{lemma}\label{lemmaminrelation2}
    Let $R$ be as in Theorem \ref{theoremprocedure}. Delete from $R$ the regions $\{x > a_{11}\}$, $\{y > a_{22}\}$, and $\{z > a_{33}\}$. Then, in the resulting collection, each element of the Ap{\'e}ry set labels at most one cube.
\end{lemma}
\begin{proof}
     Suppose, to the contrary, that for some $\lambda_1, \lambda_2, \lambda_3,\lambda_1', \lambda_2', \lambda_3' \in \mathbb{N}$ with $\lambda_i ,\lambda'_i < a_{ii}$ for $i = 1,2,3$, we have
    \begin{equation}\label{propeq1}
        \lambda_1d_1 + \lambda_2d_2 + \lambda_3d_3 = \lambda_1'd_1 + \lambda_2'd_2 + \lambda_3'd_3 
    \end{equation}
    \begin{equation}\label{propeq2}
        \implies  (\lambda_1' - \lambda_1)d_1 + (\lambda_2' - \lambda_2)d_2 + (\lambda_3' - \lambda_3)d_3 = 0.
    \end{equation}
    If all coefficients in \eqref{propeq2} are zero, we have nothing to prove. Therefore, either one or two coefficients in \eqref{propeq2} are positive. Regardless, without loss of generality, we can assume that
    \begin{equation}\label{propeq3}
         (\lambda_1' - \lambda_1)d_1 = (\lambda_2 - \lambda_2')d_2 + (\lambda_3 - \lambda_3')d_3,
    \end{equation}
    where all the coefficients in the above equation are nonnegative. However, $(\lambda_1' - \lambda_1)$ is positive, because otherwise all the coefficients would be zero. This implies $a_{11} \leq (\lambda_1' - \lambda_1) \leq \lambda_1'$, which is a contradiction.
\end{proof}

\begin{prop}\label{proptrick}
    Suppose that $a_{ii} = b_{ii}$ for at least two distinct indices $i \in \{1,2,3\}$. Let $R$ be as in Theorem \ref{theoremprocedure}. Delete from $R$
    the regions $\{x > a_{11}\}$, $\{y > a_{22}\}$ and $\{z > a_{33}\}$. Then each element of the Ap{\'e}ry set labels exactly one cube in the resulting collection.
\end{prop}

\begin{proof}
    If $a_{ii} = b_{ii}$ for all $i = 1,2,3$, then Lemma \ref{lemmaminrelation2} gives the result. Otherwise, we have $a_{ii} < b_{ii}$ for exactly one $i \in \{1,2,3\}$. In such case, Lemma \ref{lemmaminrelation1} together with Lemma \ref{lemmaminrelation2} give the result.
\end{proof}
\noindent
Define an \emph{$L$-shape} as a subset of the initial collection of cubes such that:
\begin{itemize}
    \item all its labels are elements of the Ap{\'e}ry set,
    \item every element of the Ap{\'e}ry set labels exactly one cube in it,
    \item if a cube $[[a,b,c]]$, $(a,b,c) \in \mathbb{N}^3$, is not in it, then neither is any cube in the region associated to $(a,b,c)$.
\end{itemize}
In particular, the collection from Proposition \ref{proptrick} is an $L$-shape. This notation is inherited from \cite{AGUILOGOST2015} and is established in the literature.

It turns out that when the conditions of Proposition \ref{proptrick} are met, we can employ the strategy of essentially guessing the Ap{\'e}ry set. We already mentioned in the introduction that a general polynomial formula for the Frobenius number does not exist. When we seek the Frobenius number of a special family of numerical semigroups, for which we suspect the existence of a polynomial formula, we also expect that the corresponding Ap{\'e}ry set --- and also the figure $R$ --- will also exhibit a certain regularity. To test these conjectures, we can examine numerical examples with the aid of an algebraic calculator, thus clarifying the behavior of the family in question.
If a family of numerical semigroups is well-behaved enough, the minimal relations and the identities from Theorem \ref{theoremprocedure} should have a form that we should be able to predict. 

If the conditions of  Proposition \ref{proptrick} are satisfied, it tells us 
that we are able to remove enough regions to leave an $L$-shape. This is crucial since an $L$-shape has exactly $d_0$ cubes. In this way, we do not need to prove any properties of the equations we have found, but we can just use Lemmas \ref{1remlemma} and \ref{lemmaminrelation1}, to remove the corresponding regions and leave the Ap{\'e}ry set behind. If we are left with exactly $d_0$ cubes, it means that we have found an $L$-shape. We will exploit this idea in Sections \ref{findingtheFnumber} and \ref{triangularnumbers}.   

\begin{remark}
    If the conditions of Proposition \ref{proptrick} are not met with respect to $d_0$, we could, for example, look for the Ap{\'e}ry set with respect to $d_1$ instead, if for such generator the conditions for applying Proposition \ref{proptrick} are satisfied. We will see that in Sections \ref{findingtheFnumber} and \ref{triangularnumbers}, such a switch was possible or was not needed. However, it is possible that we cannot do this. Consider the numerical semigroup $\langle 22,38,55,57 \rangle$. It has embedding dimension four and the minimal relations can be seen to be: 
    \begin{align*}
        5 \cdot 22 = 0\cdot 38 + 2 \cdot 55 + 0\cdot57, \\
        3 \cdot 38 = 0\cdot 22 + 0\cdot 55 +2 \cdot 57, \\
        2 \cdot 55 = 5 \cdot 22 + 0\cdot 38 + 0\cdot 57,\\
        2 \cdot 57 = 0\cdot 22 +3 \cdot 38 + 0\cdot 55.
    \end{align*}
    We would like to be able to apply the proposed method for finding the Ap{\'e}ry set to all numerical semigroups with embedding dimension four, not just to those which satisfy the assumptions of Proposition \ref{proptrick}.
\end{remark}

For the next lemma, let $b_{i,mm}$ denote the coefficient of $d_m$ in the $d_i$-positive minimal relation of $d_m$ (with this notation $b_{0,mm} = b_{mm}$).

\begin{lemma}\label{lemma2cases}
     There exists an arrangement $\{i,j,k,l\} = \{0,1,2,3\}$ such that $a_{mm} = b_{i,mm}$ for at least two distinct indices $m \in \{j,k,l\}$, or $a_{ii}d_i = a_{jj}d_j$ and $a_{kk}d_k = a_{ll}d_l$.
\end{lemma}

\begin{proof}
    Suppose that for every arrangement $\{i,j,k,l\} = \{0,1,2,3\}$ we have $a_{mm} < b_{i,mm}$ for at least two distinct indices $m \in \{j,k,l\}$. This implies that at least two of $a_{ji}, a_{ki},a_{li}$ are equal to zero for every such arrangement.

    Since for any arrangement $\{i,j,k,l\} = \{0,1,2,3\} $ at most two of $a_{ij}, a_{ik},a_{il}$ can be zero, it follows that for every index $i \in \{0,1,2,3\}$ every minimal relation of $d_i$ is of the form $a_{ii}d_i = a_{ij}d_j$ for some $j \in \{0,1,2,3\} \setminus \{i\}$. 

    Observe that $a_{ij} \geq a_{jj}$. If $a_{ij} > a_{jj}$, then
    \begin{equation}\label{notprop-eq1}
        a_{ii}d_i = a_{ij}d_j = ( a_{ij}-a_{jj})d_j + a_{jm}d_m,
    \end{equation}
    for some $m \in \{0,1,2,3\} \setminus \{j\}$. If $m = i$, subtracting $d_i$ from both sides of \eqref{notprop-eq1} contradicts the minimality of $a_{ii}$. If $m \neq i$, we get a contradiction with the fact that every minimal relation of $d_i$ has the form $a_{ii}d_i = a_{ij}d_j$. 
    
    Hence $a_{ij} = a_{jj}$, and therefore $a_{ii}d_i = a_{jj}d_j$. Repeating the same argument for the indices $\{k,l\} = \{0,1,2,3\} \setminus \{i,j\}$ yields the claim.
\end{proof}

If $a_{mm} = b_{i,mm}$ for at least two distinct indices $m \in \{j,k,l\}$ for some arrangement $\{i,j,k,l\} = \{0,1,2,3\}$, then clearly we can apply Proposition \ref{proptrick} to such semigroup (perhaps rearranging the generators). 

By Lemma \ref{lemma2cases}, if Proposition \ref{proptrick} cannot be applied to a numerical semigroups with embedding dimension four, then $a_{ii}d_i = a_{jj}d_j$ and $a_{kk}d_k = a_{ll}d_l$ for some arrangement $\{i,j,k,l\} = \{0,1,2,3\}$. Without loss of generality, we may assume that $a_{00}d_0 = a_{11}d_1$ and $a_{22}d_2 = a_{33}d_3$. For all such semigroups, the following proposition provides a result of the same strength as Proposition \ref{proptrick}.  

\begin{prop}\label{propnottrick}
    Suppose that $a_{11} = b_{11}$ and $a_{22}d_2 = a_{33}d_3$. Let $R$ be as in Theorem \ref{theoremprocedure}. Delete from $R$ the region $\{z > a_{33}\}$. Then the resulting figure is an $L$-shape.
\end{prop}
\begin{proof}
    We have an injection between the cubes of $R \cap \{ y \geq a_{22}\}$ and $R \cap \{ z \geq a_{33}\}$, which sends $[[\mu_1,\mu_2,\mu_3]] \in R \cap \{ y \geq a_{22}\} $ to $[[\mu_1,\mu_2 - a_{22},\mu_3 + a_{33}]] \in R \cap \{ z \geq a_{33}\}$. The cube $[[\mu_1,\mu_2 - a_{22},\mu_3 + a_{33}]]$ is in $R$ by Theorem \ref{theoremprocedure}, because it is labeled with an element of the Ap{\'e}ry set. Similarly, we can define an injection in the opposite direction, which is an inverse of the first map. Hence, the correspondence is a bijection. Because this bijection preserves labels, in order to obtain an $L$-shape we must delete either the region $\{y > a_{22}\}$ or $\{z > a_{33}\}$. We choose to delete $\{z > a_{33}\}$ and let $T$ be the resulting figure.   
     
    Suppose that some label in $T$ appears twice. This gives
    \begin{equation}\label{propnottrick-eq1}
        \mu_1d_1 + \mu_2d_2 + \mu_3d_3 = \nu_1d_1 + \nu_2d_2 + \nu_3d_3,
    \end{equation}
    for some $\mu_i, \nu_i \in \mathbb{N}$, and where $[[\mu_1,\mu_2,\mu_3]]$ and $[[\nu_1,\nu_2,\nu_3]]$ are two distinct cubes in $T$. Without loss of generality, assume $\mu_2 \geq \nu_2$. 
    
    If $\mu_1 \leq \nu_1$ and $\mu_3 \leq \nu_3$, we obtain
    \begin{equation}\label{propnottrick-eq2}
         (\mu_2-\nu_2)d_2  = (\nu_1-\mu_1)d_1  + (\nu_3-\mu_3)d_3,
    \end{equation}
    where all coefficients above are nonnegative, which implies $(\mu_2-\nu_2) \geq a_{22}$. If $\nu_3 \geq (\nu_3-\mu_3) \geq a_{33}$, we get a contradiction with $[[\nu_1,\nu_2,\nu_3]] \in T$, hence $(\nu_3-\mu_3) < a_{33}$. This gives
    \begin{equation}\label{propnottrick-eq21}
         (\mu_2-\nu_2-a_{22})d_2 + (a_{33}-\nu_3+\mu_3)d_3  = (\nu_1-\mu_1)d_1,
    \end{equation}
    where all coefficients above are nonnegative and $(a_{33}-\nu_3+\mu_3)$ is positive. Hence $(\nu_1-\mu_1)$ is also positive, which implies $\nu_1 \geq (\nu_1-\mu_1) \geq a_{11} = b_{11}$, contradicting $[[\nu_1,\nu_2,\nu_3]] \in T$.

    If $\mu_1 > \nu_1$ and $\mu_3 \leq \nu_3$, then
    \begin{equation}\label{propnottrick-eq3}
         (\mu_2-\nu_2)d_2 + (\mu_1-\nu_1)d_1  =  (\nu_3-\mu_3)d_3,
    \end{equation}
    which implies $\nu_3 \geq (\nu_3-\mu_3) \geq a_{33}$, so a contradiction with $[[\nu_1,\nu_2,\nu_3]] \in T$.

    Lastly, if $\mu_1 \leq \nu_1$ and $\mu_3 > \nu_3$, then
    \begin{equation}\label{propnottrick-eq4}
         (\mu_2-\nu_2)d_2 + (\mu_3-\nu_3)d_3  =  (\nu_1-\mu_1)d_1,
    \end{equation}
    which implies $\nu_1 \geq (\nu_1-\mu_1) \geq a_{11} = b_{11}$, so again a contradiction with $[[\nu_1,\nu_2,\nu_3]] \in T$.

    In total, every element of the Ap{\'e}ry set appears as a label in $T$, and each label in $T$ is distinct. Hence $T$ is an $L$-shape.
\end{proof}

\subsection{Additional results}

Suppose that from the initial collection of cubes we have removed regions whose removal satisfied the assumptions of Lemma \ref{1remlemma} or \ref{lemmaminrelation1}, ensuring that every element of the Ap{\'e}ry set still labels a cube in the resulting figure. Assume further that the resulting collection contains exactly $d_0$ cubes; this implies that we have obtained an $L$-shape. We would like to determine whether the equations we used --- specifically, relations of the form $\lambda_id_i = \lambda_jd_j + \lambda_kd_k + \lambda_ld_l$ --- are in fact the minimal relations or the $d_0$-positive minimal relations.

\begin{prop}\label{propminimalrelations}
     Suppose that we have deleted from the initial collection of cubes the regions associated to certain points of the forms $(-\lambda_1,\lambda_2,\lambda_3)$, $(\lambda_1,-\lambda_2,\lambda_3)$, $(\lambda_1,\lambda_2,-\lambda_3)$, with $\lambda_i \in \mathbb{Z}_{>0}$ and each satisfying that $xd_1 + yd_2 + zd_3$ is a positive multiple of $d_0$. Furthermore, we have deleted the regions associated to certain points $(a_{01},a_{02},a_{03}) \in \mathbb{N}^3$ satisfying \eqref{minrelation-0}. Finally, we have also deleted the regions $\{x > b_{11}'\}$, $\{y > b_{22}'\}$, and $\{z > a_{33}'\}$, where: 
    \begin{align*}
        b_{11}'d_1 &= b_{10}'d_0 + b_{12}'d_2 + b_{13}'d_3,   \\
        b_{22}'d_2 &= b_{20}'d_0 + b_{21}'d_1 + b_{23}'d_3,   \\
        a_{33}'d_3 &= a_{30}'d_0 + a_{31}'d_1 + a_{32}'d_2, 
    \end{align*}
    with $a_{ij}',b_{ij}' \in \mathbb{N}$, $b_{10}',b_{20}' >0$, and $a_{30}' + a_{31}'a_{32}' > 0$. If the resulting collection has exactly $d_0$ cubes, then it is an $L$-shape, $b_{11}' = b_{11} = a_{11}$, $b_{22}'=b_{22}=a_{22}$, and $a_{33}' = a_{33}$.
\end{prop}

\begin{proof}
    Lemmas \ref{1remlemma} and \ref{lemmaminrelation1} imply that the resulting collection is an $L$-shape. Denote it by $T$, and let $R$ be as in Theorem \ref{theoremprocedure}.
    Note that $R$ is contained in $T \cup \{z > a_{33}'\}$. Therefore, if 
    $a_{33}'>a_{33}$, Lemma \ref{lemmaminrelation1} tells us that every element of the Ap{\'e}ry set labels a cube in 
    $$
    R \cap \{z \leq a_{33}\} \subseteq  \left(T \cup \{z > a_{33}'\} \right)\cap \{z \leq a_{33}\} = T \cap \{z \leq a_{33}\}.
    $$  
    Now $T \cap \{z \leq a_{33}\}$ is strictly smaller than $T$, which contradicts the fact that $T$ is an $L$-shape, hence $a_{33}'=a_{33}$. For convenience, set $(a_{30},a_{31},a_{32}) = (a_{30}',a_{31}',a_{32}')$.

    Recall that $R$ is the set of all cubes labeled with an element of the Ap{\'e}ry set, according to Theorem \ref{theoremprocedure}, and that $b_{11}$ and $b_{22}$ are the boundary values of $R$ in the $x$- and $y$- directions, respectively. On the other hand, $b_{11}'$ and $b_{22}'$ are the boundary values of $T$ in the $x$- and $y$- directions, respectively, and we have $b_{11}' \geq b_{11}$ and $b_{22}' \geq b_{22}$. This yields $b_{11}' = b_{11}$ and $b_{22}' = b_{22}$.
    
    Moreover, we get $T = R \cap \{z \leq a_{33}\}$ and 
    the only claims left to prove are $a_{11} = b_{11}$ and $a_{22} = b_{22}$.
    
    Suppose that $a_{11} < b_{11}$. This gives $a_{11}d_1 = a_{12}d_2 + a_{13}d_3$ (for some $a_{12},a_{13} \in \mathbb{N}$), which is an element of the Ap{\'e}ry set, since it is a label in $T$. 
    If $a_{13} < a_{33}$, then $[[0,a_{12},a_{13}]] \in T$ (as $T =R \cap \{z \leq a_{33}\} $), which contradicts the fact that $T$ is an $L$-shape, as $[[a_{11},0,0]] \in T$ also. Hence $a_{13} \geq a_{33}$.

    Now, we can write
    \begin{equation}\label{propminimalrelations-eq1}
        a_{11}d_1 =  a_{12}d_2 + (a_{13}-a_{33})d_3 + a_{30}d_0 + a_{31}d_1 + a_{32}d_2.
    \end{equation}
    If $a_{31}d_1 > 0$, subtracting $d_1$ from both sides of \eqref{propminimalrelations-eq1} gives a contradiction with the minimality of $a_{11}$, unless $a_{12} =  (a_{13}-a_{33})=a_{30}=a_{32}=0$. However, this contradicts $a_{30} + a_{31}a_{32}=a_{30}' + a_{31}'a_{32}' > 0$. If $a_{30}d_0 > 0$, we get a contradiction, as we assumed $a_{11} < b_{11}$. 
   
   Hence, $a_{31}=a_{30}=0$ and $a_{33}d_3 = a_{32}d_2$. We have 
   $$a_{11}d_1 =  (a_{12} + \lambda a_{32})d_2 + (a_{13}-\lambda a_{33})d_3,$$ 
   where $\lambda \in \mathbb{Z}_{>0}$ is such that $ a_{33} > (a_{13}-\lambda a_{33}) \geq 0$ (recall $a_{13} \geq a_{33}$). This gives $[[0,a_{12} + \lambda a_{32},a_{13}-\lambda a_{33}]] \in T$, as $T =R \cap \{z \leq a_{33}\} $ and no cube labeled with an element of the Ap{\'e}ry set is outside of $R$ by Theorem \ref{theoremprocedure}.
   This contradicts the fact that $T$ is an $L$-shape, hence $a_{11} = b_{11} $. 

   The claim $a_{22}=b_{22}$ follows analogously.
\end{proof}
\begin{remark}\label{remark_propminimalrelations}
    Note that the semigroup from Proposition \ref{propminimalrelations} has $a_{11} = b_{11}$ and $a_{22}=b_{22}$, so it satisfies the assumptions of Proposition \ref{proptrick}. Moreover, if a numerical semigroup has $a_{11}=b_{11}$ and $a_{22}=b_{22}$, then there exists a representation $a_{33}d_3 = a_{30}d_0 + a_{31}d_1 + a_{32}d_2$ with $a_{30}, a_{31}, a_{32} \in \mathbb{N}$ and $a_{30} + a_{31}a_{32} > 0$. Otherwise, without loss of generality, suppose $a_{30} = a_{31} = 0$, which implies $a_{33}d_3 = a_{32}d_2$. Hence $a_{32} \geq a_{22} = b_{22}$ and $a_{33}d_3 = b_{20}d_0 + b_{21}d_1 + (a_{32}-b_{22})d_2$ is the desired representation (we have $b_{23} = 0$, as otherwise we would get a contradiction with the minimality of $a_{33}$).
    Consequently, Proposition \ref{propminimalrelations} applies precisely to all numerical semigroups that satisfy the assumptions of Proposition \ref{proptrick}.
\end{remark}

We end this section with a result connected to \cite{AGUILOGOST2015}, where the authors constructed a family of numerical semigroups with embedding dimension four that admits arbitrarily many related $L$-shapes. The following proposition characterizes when an Ap{\'e}ry set of a numerical semigroup with embedding dimension four admits exactly one $L$-shape. For the study of such semigroups in arbitrary embedding dimension, see \cite{Rosales2000}.

For the next proof, let the region associated to a \emph{cube} be the region associated to the vertex of the \emph{cube} that is closest to the origin.

\begin{prop}\label{propL-shapes}
    The Ap{\'e}ry set with respect to $d_0$ has a unique $L$-shape if and only if $a_{ii}=b_{ii}$ for $i=1,2,3$.
\end{prop}

\begin{proof}
    If $a_{ii}=b_{ii}$ for $i=1,2,3$, then Lemma \ref{lemmaminrelation2} implies that the figure $R$ is an $L$-shape. Moreover, $R$ is unique, and it is the set of all cubes from the initial collection of cubes that are labeled with an element of the Ap{\'e}ry set by Theorem \ref{theoremprocedure}. Hence the Ap{\'e}ry set admits exactly one $L$-shape.

    Suppose that $a_{ii}<b_{ii}$ for some $i \in \{1,2,3\}$. Then $R$ is not an $L$-shape, and some label appears more than once in $R$ --- as the label $a_{ii}d_i$ does. Let $\ell$ be the the smallest such label. Write 
    $$\ell =\lambda_1d_1 + \lambda_2d_2 + \lambda_3d_3 = \lambda_1'd_1 + \lambda_2'd_2 + \lambda_3'd_3,$$
    with $\lambda_i,\lambda_i' \in \mathbb{N}$, and $[[\lambda_1,\lambda_2,\lambda_3]]$, $[[\lambda_1',\lambda_2',\lambda_3']]$ being two distinct cubes in $R$. Without loss of generality, assume that $\lambda_1 > \lambda_1'$ and $\lambda_3 < \lambda_3'$. 

    Among all cubes labeled with $\ell$, choose the one with the largest $x$-coordinate; if several have the same maximal $x$-coordinate, choose among them the one with the largest $y$-coordinate (and only one such cube can exist, since they all have the same label). Delete from $R$ the regions associated to all other cubes labeled with $\ell$, leaving only the chosen one. Then proceed to the next smallest label that appears more than once, and repeat the process. We claim that this procedure eventually produces an $L$-shape.

    Suppose, to the contrary, that the procedure does not produce an $L$-shape. Then, at some stage, every cube labeled with a particular element is deleted. Let $a$ be the first element of the Ap{\'e}ry set whose label disappears entirely, and let $b$ be the label of the cubes whose associated regions deleted the last cube labeled with $a$. 
    
    Write $a =b+ \delta_1d_1 +\delta_2d_2+\delta_3d_3  $ with $\delta_i \in \mathbb{N}$, and let $b = \mu_1d_1 + \mu_2d_2 + \mu_3d_3$, where $[[\mu_1,\mu_2,\mu_3]]$ is the cube labeled with $b$ that survives our process --- such a cube exists because in our process when we preserve one cube with currently the smallest repeated label, this cube cannot be deleted by some region later because of the minimality of the label. We have 
    $$a = ( \delta_1+\mu_1)d_1 + ( \delta_2+\mu_2)d_2 + ( \delta_3+\mu_3)d_3,$$
    so the cube $[[\delta_1+\mu_1,\delta_2+\mu_2,\delta_3+\mu_3]]$ is labeled with $a$. Since all cubes labeled with $a$ were deleted, this cube also has been deleted by the region associated to some cube $[[\nu_1,\nu_2,\nu_3]]$. 
    
    However, there exists a cube $[[\nu_1',\nu_2',\nu_3']]$ with the same label as $[[\nu_1,\nu_2,\nu_3]]$ that was not deleted. Consequently, $a$ also labels the cube 
    $$[[\nu_1' + \delta_1+\mu_1 -\nu_1,\nu_2' + \delta_2+\mu_2 -\nu_2,\nu_3' + \delta_3+\mu_3 -\nu_3]].$$ 
    By construction, either its first coordinate exceeds $\delta_1+\mu_1$, or its first coordinate is equal to $\delta_1+\mu_1$ and its second coordinate is larger than $\delta_2+\mu_2$. Again, because $a$ is not a label in the final collection, this next cube we just constructed also has been deleted by some region. Repeating this argument produces cubes labeled by $a$ with arbitrarily large $x$- or $y$-coordinates, contradicting Theorem \ref{theoremprocedure} (as $R \subseteq \{x \leq a_{11}, y \leq a_{22}, z \leq a_{33}\}$). Thus, the procedure must produce an $L$-shape.

    Performing the same procedure, but choosing, at each step, the cube with maximal $z$-coordinate (and breaking ties by maximal $y$-coordinate), yields a different $L$-shape. It is indeed different, since by the assumption $\lambda_1 > \lambda_1'$ and $\lambda_3 < \lambda_3'$, the first cube selected in this variant is not the same as in the previous procedure.
\end{proof}

For an insight into the geometry of the Ap{\'e}ry set in arbitrary embedding dimension, we refer the reader to \cite{Aicardi2009} and \cite{Kannan1992}.

\section{Catenary degree}\label{bettielements}

In this section, we will be interested in describing the set of Betti elements, minimal presentations, and the catenary degree using the Ap{\'e}ry set, building on the results from Section \ref{geometricprocedure}. We start with some preliminary definitions and then proceed to the main results. To apply those results in Sections \ref{findingtheFnumber} and \ref{triangularnumbers}, Proposition \ref{propminimalrelations} will be essential. We do not have a version of Proposition \ref{propminimalrelations} for numerical semigroups that do not satisfy the assumption of Proposition \ref{proptrick} (see Remark \ref{remark_propminimalrelations}). As a consequence, in this section we will consider only numerical semigroups fulfilling those assumptions; see Problem \ref{problemcatenary}.

Results closely related to those of this section were obtained by Bresinsky in \cite{Bresinsky1988}, and in fact, his approach applies to all numerical semigroups of embedding dimension four. Our method, however, provides a different perspective; it explicitly relates the Betti elements and minimal presentations to the Apéry set, specifically to an $L$-shape --- a connection that is absent from \cite{Bresinsky1988}. Consequently, Bresinsky’s framework does not yield the results that we establish in Sections \ref{findingtheFnumber} and \ref{triangularnumbers}, which are based on the viewpoint developed here.

\subsection{Preliminaries}

Let $S$ be a numerical semigroup minimally generated by $\{d_0, d_1, \dots, d_k\} $. The \emph{factorization morphism} of $S$ is 
$$ \varphi : \mathbb{N}^{k+1} \to S, \quad \varphi(z_0, \dots, z_k) = z_0d_0 + \dots + z_kd_k.$$
The \emph{set of factorizations} of an element $s \in S$ is
$$ \varphi^{-1}(s) = \{(z_0,\dots,z_k)\in \mathbb{N}^{k+1} \mid z_0d_0 + \dots + z_kd_k = s\}.$$
If $z =(z_0,\dots,z_k) \in \varphi^{-1}(s)$, then the \emph{length} of $z$ is $|z| = z_0 + \dots + z_k$. The \emph{support} of $z$ is defined by
$$ \text{supp}(z) = \{i \in \{0,\dots,k\} \mid z_i > 0\}.$$
For $z = (z_0,\dots,z_k), z' = (z_0',\dots,z_k') \in \mathbb{N}^{k+1}$ we set
$$ \gcd(z,z') = ( \min\{z_0,z_0'\}, \dots, \min\{z_k,z_k'\}).$$
The \emph{distance} between $z$ and $z'$ is defined by
$$ \text{d}(z,z') = \max\{ |z - \gcd(z,z')|,|z'-\gcd(z,z')|\}.$$

Given $s \in S$ and $z,z' \in \varphi^{-1}(s)$, an \emph{$N$-chain} of factorizations from $z$ to $z'$ is a sequence $z_1,\dots,z_p \in \varphi^{-1}(s) $ such that $z_1 = z$, $z_p = z'$, and d$(z_i,z_{i+1}) \leq N$ for all $i \in \{1, \dots, p-1\}$. The \emph{catenary degree} of $s$, $\textup{c}(s)$, is the minimal $N \in \mathbb{N} \cup \{\infty\}$ such that for any two factorizations $z,z' \in \varphi^{-1}(s)$, there is an $N$-chain from $z$ to $z'$. The catenary degree of $S$ is defined by
$$ \textup{c}(S) = \sup \{ \textup{c}(s) \mid s\in S \}.$$

Two elements $z$ and $z'$ of $\mathbb{N}^{k+1}$ are $\mathcal{R}$-related if there exists a finite sequence $z = z_1, \dots, z_p = z'$ in $\mathbb{N}^{k+1}$ such that $\text{supp}(z_i) \cap \text{supp}(z_{i+1})$ is nonempty for all $i \in \{1, \dots, p-1\}$. We write $z \, \mathcal{R} \, z'$ in this case. We will refer to equivalence classes of $\varphi^{-1}(s)$ under the relation $\mathcal{R}$ as the $\mathcal{R}$-classes. 
We say that $s \in S$ is a \emph{Betti element} if there exist two elements of $\varphi^{-1}(s)$ that are not $\mathcal{R}$-related (i.e. it has at least two different $\mathcal{R}$-classes). The set of all Betti elements of $S$ is denoted $\text{Betti}(S)$.

\begin{thm}\label{theoremcatenarybetti}
    \textup{\cite[Theorem 15]{AssiDAnnaGarciaSanchez2020}} Let $S$ be a numerical semigroup. Then $$ \textup{c}(S) = \max \{\textup{c}(b) \mid b \in \textup{Betti}(S) \}.$$
\end{thm}

We define the \emph{kernel congruence} of $\varphi$ as $\sigma = \{ (a,b) \in \mathbb{N}^{k+1} \times \mathbb{N}^{k+1} \mid \varphi(a) = \varphi(b) \}$. It is well known that $S$ is isomorphic to the monoid $\mathbb{N}^{k+1} / \sigma$. 
A \emph{presentation} for $S$ is a subset $\rho$ of $\sigma$ such that $\sigma$ is the least congruence (with respect to set inclusion) containing $\rho$. A \emph{minimal presentation} is a presentation that is minimal with respect to set inclusion.
It turns out that the set of Betti elements is also helpful in finding minimal presentations.

We will describe a process to find all minimal presentations of $S$. Let $\mathcal{R}_1, \dots, \mathcal{R}_l$ be the different $\mathcal{R}$-classes of $\varphi^{-1}(b)$ for some $b \in \text{Betti}(S)$. We choose $v_i \in \mathcal{R}_i$ for each $i \in \{1,\dots,l\}$. Consider $\rho_b$ to be any set of $l-1$ pairs of elements in $\{v_1, \dots, v_l\}$ such that the graph with vertices $\{v_1, \dots, v_l\}$, having an edge between $v_i$ and $v_j$ if and only if $(v_i,v_j)$ or $(v_j,v_i)$ is in $\rho_b$, is connected. 
\begin{thm}\label{theorem_minimalpresentation}
    \textup{\cite{Rosales1996MinimalRelation}} The set $\rho = \bigcup_{b \in \text{Betti}(S)} \rho_b$ is a minimal presentation and every minimal presentation is of this form.
\end{thm}

\subsection{Finding Betti elements}

Let $S = \langle d_0,d_1,d_2,d_3 \rangle$ be a numerical semigroup with embedding dimension four. Let $a_{ij}, b_{ij}$, for $i,j \in \{0,1,2,3\}$, be the coefficients of the minimal and $d_0$-positive minimal relations, as in Section \ref{geometricprocedure}.

Suppose that $a_{ii} = b_{ii}$ for at least two distinct indices $i \in \{1,2,3\}$. This means that we can choose the coefficients $a_{ij}$ such that
\begin{equation}\label{catenary_assumption}
    \text{at most one of} \ \{a_{10},a_{20},a_{30}\} \ \text{is equal to 0}.
\end{equation}
Let $T$ be the $L$-shape from Proposition \ref{proptrick}.
Let $V$ be a minimal set of points $(x,y,z) \in \mathbb{Z}^3$ such that $xd_1 + yd_2 +zd_3$ is a nonnegative multiple of $d_0$ and after removing from the initial collection of cubes the regions associated to those points, we obtain $T$. 
Define
$$ U = \{ \max\{ x,0\} d_1 + \max\{y,0\}d_2 + \max\{z,0\}d_3 \mid (x,y,z) \in V \}.$$

\begin{figure}[h]
    \caption{For $\langle 103,133,165,228\rangle$ (see Fig$.$ \ref{imgnorweg}), $U$ is the set of labels of the cubes in red}
    \vspace{6pt}
    \centering
    \includegraphics[width = 0.35\textwidth]{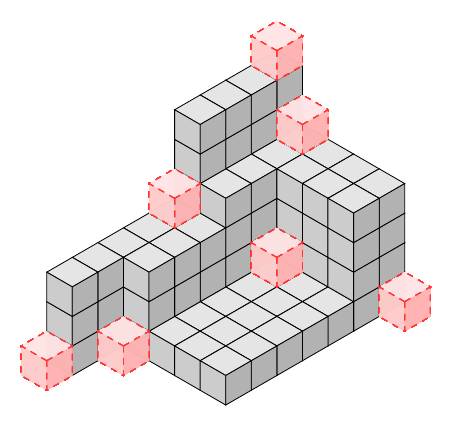}
    \label{imgbetti}
\end{figure}

\begin{lemma}\label{lemmabettifind}
   Assume that \eqref{catenary_assumption} holds. If $b \in \text{Betti}(S)$ labels a cube in $T$, then $b \in U$.
\end{lemma}
    
\begin{proof}
    Consider the factorization $z_0 =(0,\lambda_1,\lambda_2, \lambda_3) \in \varphi^{-1}(b)$, with $[[\lambda_1,\lambda_2,\lambda_3]] \in T$. Let $(\lambda_0',\lambda_1',\lambda_2', \lambda_3')$ be a different factorization of $b$. We get
    \begin{equation}\label{betticharacter_eq0}
        (\lambda_1-\lambda_1')d_1 +  (\lambda_2-\lambda_2')d_2 +  (\lambda_3-\lambda_3')d_3 = \lambda_0'd_0.
    \end{equation}
    
    If all coefficients on the left-hand side are nonnegative, we also clearly have $\lambda_0'>0$, as the two factorizations are distinct. This contradicts $[[\lambda_1,\lambda_2,\lambda_3]] \in T$ by Lemma \ref{1remlemma}. 

    Next, suppose that exactly one coefficient on the left-hand side of \eqref{betticharacter_eq0} is nonnegative, and, without loss of generality, assume that every coefficient below is nonnegative.
    \begin{equation}\label{betticharacter_eq00}
        (\lambda_1-\lambda_1')d_1   = \lambda_0'd_0 + (\lambda_2'-\lambda_2)d_2 +  (\lambda_3'-\lambda_3)d_3.
    \end{equation}
    This gives $a_{11} \leq (\lambda_1-\lambda_1') \leq \lambda_1$, which contradicts $[[\lambda_1,\lambda_2,\lambda_3]] \in T$.
    
    Hence, exactly two coefficients on the left-hand side of \eqref{betticharacter_eq0} are nonnegative, and, without loss of generality, assume that every coefficient below is nonnegative.
     \begin{equation}\label{betticharacter_eq01}
        (\lambda_1-\lambda_1')d_1 +  (\lambda_2-\lambda_2')d_2  = \lambda_0'd_0 +  (\lambda_3'-\lambda_3)d_3.
    \end{equation}
    If $\lambda_0'>0$, by Lemma \ref{1remlemma}, we get a contradiction with $[[\lambda_1,\lambda_2,\lambda_3]] \in T$, hence $\lambda_0'=0$. This means that for all $z \in \varphi^{-1}(b)$ we have $0 \notin \text{supp}(z)$, as $(\lambda_0',\lambda_1',\lambda_2', \lambda_3') \in \varphi^{-1}(b)$ was chosen arbitrarily. 

    Let $z=(0,\mu_1, \mu_2,  \mu_3) \in \varphi^{-1}(b)$ be a factorization that is not $\mathcal{R}$-related to $z_0$. Repeating the same consideration as for $(\lambda_0',\lambda_1',\lambda_2', \lambda_3')$, we can, without loss of generality, suppose that $\lambda_1 \geq \mu_1$, $\lambda_2 \geq \mu_2$, $\lambda_3 < \mu_3$. This implies $\mu_1, \mu_2, \lambda_3 =0$, as otherwise $z \, \mathcal{R} \, z_0$. Hence
    \begin{equation}\label{betticharacter_eq03}
       b= \lambda_1d_1 + \lambda_2d_2   =  \mu_3d_3,
    \end{equation}
    with $\lambda_1, \lambda_2, \mu_3> 0$ --- if $\lambda_1 =0$ or $\lambda_2 = 0$, we derive a contradiction as for \eqref{betticharacter_eq00}.
    If $ \mu_3 > a_{33}$, we could write $b  =  a_{30}d_0 + a_{31}d_1 + a_{32}d_2 +  (\mu_3-a_{33})d_3$, where $a_{30} = 0$ by what was proved before. Hence $\max \{a_{31},a_{32}\} > 0$, which implies $z_0 \, \mathcal{R} \, (0,a_{31},a_{32},\mu_3-a_{33}) \, \mathcal{R} \,z$, which is a contradiction. Thus $\mu_3 = a_{33}$ and $b = a_{33}d_3$, so $b \in U
    $.
\end{proof}

\begin{thm}\label{theorembettifind}
Assume that \eqref{catenary_assumption} holds. Then the whole set $\textup{Betti}(S)$ is contained in $U$. 
\end{thm}

\begin{proof}
    Suppose that there exists an element $b \in \text{Betti}(S)$ that is not in $U$.
    There must exist a factorization $z_0 =(0,\kappa_1,\kappa_2,\kappa_3) \in \varphi^{-1}(b)$; otherwise, if every factorization $z \in \varphi^{-1}(b)$ had $0 \in \text{supp}(z)$, then all factorizations would be $\mathcal{R}$-related, which cannot happen for a Betti element. By Lemma \ref{lemmabettifind}, we have $[[\kappa_1,\kappa_2,\kappa_3]] \notin T$, hence $[[\kappa_1,\kappa_2,\kappa_3]]$ is in the region associated to some point $P \in V$.

    \begin{itemize}

    \item  If $P =(-\lambda_1,\lambda_2,\lambda_3)$, where
    \begin{equation}\label{betticharacter_eq1}
        \lambda_0d_0 + \lambda_1d_1 = \lambda_2d_2 + \lambda_3d_3,
    \end{equation}
    and $\lambda_i \in \mathbb{Z}_{>0}$ for $i =0,1,2,3$, we get
     \begin{equation}\label{betticharacter_eq10}
        b =\lambda_0d_0 + (\lambda_1+\delta_1)d_1 + \delta_2d_2 + \delta_3d_3  =  \delta_1d_1 + (\lambda_2+\delta_2)d_2 + (\lambda_3+\delta_3)d_3,
    \end{equation}
    for some $\delta_i \in \mathbb{N}$ with $\max \{\delta_1,\delta_2,\delta_3\} > 0$. The factorizations $z_1 =(\lambda_0,\lambda_1+\delta_1,\delta_2,\delta_3)$ and $z_2 =(0, \delta_1,\lambda_2+\delta_2,\lambda_3+\delta_3)$ have a nonempty intersection of supports, and hence they are $\mathcal{R}$-related. Moreover, we have $0,1,2,3 \in \textup{supp}(z_1) \cup \textup{supp}(z_2)$, so every factorization of $b$ will be $\mathcal{R}$-related to $z_1$ and $z_2$. This contradicts $b \in \text{Betti}(S)$, as it must have at least two different $\mathcal{R}$-classes. 
    
    The cases $P=(\lambda_1,-\lambda_2,\lambda_3)$ and $P=(\lambda_1,\lambda_2,-\lambda_3)$ follow analogously.
    
    \item If $P =(a_{01},a_{02},a_{03})$, where
    \begin{equation}\label{betticharacter_eq2}
        a_{00}d_0  =  a_{01}d_1 + a_{02}d_2 + a_{03}d_3,
    \end{equation}
     and $a_{01},a_{02},a_{03} \in \mathbb{N}$, we get
     \begin{equation}\label{betticharacter_eq3}
        b =a_{00}d_0 + \delta_1d_1 + \delta_2d_2 + \delta_3d_3  =  (a_{01}+\delta_1)d_1 + (a_{02}+\delta_2)d_2 + (a_{03}+\delta_3)d_3,
    \end{equation}
    for some $\delta_i \in \mathbb{N}$ with $\max \{\delta_1,\delta_2,\delta_3\} > 0$. Without loss of generality, suppose that $\delta_1 > 0$. The factorizations $z_1 =(a_{00},\delta_1,\delta_2,\delta_3)$ and $z_2 =(0,a_{01}+\delta_1,a_{02}+\delta_2,a_{03}+\delta_3)$ have a nonempty intersection of supports, and hence they are $\mathcal{R}$-related. Because $b \in \text{Betti}(S)$, there exists a factorization $z \in \varphi^{-1}(b)$ which is not $\mathcal{R}$-related to them. As $a_{00},\delta_1 >0$, we have $0,1 \notin \text{supp}(z)$. 

     First, suppose that $z = (0,0,\mu_2,\mu_3)$ where $\mu_2,\mu_3>0$. This implies $z_1 = (a_{00},\delta_1,0,0 )$ and $z_2 =(0,a_{01}+\delta_1,0,0)$. 
    If $\mu_2 \geq a_{22}$, we can write
   \begin{equation}\label{betticharacter_eq40}
         b =\mu_2d_2 + \mu_3d_3 = a_{20}d_0 + a_{21}d_1 +(\mu_2-a_{22})d_2 + (\mu_3+a_{23})d_3.
   \end{equation}
   If $a_{20}>0$ or $a_{21}>0$, then $z_1 \, \mathcal{R} \,(a_{20},a_{21},\mu_2-a_{22},\mu_3+a_{23}) \, \mathcal{R} \, z$, which is a contradiction. Hence $a_{20},a_{21}=0$ and $a_{22}d_2 = a_{23}d_3$, which implies $a_{23} \geq a_{33}$. We obtain
   \begin{equation}\label{betticharacter_eq401}
        b=  a_{30}d_0 + a_{31}d_1 + (\mu_2-a_{22} + a_{32})d_2 + (\mu_3+a_{23}-a_{33})d_3,
   \end{equation}
   and, as in \eqref{betticharacter_eq40}, this yields $a_{30},a_{31}=0$. Since now $a_{20},a_{30} =0$, we get a contradiction with \eqref{catenary_assumption}.
   If $\mu_3 \geq a_{33}$ we get the same contradiction. 
   
   Thus $\mu_2 < a_{22}$ and $ \mu_3 < a_{33}$. 
   Consider the cube $[[0,\mu_2,\mu_3]]$. If it is in $T$, then $b \in U$ by Lemma \ref{lemmabettifind}. Hence $[[0,\mu_2,\mu_3]]$ is in the region associated to a point $Q \in V$.
   
   As $\mu_2 < a_{22}$ and $ \mu_3 < a_{33}$, we have $Q =(-\nu_1,\nu_2,\nu_3)$, where
    \begin{equation}\label{betticharacter_eq41}
        - \nu_1d_1 + \nu_2d_2 +  \nu_3d_3   =  \nu_0d_0,
    \end{equation}
    $\nu_0,\nu_2,\nu_3 \in \mathbb{Z}_{>0}$, $\nu_1 \in \mathbb{N}$ (possibly $\nu_1=0$) and $\mu_2 \geq \nu_2$, $\mu_3 \geq \nu_3$.  We can write 
    \begin{equation}\label{betticharacter_eq42}
        b =  \nu_0d_0 + \nu_1d_1 + (\mu_2-\nu_2)d_2 + (\mu_3-\nu_3)d_3.
    \end{equation}
    We have $\max \{\mu_2-\nu_2,\mu_3-\nu_3 \}>0$, as $b \notin U$, hence $z \, \mathcal{R} \, (\nu_0,\nu_1,\mu_2-\nu_2,\mu_3-\nu_3) \, \mathcal{R}  \, z_1$, which is a contradiction.

    Next, suppose that $z = (0,0,\mu_2,0)$. 
    If $\mu_2 = a_{22} $, we would have $b = a_{22}d_2 \in U$, so $\mu_2 > a_{22}$. We can write
    \begin{equation}\label{betticharacter_eq6}
         b=\mu_2d_2 = a_{20}d_0 + a_{21}d_1 + (\mu_2-a_{22})d_2 + a_{23}d_3.
    \end{equation}
    Because $(\mu_2-a_{22}) > 0$, we have $z \, \mathcal{R} \, (a_{20},a_{21},\mu_2-a_{22},a_{23})$. Since $z$ is not $\mathcal{R}$-related to $z_1$, we have $a_{20} = a_{21} = 0$ (recall $a_{00},\delta_1 >0$). This gives the factorization $(0,0,\mu_2-a_{22},a_{23}) \in \varphi^{-1}(b)$ that has $(\mu_2-a_{22}),a_{23} > 0$ and which is not $\mathcal{R}$-related to $z_1$ nor $z_2$, as it is $\mathcal{R}$-related to $z$. Thus, we can proceed as in the case $z =(0,0,\mu_2,\mu_3)$. 

    The case $z = (0,0,0,\mu_3)$ follows analogously.

    \item   If $P =(a_{11},-a_{12},-a_{13})$, where
    \begin{equation}
        a_{11}d_1  =  a_{10}d_0 + a_{12}d_2 + a_{13}d_3,
    \end{equation}
     and $a_{10},a_{12},a_{13} \in \mathbb{N}$, we get
    \begin{equation}\label{betticharacter_eq8}
        b = (a_{11}+\delta_1)d_1 +   \delta_2d_2 + \delta_3d_3  =  a_{10}d_0 + \delta_1d_1+ (a_{12}+\delta_2)d_2 + (a_{13}+\delta_3)d_3,
    \end{equation}
    for some $\delta_i \in \mathbb{N}$ with $\max \{\delta_1,\delta_2,\delta_3\} > 0$. The factorizations $z_1 =(0,a_{11} +\delta_1,\delta_2,\delta_3)$ and $z_2 =(a_{10},\delta_1,a_{12}+\delta_2,a_{13}+\delta_3)$ have a nonempty intersection of supports, and hence they are $\mathcal{R}$-related. Because $b \in \text{Betti}(S)$, there exists a factorization $z \in \varphi^{-1}(b)$ that is not $\mathcal{R}$-related to them. 
    
        If $a_{12} = a_{13} = 0$, we get $P = (a_{11},0,0)$, so we are in the case $P = (a_{01},a_{02},a_{03})$ and we can proceed analogously as there. 
        
        If $a_{12},a_{13}>0$, then $1,2,3 \in \textup{supp}(z_1) \cup \textup{supp}(z_2)$ and thus $z = (\mu_0,0,0,0)$. If $\mu_0 = a_{00}$, we would have $b = a_{00}d_0 \in U$, hence $\mu_0 > a_{00}$. We have $z   \, \mathcal{R} \, (\mu_0-a_{00}, a_{01}, a_{02}, a_{03}) \, \mathcal{R} \, z_1  \, \mathcal{R} \, z_2$, since $\max\{a_{01}, a_{02}, a_{03}\} > 0$, which gives a contradiction. 
        
        Thus, without loss of generality, we can assume $a_{12}>0, a_{13} = 0$. 
        
        If $a_{10}>0$, then $0,1,2 \in \textup{supp}(z_1) \cup \textup{supp}(z_2)$ and thus $z = (0,0,0,\mu_3)$. 
        If $\mu_3 = a_{33}$, we would have $b = a_{33}d_3 \in U$, hence $\mu_3 > a_{33}$. We have $z \, \mathcal{R} \, (a_{30},a_{31},a_{32},\mu_3-a_{33}) \, \mathcal{R} \, z_1  \, \mathcal{R} \, z_2$, since $\max\{a_{30},a_{31},a_{32}\}>0$, which is a contradiction.  
        
        Consequently,
        \begin{equation}\label{betticharacter_eq81}
            \text{for all} \ (a_{10},a_{12},a_{13}) \in \mathbb{N}^3 \ \text{satisfying} \ a_{11}d_1 = a_{10}d_0 + a_{12}d_2 + a_{13}d_3, \ \text{we have} \ a_{10}=0.
        \end{equation}
        We also have $a_{11}d_1 = a_{12}d_2$ (recall that we have assumed $a_{13} = 0$), which implies $a_{12} \geq a_{22}$. This gives
        \begin{equation}\label{betticharacter_eq9}
            a_{11}d_1 = a_{20}d_0 + a_{21}d_1 + (a_{12}-a_{22})d_2 + a_{23}d_3,
        \end{equation}
        which implies $a_{20}=0$ by \eqref{betticharacter_eq81} (we also have $a_{21} = 0$ in \eqref{betticharacter_eq9}, by the minimality of $a_{11}$). Since now $a_{10},a_{20} = 0$, we get a contradiction with \eqref{catenary_assumption}.

    The cases $P = (-a_{21},a_{22},-a_{23})$ and $P = (-a_{31},-a_{32},a_{33})$ follow analogously, and the claim follows.
\end{itemize}
\end{proof}

\begin{remark}\label{remarkcatenary}
    There exists a weaker version of Theorem \ref{theorembettifind} that works for all numerical semigroups with arbitrary embedding dimension.
    Let $b \in \textup{Betti}(\langle d_0, \dots, d_k \rangle)$. Then for any index $i \in \{0,\dots,k\}$, there exists an element $s_i \in \textup{Ap}(\langle d_0, \dots, d_k \rangle,d_i)$ and an index $j \neq i$, $ j\in \{0,\dots,k\}$, such that $b = s_i + d_j$. See \cite[Proposition 2.2]{Rosales1996MinimalRelation} or \cite[Proposition 66]{AssiDAnnaGarciaSanchez2020}.
\end{remark}

In order to determine the catenary degree, the minimal presentation, and the set of Betti elements inside $U$, we will need additional lemmas. The following lemma implies that the common value of every minimal relation is a Betti element.

\begin{lemma}\label{lemma_catenary2}
    For any $i \in \{0,1,2,3\}$ and $(\mu_0,\mu_1,\mu_2,\mu_3) \in \varphi^{-1}( a_{ii}d_i)$ we have $\mu_i = 0$ or $\mu_i = a_{ii}$ and $\mu_j = 0$ for all $j \in \{0,1,2,3\} \setminus \{i\}$.
\end{lemma}

\begin{proof}
     Without loss of generality, suppose $i=0$. We obtain
     \begin{equation}\label{lemma_catenary2_eq1}
         (a_{00}-\mu_0)d_0 =  \mu_1d_1 + \mu_2d_2 + \mu_3d_3.
     \end{equation}
     The right-hand side is clearly nonnegative. If $(a_{00}-\mu_0) = 0$, we also have $\mu_1 =\mu_2 =\mu_3 = 0$.  If $(a_{00}-\mu_0) > 0$, then the minimality of $a_{00}$ implies $\mu_0 = 0$.
\end{proof}

\begin{lemma}\label{lemma_catenary3}
    Assume that \eqref{catenary_assumption} holds.  If $(-\lambda_1,\lambda_2,\lambda_3) \in V$ satisfies
    $$ \lambda_0d_0 + \lambda_1d_1 = \lambda_2d_2 + \lambda_3d_3,$$
    with $\lambda_i \in \mathbb{Z}_{>0}$ and $a_{00}>\lambda_0$, then for any $(\mu_0,\mu_1,\mu_2,\mu_3) \in \varphi^{-1}(\lambda_0d_0 + \lambda_1d_1)$, we have $(\mu_0,\mu_1,\mu_2,\mu_3) = (0,0,\lambda_2,\lambda_3)$ or $\mu_2=\mu_3 = 0$. An analogous statement holds for $(\lambda_1,-\lambda_2,\lambda_3)$ and $(\lambda_1,\lambda_2,-\lambda_3) \in V$.
\end{lemma}   

\begin{proof}
Note that $a_{22} > \lambda_2$ and $a_{33} > \lambda_3$ from the definition of $V$. 

Suppose $\mu_0>0$. If $\mu_2 > 0$, then by Lemma \ref{1remlemma}, no element of the Ap{\'e}ry set is a label in the region associated to $(-\mu_1,\lambda_2-\mu_2,\lambda_3-\mu_3)$. This is a contradiction, as $[[0,\lambda_2-1, \lambda_3]] \in T$. If $\mu_3 > 0$, we get the same contradiction, hence $\mu_2=\mu_3 = 0$.

Thus, suppose $\mu_0 = 0$. If $\mu_2 \geq \lambda_2$, then we get
\begin{equation}\label{lemma_catenary3_eq1}
     (\lambda_3-\mu_3)d_3=\mu_1d_1 + (\mu_2 - \lambda_2)d_2,
\end{equation}
where all coefficients above are nonnegative. If $(\lambda_3-\mu_3) > 0$, we get a contradiction with $a_{33} > \lambda_3$, hence $(\lambda_3-\mu_3) = 0$, which gives $(\mu_0,\mu_1,\mu_2,\mu_3) = (0,0,\lambda_2,\lambda_3)$. If $\mu_3 \geq \lambda_3$, we get the same thing; therefore, suppose that $\mu_2 < \lambda_2$ and $\mu_3 < \lambda_3$. We have
\begin{equation}\label{lemma_catenary3_eq2}
   \lambda_0d_0 + (\lambda_1-\mu_1)d_1 = \mu_2d_2 + \mu_3d_3.
\end{equation}
If $(\lambda_1-\mu_1) \leq 0$, we get a contradiction with $a_{00} > \lambda_0$, hence $(\lambda_1-\mu_1) > 0$. Since $\lambda_0>0$, according to Lemma \ref{1remlemma}, no element of the Ap{\'e}ry set is a label in the region associated to $(\mu_1 - \lambda_1,\mu_2,\mu_3)$. This is a contradiction, as $[[0,\lambda_2-1,\lambda_3-1]] \in T$.

An analogous argument applies to $(\lambda_1,-\lambda_2,\lambda_3)$ and $(\lambda_1,\lambda_2,-\lambda_3) \in V$.
\end{proof}

\begin{lemma}\label{lemma_catenary1}
    Assume that \eqref{catenary_assumption} holds. If $(-\lambda_1,\lambda_2,\lambda_3) \in V$ satisfies
    $$ \lambda_0d_0 + \lambda_1d_1 = \lambda_2d_2 + \lambda_3d_3,$$
    with $\lambda_i \in \mathbb{Z}_{>0}$, $a_{00}>\lambda_0$, and $a_{11} > \lambda_1$, then $\varphi^{-1}(\lambda_0d_0 + \lambda_1d_1) = \{(\lambda_0,\lambda_1,0,0), (0,0,\lambda_2,\lambda_3) \} $. An analogous statement holds for $(\lambda_1,-\lambda_2,\lambda_3)$ and $(\lambda_1,\lambda_2,-\lambda_3) \in V$.
\end{lemma}
\begin{proof}
Note that $a_{22} > \lambda_2$ and $a_{33} > \lambda_3$ from the definition of $V$.
Suppose that there exists a third factorization $ (\mu_0,\mu_1,\mu_2,\mu_3) \in \varphi^{-1}(\lambda_0d_0 + \lambda_1d_1)$. 

By Lemma \ref{lemma_catenary3}, we have $\mu_2 = \mu_3 = 0 $. 
If $\lambda_0 \geq \mu_0$, we get
\begin{equation}\label{lemma_catenary1_eq1}
    (\lambda_0- \mu_0) d_0 =  (\mu_1 - \lambda_1)d_1,
\end{equation}
where all coefficients above are nonnegative. In fact, they are positive, as otherwise $(\mu_0,\mu_1,\mu_2,\mu_3) = (\lambda_0,\lambda_1,0,0)$. This contradicts $a_{00}>\lambda_0$. If $\lambda_0 < \mu_0$, we have $\lambda_1 > \mu_1$, and we get a contradiction with $a_{11}>\lambda_1$. 

An analogous argument applies to $(\lambda_1,-\lambda_2,\lambda_3)$ and $(\lambda_1,\lambda_2,-\lambda_3) \in V$.
\end{proof}

\begin{lemma}\label{lemma_catenarycase1}
    Assume that \eqref{catenary_assumption} holds. Consider $(-\lambda_1,\lambda_2,\lambda_3) \in V$ satisfying
    $$ \lambda_0d_0 + \lambda_1d_1 = \lambda_2d_2 + \lambda_3d_3,$$
    with $\lambda_i \in \mathbb{Z}_{>0}$, and $\lambda_0 \geq a_{00}$ or $\lambda_1 \geq a_{11}$. If $a_{00}d_0 \neq a_{11}d_1$, then $\lambda_0d_0 + \lambda_1d_1 \notin \textup{Betti}(S)$. Otherwise, if $a_{00}d_0 = a_{11}d_1$, then $\lambda_0d_0 + \lambda_1d_1 \in \textup{Betti}(S)$ and $\varphi^{-1}(\lambda_0d_0 + \lambda_1d_1) = \mathcal{F}$, where 
    $$\mathcal{F} =  \{(0,0,\lambda_2,\lambda_3)\}  \cup \{(\lambda_0 + \lambda a_{00},\lambda_1 - \lambda a_{11},0,0) \mid \lambda \in \mathbb{Z}, \lambda_0 + \lambda a_{00} \geq 0, \lambda_1 - \lambda a_{11} \geq 0 \} . $$   
    An analogous statement holds for $(\lambda_1,-\lambda_2,\lambda_3)$ and $(\lambda_1,\lambda_2,-\lambda_3) \in V$.

\end{lemma}

\begin{proof}
    Note that $a_{22} > \lambda_2$ and $a_{33} > \lambda_3$ from the definition of $V$.
    
    Suppose $a_{00}d_0 \neq a_{11}d_1$ and $\lambda_0 \geq a_{00}$. If $\max \{a_{02},a_{03}\} > 0$, then
    \begin{equation}
        (\lambda_0,\lambda_1,0,0) \, \mathcal{R} \, (\lambda_0 - a_{00},\lambda_1 + a_{01},a_{02},a_{03}) \, \mathcal{R} \, (0,0,\lambda_2,\lambda_3),
    \end{equation}
    which results in only one $\mathcal{R}$-class. If $a_{02}=a_{03}=0$, then $a_{00}d_0 = a_{01}d_1$ with $a_{01} > a_{11}$ (as $a_{00}d_0 \neq a_{11}d_1$). Also, we have $a_{10} = 0$, because otherwise we would get a contradiction with the minimality of $a_{00}$. This implies  $\max \{a_{12},a_{13}\} > 0$ and hence 
    \begin{equation}
        (\lambda_0,\lambda_1,0,0) \, \mathcal{R} \, (\lambda_0 - a_{00},\lambda_1 + a_{01},0,0) \, \mathcal{R} \, (\lambda_0 - a_{00},\lambda_1 + a_{01}-a_{11},a_{12},a_{13}) \, \mathcal{R} \, (0,0,\lambda_2,\lambda_3),
    \end{equation}
    which results in only one $\mathcal{R}$-class. If $a_{00}d_0 \neq a_{11}d_1$ and $\lambda_0 < a_{00}$, then $\lambda_1 \geq a_{11}$ and we can apply the same argument as above. Thus, the first result follows.

    Suppose $a_{00}d_0 = a_{11}d_1$. Clearly $\mathcal{F} \subseteq \varphi^{-1}(\lambda_0d_0 + \lambda_1d_1)$. Suppose that there exists $(\mu_0,\mu_1,\mu_2,\mu_3) \in \varphi^{-1}(\lambda_0d_0 + \lambda_1d_1)$ with $(\mu_0,\mu_1,\mu_2,\mu_3) \notin \mathcal{F}$. As $(\mu_0,\mu_1,\mu_2,\mu_3) \neq (0,0,\lambda_2,\lambda_3)$, we have $\mu_2=\mu_3 = 0$, by Lemma \ref{lemma_catenary3} (we can apply this lemma, because in $\mathcal{F}$, there exists a factorization that satisfies its assumptions). Then, as $(\mu_0,\mu_1,0,0) \notin \mathcal{F}$, we can find $\mu \in \mathbb{Z}$ such that $\lambda_0 + a_{00} > \mu_0 + \mu a_{00} > \lambda_0$. This gives
    \begin{equation}
        \lambda_0d_0 + \lambda_1d_1 = (\mu_0 + \mu a_{00})d_0 + (\mu_1 - \mu a_{11})d_1 \implies (\mu_0 + \mu a_{00} - \lambda_0)d_0 = (\lambda_1 - \mu_1 + \mu a_{11})d_1,  
    \end{equation}
    which contradicts $ a_{00} > (\mu_0 + \mu a_{00} - \lambda_0) > 0$.

    An analogous argument applies to $(\lambda_1,-\lambda_2,\lambda_3)$ and $(\lambda_1,\lambda_2,-\lambda_3) \in V$.
\end{proof}

Lemmas \ref{lemma_catenary1} and \ref{lemma_catenarycase1} determine which elements of $U$ that arise from equations of the form $\lambda_0d_0 + \lambda_id_i = \lambda_jd_j + \lambda_k d_k$ are Betti elements. Moreover, they provide their corresponding sets of factorizations. In combination with Lemma \ref{lemma_catenary2}, we obtain a precise description of which elements of $U$ belong to $\textup{Betti}(S)$. However, Lemma \ref{lemma_catenary2} does not determine the sets $\varphi^{-1}(a_{ii}d_i)$. As we will see in Sections \ref{findingtheFnumber} and \ref{triangularnumbers}, this will not be an obstacle when determining the catenary degree and the minimal presentation.

\begin{remark}\label{remark_a00_procedure}
    The procedure from Section \ref{geometricprocedure} retrieves $a_{11},a_{22},a_{33}$ via Proposition \ref{propminimalrelations}, but does not immediately yield $a_{00}$. By Lemma \ref{lemma_catenary2}, $a_{00}d_0 \in \textup{Betti}(S)$, and since $\textup{Betti}(S) \subseteq U$ by Theorem \ref{theorembettifind}, we have $a_{00}d_0 \in U$. Thus, $a_{00}$ can be identified as the smallest $\lambda \in \mathbb{Z}_{>0}$ such that $\lambda d_0 \in U$. To verify membership in $U$: for elements of the form $a_{ii}d_i$, $i \in \{1,2,3\}$, one checks whether they are divisible by $d_0$; for elements that arise from equations of the form $\lambda_0d_0 + \lambda_id_i = \lambda_jd_j + \lambda_k d_k$, their sets of factorizations are determined by Lemmas \ref{lemma_catenary1} and \ref{lemma_catenarycase1}, and can be inspected directly; and all other elements of $U$ are just equal to $a_{00}d_0$ by Theorem \ref{theoremprocedure}.
\end{remark}

\section{Four consecutive squares}\label{findingtheFnumber}

In this section, we will find the Frobenius number, the genus, the set of Betti elements, the catenary degree, and a minimal presentation of the numerical semigroup generated by four consecutive squares; $S = \langle n^2,(n+1)^2,(n+2)^2,(n+3)^2 \rangle $.
We will use the methods developed in Sections \ref{geometricprocedure} and \ref{bettielements}, where we put $(d_0,d_1,d_2,d_3)=(n^2,(n+1)^2,(n+2)^2,(n+3)^2)$.

We will call $T$ the figure that will be obtained by removing regions from the initial collection of cubes in each case of $n \pmod{12}$. The cubes of $T$ will be labeled (as always), where the cube $[[a,b,c]]$ will be labeled with $ad_1 + bd_2 + cd_3$. We will know that the set of labels of $T$ is the Ap{\'e}ry set with respect to $d_0 = n^2$ after we calculate that $T$ has exactly $d_0 = n^2$ cubes.
Since we have the formula $F(S) = \max\{\text{Ap}(S,a) \}- a$, in order to compute the Frobenius number, we need to find the largest label in $T$. In addition, the formula $G(S) = \frac{1}{a} \left( \sum_{s \in \text{Ap}(S,a)} s \right) - \frac{a-1}{2}$ reduces the search for the genus to finding the sum of labels in $T$.

We sum up the results of this section in Corollaries \ref{corfrobsquares}-\ref{corbettisquares}.

\subsection{$n=12k, \ k>1$}

\subsubsection{Ap{\'e}ry set}
We have the following:
\begin{align}
    (7k+2)(12k)^2 &= 2(12k+1)^2 + (3k-5)(12k+2)^2 + (4k+2)(12k+3)^2, \label{12kminrelation0} \\
    (6k+4)(12k+1)^2 &= 1(12k)^2 + 1(12k+2)^2 + 6k(12k+3)^2, \label{12kminrelation1} \\
    9k(12k+2)^2 &= (5k+1)(12k)^2 + 0(12k+1)^2 + 4k(12k+3)^2, \label{12kminrelation2}\\
    (6k+1)(12k+3)^2 &= 0(12k)^2 + (6k+1)(12k+1)^2 + 2(12k+2)^2,  \label{12kminrelation3}
\end{align}
By Lemmas \ref{1remlemma} and \ref{lemmaminrelation1}, constructing points from the above equations and deleting their associated regions, the Ap{\'e}ry set is contained in the set of labels of the remaining figure.
Furthermore, we have the following relations that satisfy the conditions of Theorem \ref{theoremprocedure}:
\begin{align}
    (7k+1)(12k)^2 + (12k+1)^2 &= (3k-2)(12k+2)^2 + (4k+1)(12k+3)^2,\label{12k_eq1}\\
    (12k)^2 + 3(12k+2)^2 &= 3(12k+1)^2 + (12k+3)^2,\label{12k_eq2}\\
    (7k+2)(12k)^2 + (2k-1)(12k+3)^2 &= (6k+3)(12k+1)^2 + (3k-3)(12k+2)^2. \label{12k_eq3}
\end{align}

Subtracting equation \eqref{12k_eq2} from \eqref{12k_eq1}, we obtain additional $2k$ relations:
\begin{equation}\label{12k_2kequations_1}
\begin{aligned}    
(7k)(12k)^2 + 4(12k+1)^2 &= (3k+1)(12k+2)^2 + (4k)(12k+3)^2 ,\\
(7k-1)(12k)^2 + 7(12k+1)^2 &= (3k+4)(12k+2)^2 + (4k-1)(12k+3)^2 ,\\
 & \hspace{5.6pt}  \vdots      \\
(5k+2)(12k)^2 + (6k-2)(12k+1)^2 &= (9k-5)(12k+2)^2 + (2k+2)(12k+3)^2 ,\\
(5k+1)(12k)^2 + (6k+1)(12k+1)^2 &= (9k-2)(12k+2)^2 + (2k+1)(12k+3)^2 .
\end{aligned}
\end{equation}
We could subtract equation \eqref{12k_eq2} again from the above, and the resulting point will delete the region $\{y > 9k+1, z > 2k\}$. This region is already deleted by \eqref{12kminrelation2} (the region $\{y > 9k\}$ is already gone). 

Similarly, we can subtract equation \eqref{12k_eq2} from \eqref{12k_eq3} and obtain further $2k$ relations:
\begin{equation}\label{12k_2kequations_2}
\begin{aligned} 
(7k+1)(12k)^2 + (2k)(12k+3)^2 &= (6k)(12k+1)^2 + (3k)(12k+2)^2,\\
(7k)(12k)^2 + (2k+1)(12k+3)^2 &= (6k-3)(12k+1)^2 + (3k+3)(12k+2)^2,\\
(7k-1)(12k)^2 + (2k+2)(12k+3)^2 &= (6k-6)(12k+1)^2 + (3k+6)(12k+2)^2,\\
& \hspace{5.6pt} \vdots\\
(5k+3)(12k)^2 + (4k-2)(12k+3)^2 &= 6(12k+1)^2 + (9k-6)(12k+2)^2, \\
(5k+2)(12k)^2 + (4k-1)(12k+3)^2 &= 3(12k+1)^2 + (9k-3)(12k+2)^2. \\
\end{aligned}
\end{equation}
We do not need to subtract anymore because the resulting points will have $y \geq 9k$, which is already taken care of by \eqref{12kminrelation2}.

Now we will calculate the volume of the figure $T$, which is carved out using all the above equations. Firstly, consider the figure carved out only by \eqref{12kminrelation1}, \eqref{12kminrelation2}, \eqref{12kminrelation3} and \eqref{12k_eq2}. It will have the volume
\begin{equation*}\label{12kvol1}
    (6k+4) \cdot 9k   +  3 \cdot 9k \cdot 6k  =  216 k^2 + 36 k,
\end{equation*}
where the first summand is the volume of the $(6k+4) \times 9k \times 1$ bottom layer and the second summand is the rest, which will be a $3 \times 9k \times 6k$ box. Now, we will see how many cubes are deleted from the aforementioned figure, by equation \eqref{12k_eq3} and equations \eqref{12k_2kequations_2}. They will remove a section of stairs from the bottom layer (see Figure \ref{img1}). It has the volume
\begin{equation*}\label{12kvol2}
     \frac{(6k+3)(6k+2)}{2} = 18 k^2 + 15 k + 3.
\end{equation*}
Lastly, equations \eqref{12kminrelation0}, \eqref{12k_eq1}, and \eqref{12k_2kequations_1} will further remove the volume
\begin{equation*}\label{12kvol3}
     3 \cdot 2 \cdot 4k + 3  \cdot \frac{6k(6k-1)}{2} + 3 \cdot (2k-1) = 54 k^2 + 21 k - 3,
\end{equation*}
where the first summand is the $3 \times 2 \times 4k$ strip, removed by the region associated to $(-(6k+1),9k-2,2k+1)$, the second summand is the volume removed by the stairs with step $3 \times 1$ produced by \eqref{12k_eq1} and the rest of \eqref{12k_2kequations_1}, and the final part is lastly removed by \eqref{12kminrelation0}. The volume of $T$ comes down to
\begin{equation*}\label{12kvol4}
     216 k^2 + 36 k - (18 k^2 + 15 k + 3) - ( 54 k^2 + 21 k - 3) = 144k^2.
\end{equation*}
Because the volume of $T$ is exactly $(12k)^2$, it is an $L$-shape.

\begin{figure}[h]
    \caption{the figure $T$ for $n=12k$ and $k=3$}
    \centering
    \includegraphics[width = 0.5\textwidth]{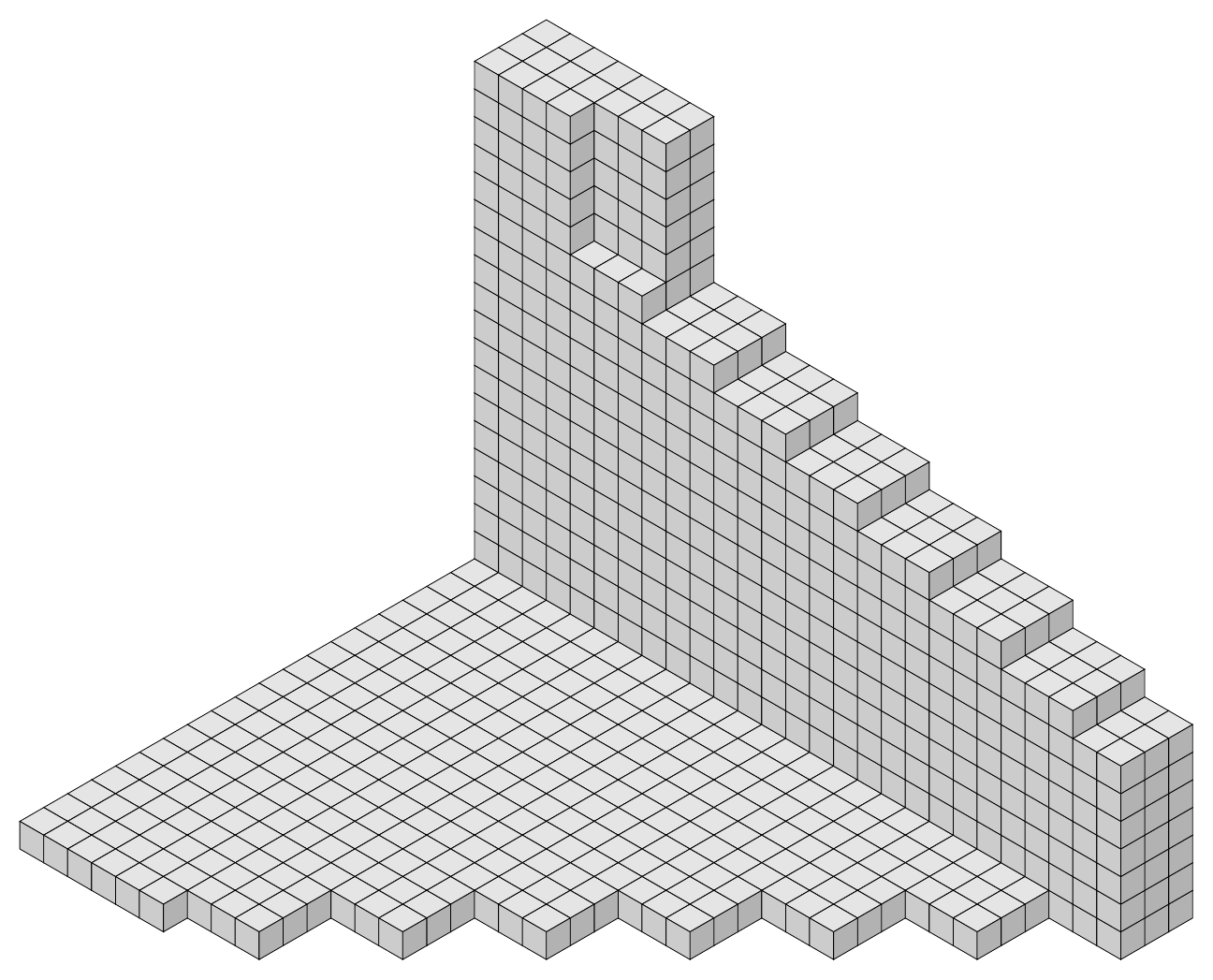}
    \label{img1}
\end{figure}

\subsubsection{Frobenius number}

$F(S)+(12k)^2$ is equal to the largest label in $T$. It clearly labels one of the corners of $T$ --- the cubes that have exactly three of their faces visible, when looking at the figure from the first octant.

First, the corners of the bottom stairs of $T$, from left to right, have labels:
\begin{align*}
(6k+3)(12k+1)^2 +(3k-4)(12k+2)^2 &= 1296 k^3 + 144 k^2 - 102 k - 13,\\
(6k+2)(12k+1)^2 +(3k-1)(12k+2)^2 &= 1296 k^3 + 432 k^2 + 18 k - 2,\\
(6k-1)(12k+1)^2 +(3k+2)(12k+2)^2 &= 1296 k^3 + 432 k^2 + 90 k + 7,\\
& \hspace{5.8pt} \vdots \\
8(12k+1)^2 +(9k-7)(12k+2)^2 &=1296 k^3 + 576 k^2 - 108 k - 20,\\
5(12k+1)^2 +(9k-4)(12k+2)^2 &= 1296 k^3 + 576 k^2 - 36 k - 11.
\end{align*}
The labels of the corners of the right stairs of $T$, listed from bottom to top,
are:
\begin{align*}
2(12k+1)^2 +(9k-1)(12k+2)^2 +(2k) (12k+3)^2 &=1584 k^3 + 720 k^2 + 54 k - 2,\\
2(12k+1)^2 +(9k-3)(12k+2)^2 +(2k+1) (12k+3)^2 &=1584 k^3 + 576 k^2 + 30 k - 1,\\
2(12k+1)^2 +(9k-6)(12k+2)^2 +(2k+2) (12k+3)^2 &=1584 k^3 + 288 k^2 - 42 k - 4, \\
& \hspace{5.8pt} \vdots \\
2(12k+1)^2 +(3k)(12k+2)^2 +(4k) (12k+3)^2 &=1008 k^3 + 720 k^2 + 96 k + 2,\\
2(12k+1)^2 +(3k-3)(12k+2)^2 +(4k+1) (12k+3)^2 &=1008 k^3 + 432 k^2 + 24 k - 1.
\end{align*}
Finally, the last two corners of the top layer of $T$ have labels:
\begin{align*}
2(12k+1)^2 +(3k-6)(12k+2)^2 +(6k) (12k+3)^2 &= 1296 k^3 - 174 k - 22,\\
(12k+1)^2 + (3k-3)(12k+2)^2 +(6k) (12k+3)^2 &=1296 k^3 + 288 k^2 - 54 k - 11.
\end{align*}
The value $1584 k^3 + 720 k^2 + 54 k - 2$ is the largest from above and is equal to $F(S)+(12k)^2$. This gives
$$ F(S) = 1584 k^3 + 576 k^2 + 54 k - 2.$$

\subsubsection{Genus}

Next, in order to find $G(S)$, we have to find the sum of the labels in $T$. To accomplish that, we will use the previously found corners.

The whole calculation is tedious, so we will split it into smaller parts. Specifically, as we will be adding numbers from $\langle (12k+1)^2,(12k+2)^2,(12k+3)^2 \rangle$ we can separately find how many times we will be adding each of the numbers $(12k+1)^2,(12k+2)^2,(12k+3)^2$ --- i.e. we will separately find the sum of $x$-, $y$- and $z$-coordinates of the cubes of $T$. Furthermore, we will look at $T$ in sections similarly as we did when computing its volume. 

First, we can use the corners from the bottom stairs of $T$, to find for each of them the sum of $x$-coordinates of cubes from the region $\{ a-2 \leq x \leq a +1, y \leq b+1 , 0 \leq z \leq 1  \}$, where $(a,b,0)$ are the coordinates of a corner cube from the bottom layer (with the exception of the leftmost corner). In this way, we get manageable sections, which we can calculate using the already found coordinates of the corners. Adding the sum of $x$-coordinates of cubes from the region $\{ 6k+3 \leq x \leq 6k+4, y \leq 3k-3 , 0 \leq z \leq 1\}$ we find that the sum of $x$-coordinates of cubes from the bottom layer of $T$, which have the upper face visible in Figure \ref{img1}, is equal to \footnotesize 
\begin{align*}
&(5+4+3)(9k-3) + (8+7+6)(9k-6) + \dots + ((6k+2) + (6k+1) + 6k)3k +  (6k+3)(3k-3) \\
&=  (6k+3)(3k-3) + \sum_{i=0}^{2k-1} (12+9i)(9k-3-3i)\\
&=(6k+3)(3k-3) + \sum_{i=0}^{2k-1} 108k-36 + i(81k-63) - 27i^2 \\
& = (6k+3)(3k-3) + 2k(108k-36) + (81k-63) \frac{(2k-1)2k}{2} -27 \frac{(2k-1)2k(4k-1)}{6} \\
& = 90 k^3 + 81 k^2 - 27 k - 9.
\end{align*} \small
Similarly, the sum of the $y$-coordinates of cubes from the bottom layer of $T$, which have the top face visible in Figure  \ref{img1}, is equal to \footnotesize 
\begin{align*}
& 3 \left( \frac{(9k-4)(9k-3)}{2} + \frac{(9k-7)(9k-6)}{2} + \dots + \frac{(3k-1)3k}{2} \right) + \frac{(3k-4)(3k-3)}{2}\\
& = \frac{(3k-4)(3k-3)}{2} + \frac{3}{2} \sum_{i=0}^{2k-1} (9k-4 -3i)(9k-3-3i) \\
& =  \frac{(3k-4)(3k-3)}{2} + \frac{3}{2} \sum_{i=0}^{2k-1} 81k^2-63k + 12 + i(21-54k) + 9i^2 \\
&=  \frac{(3k-4)(3k-3)}{2} + 3k( 81k^2-63k + 12) + \frac{3}{2} (21-54k) \frac{(2k-1)2k}{2} + \frac{27}{2}   \frac{(2k-1)2k(4k-1)}{6} \\
& = 117 k^3 - \frac{135 k^2}{2} - \frac{3 k}{2} + 6. 
\end{align*} \small
We can apply the same approach to the other parts of $T$. It easily follows that the sum of the $x$-coordinates of all the other cubes in $T$ is
$$ 36 k^2 + 4 k + 2(36 k^2 - 2 k + 3) = 108 k^2 + 6.$$
Considering, for each corner of the right stairs with coordinates $(2,b,c)$, the region $\{0 \leq x \leq 2, b-2 \leq y \leq b+1, z \leq c+1\}$ (with the exception of the rightmost corner), it follows that the sum of $y$-coordinates of all cubes of $T$ with $x \leq 2$ is \footnotesize 
\begin{align*}
&(6k+3)\frac{(9k-1)9k}{2}   + 3 \left( \frac{(9k-3)(9k-2)}{2}  +  \frac{(9k-6)(9k-5)}{2}  + \dots +   \frac{(3k-3)(3k-2)}{2} \right) \\
& + (6k-3) \frac{(3k-6)(3k-5)}{2}   + (4k-2)( (3k-3) +(3k-4) + (3k-5))\\
& = 270 k^3 + 18 k^2 + 60 k - 21 + 3\sum_{i=0}^{2k} \frac{(9k-3-3i)(9k-2-3i)}{2} \\
& = 270 k^3 + 18 k^2 + 60 k - 21 + \frac{3}{2}\sum_{i=0}^{2k} 81k^2 - 45k+6 + i(15-54k) + 9i^2  \\
&= 270 k^3 + 18 k^2 + 60 k - 21 
+ \frac{3}{2}\left(  (2k+1)(81k^2 - 45k+6) + (15-54k)\frac{2k(2k+1)}{2} + 9 \frac{2k(2k+1)(4k+1)}{6} \right) \\
& = 387 k^3 - \frac{9 k^2}{2} + \frac{75 k}{2} - 12.
\end{align*} \small
Repeating the same approach, we see that the sum of $z$-coordinates of all the cubes of $T$ with $x \leq 2$ is \footnotesize
\begin{align*}
& 6\frac{2k(2k+1)}{2} + 9 \left (\frac{(2k+1)(2k+2)}{2} + \frac{(2k+2)(2k+3)}{2} + \dots +  \frac{(4k+1)(4k+2)}{2} \right)\\
&  + 3(3k-5) \frac{6k(6k+1)}{2}  + 6 \left( \frac{6k(6k+1)}{2} - \frac{(4k+1)(4k+2)}{2}  \right)\\
& = 162 k^3 - 171 k^2 - 57 k - 6 + \frac{9}{2} \sum_{i=0}^{2k} (2k+1+i)(2k+2+i)  \\
& = 162 k^3 - 171 k^2 - 57 k - 6 + \frac{9}{2} \sum_{i=0}^{2k} 4k^2+6k+2 + i(4k+3) + i^2  \\
& = 162 k^3 - 171 k^2 - 57 k - 6 + \frac{9}{2} \left( ( 4k^2+6k+2)(2k+1) + (4k+3) \frac{2k(2k+1)}{2} +  \frac{2k(2k+1)(4k+1)}{6} \right)\\
& = 246k^3 - 45k^2 + 3k+3.
\end{align*} \small
In total, the sum of labels in $T$ comes down to
\begin{align*}
    &(90 k^3 + 189 k^2 - 27 k - 3) (12k+1)^2 + (504 k^3 - 72 k^2 + 36 k - 6)(12k+2)^2 \\
    &+ (246k^3 - 45k^2 + 3k+3) (12k+3)^2  = 120960 k^5 + 54432 k^4 + 3888 k^3 - 72 k^2.
\end{align*}
This gives
$$ G(S) = \frac{120960 k^5 + 54432 k^4 + 3888 k^3 - 72 k^2}{ (12k)^2}  - \frac{(12k)^2-1}{2} = 840 k^3 + 306 k^2 + 27 k .$$

\subsubsection{Catenary degree}
Using Proposition \ref{propminimalrelations} and Remark \ref{remark_a00_procedure}, we find that equations \eqref{12kminrelation0}-\eqref{12kminrelation3} are in fact the minimal relations.

The common value of equation \eqref{12k_eq3} is not a Betti element, by Lemma \ref{lemma_catenarycase1}. 
By Lemmas \ref{lemma_catenary2} and \ref{lemma_catenary1}, the common values of each of equations \eqref{12kminrelation0}-\eqref{12k_eq2}, \eqref{12k_2kequations_1}, \eqref{12k_2kequations_2} are Betti elements. By Theorem \ref{theorembettifind}, they are all the Betti elements, and there are $4k+6$ of them. 

Moreover, Lemma \ref{lemma_catenary1} gives us all the factorizations of the common values of equations \eqref{12k_eq1}, \eqref{12k_eq2}, \eqref{12k_2kequations_1}, \eqref{12k_2kequations_2}. Thus, one can check that the largest catenary degree among them is $\textup{c}( (5k+1)(12k)^2 + (6k+1)(12k+1)^2) = 11k+2$. Furthermore, we have
$$   (6k+4)(12k+1)^2, \ 9k(12k+2)^2, \ (6k+1)(12k+3)^2  < 11k(12k)^2, $$
hence $\textup{c}( (7k+2)(12k)^2), \textup{c}((6k+4)(12k+1)^2), \textup{c}(9k(12k+2)^2),\textup{c}((6k+1)(12k+3)^2) < 11k+2$, which gives
$$ \textup{c}(S) = 11k+2 .$$

Lastly, to obtain a minimal presentation, we apply Theorem \ref{theorem_minimalpresentation}. We know the sets of factorizations of all the Betti elements (hence, we can construct the corresponding sets $\rho_b$), except for the common values of the minimal relations. However, we may take $\rho_{(7k+2)(12k)^2} = ( (7k+2,0,0,0) ,(0,2,3k-5,4k+2))$ and $\rho_{(6k+4)(12k+1)^2} = ( (0,6k+4,0,0) ,(1,0,1,6k))$ whose corresponding graphs are clearly connected. Moreover, we can see that $9k(12k+2)^2 \notin \langle (12k+1)^2 \rangle$ and $(6k+1)(12k+3)^2 \notin \langle (12k)^2 \rangle $, hence we may take $\rho_{9k(12k+2)^2} = ( (0,0,9k,0) ,(5k+1,0,0,4k))$ and $\rho_{(6k+1)(12k+3)^2} = ( (0,0,0,6k+1) ,(0,6k+1,2,0))$ whose corresponding graphs are also connected. Hence, we can obtain a minimal presentation, whose cardinality is $4k+6$.

\subsection{$n = 12k + 4, \  k >0$}
\subsubsection{Ap{\'e}ry set}
We have the following:
\begin{align}
    (7k+4)(12k+4)^2 &= 1(12k+5)^2 + (3k-3)(12k+6)^2 + (4k+3)(12k+7)^2, \label{12k+4minrelation0} \\
    (6k+6)(12k+5)^2 &= 1(12k+4)^2 + 1(12k+6)^2 + (6k+2)(12k+7)^2, \label{12k+4minrelation1} \\
    (9k+5)(12k+6)^2 &= (5k+2)(12k+4)^2 + 2(12k+5)^2 + (4k+2)(12k+7)^2, \label{12k+4minrelation2}\\
    (6k+3)(12k+7)^2 &= 0(12k+4)^2 + (6k+3)(12k+5)^2 + 2(12k+6)^2,  \label{12k+4minrelation3}
\end{align}
and by Lemmas \ref{1remlemma} and \ref{lemmaminrelation1}, constructing points from the above equations and deleting their associated regions, the Ap{\'e}ry set is contained in the set of labels of the remaining figure.
Furthermore, we have the following relations that satisfy the conditions of Theorem \ref{theoremprocedure}:
\begin{align}
    (7k+3)(12k+4)^2 + 2(12k+5)^2 &= (3k)(12k+6)^2 + (4k+2)(12k+7)^2 , \label{12k+4_eq1} \\
    (12k+4)^2 + 3(12k+6)^2 &= 3(12k+5)^2 + (12k+7)^2,\label{12k+4_eq2} \\
    (7k+4)(12k+4)^2 + (2k)(12k+7)^2 &= (6k+4)(12k+5)^2 + (3k-1)(12k+6)^2 .\label{12k+4_eq3}
\end{align}
Subtracting equation \eqref{12k+4_eq2} from \eqref{12k+4_eq1}, we obtain additional $2k+1$ relations:
\begin{equation}\label{12k+4_2k+1equations_1}
\begin{aligned}    
(7k+2)(12k+4)^2 + 5(12k+5)^2 &= (3k+3)(12k+6)^2 + (4k+1)(12k+7)^2,\\
(7k+1)(12k+4)^2 + 8(12k+5)^2 &= (3k+6)(12k+6)^2 + (4k)(12k+7)^2,\\
 & \hspace{5.6pt}  \vdots      \\
(5k+3)(12k+4)^2 + (6k+2)(12k+5)^2 &= (9k)(12k+6)^2 + (2k+2)(12k+7)^2, \\
(5k+2)(12k+4)^2 + (6k+5)(12k+5)^2 &= (9k+3)(12k+6)^2 + (2k+1)(12k+7)^2.
\end{aligned}
\end{equation}
Similarly, we can subtract equation \eqref{12k+4_eq2} from \eqref{12k+4_eq3} and obtain further $2k+1$ relations:
\begin{equation}\label{12k+4_2k+1equations_2}
\begin{aligned}    
(7k+3)(12k+4)^2 + (2k+1)(12k+7)^2 &= (6k+1)(12k+5)^2 + (3k+2)(12k+6)^2,\\
(7k+2)(12k+4)^2 + (2k+2)(12k+7)^2 &= (6k-2)(12k+5)^2 + (3k+5)(12k+6)^2,\\
 & \hspace{5.6pt}  \vdots      \\
(5k+4)(12k+4)^2 + (4k)(12k+7)^2 &= 4(12k+5)^2 + (9k-1)(12k+6)^2, \\
(5k+3)(12k+4)^2 + (4k+1)(12k+7)^2 &= (12k+5)^2 + (9k+2)(12k+6)^2.
\end{aligned}
\end{equation}

Analogously as before, one can check that the figure $T$, carved out using all the above equations, has the volume $(12k+4)^2$, which implies that it is an $L$-shape.

\begin{figure}[h]\label{figure12k+4}
    \caption{the figure $T$ for $n=12k+4$ and $k=2$}
    \centering
     \includegraphics[width = 0.5\textwidth]{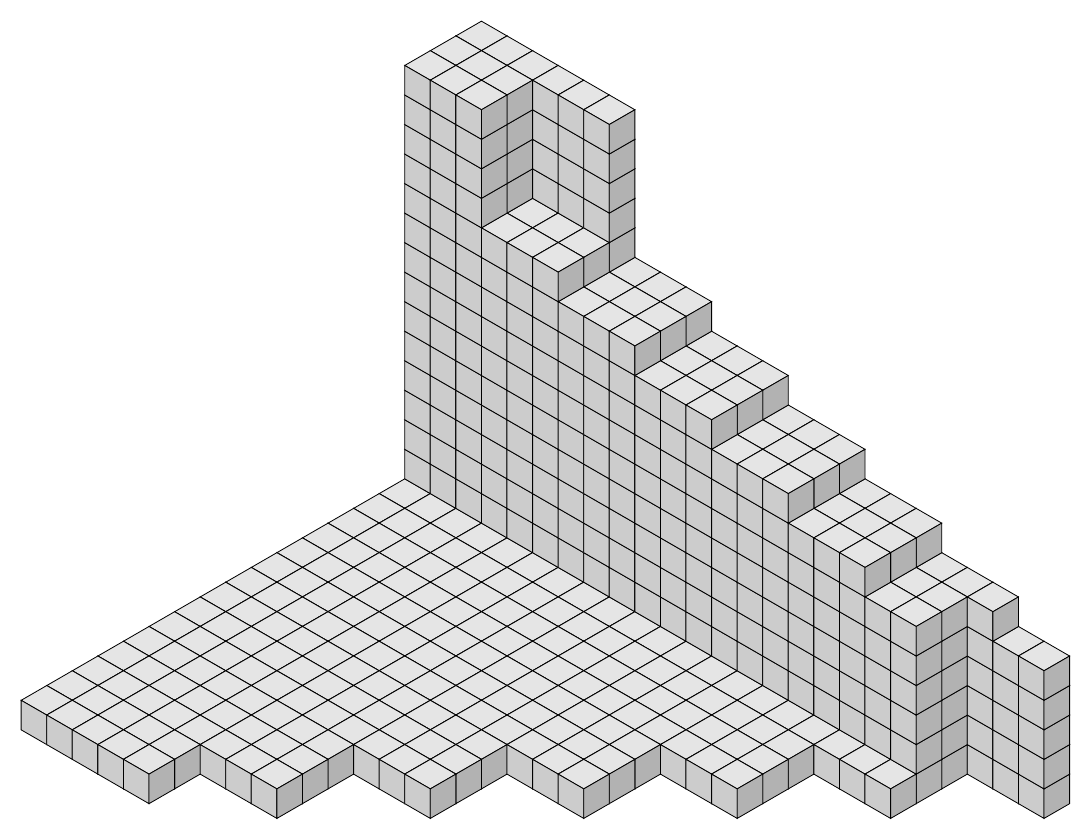}
    \label{img2}
\end{figure}

\subsubsection{Frobenius number and genus}

The largest label in $T$ is
$$(9k+4)(12k+6)^2 + (2k)(12k+7)^2 =1584 k^3 + 2208 k^2 + 998 k + 144,$$ 
which gives
$$ F(S) = 1584 k^3 + 2064 k^2 + 902 k + 128.$$
One can find that the sum of the labels of $T$ is 
$$ 120960 k^5 + 256032 k^4 + 210576 k^3 + 84552 k^2 + 16640 k + 1288,$$
which gives
$$ G(S)  = 840 k^3 + 1146 k^2 + 509 k + 73.$$
\subsubsection{Catenary degree}
Using Proposition \ref{propminimalrelations} and Remark \ref{remark_a00_procedure}, we find that equations \eqref{12k+4minrelation0}-\eqref{12k+4minrelation3} are in fact the minimal relations.

The common value of equation \eqref{12k+4_eq3} is not a Betti element, by Lemma \ref{lemma_catenarycase1}. 
By Lemmas \ref{lemma_catenary2} and \ref{lemma_catenary1}, the common values of each of equations \eqref{12k+4minrelation0}-\eqref{12k+4_eq2}, \eqref{12k+4_2k+1equations_1}, \eqref{12k+4_2k+1equations_2} are Betti elements. 
By Theorem \ref{theorembettifind}, they are all the Betti elements, and there are $4k+8$ of them.

Similarly as before, using Lemma \ref{lemma_catenary1}, one can check that the largest catenary degree among the common values of equations \eqref{12k+4_eq1}, \eqref{12k+4_eq2}, \eqref{12k+4_2k+1equations_1}, \eqref{12k+4_2k+1equations_2} is $\textup{c}((5k+2)(12k+4)^2 + (6k+5)(12k+5)^2) = 11k+7$ and that the catenary degrees of the common values of equations \eqref{12k+4minrelation0}-\eqref{12k+4minrelation3} are smaller than that, which gives
$$ \textup{c}(S) = 11k+7.$$

Lastly, we apply Theorem \ref{theorem_minimalpresentation} and construct a minimal presentation, analogously as before. The cardinality of the minimal presentation is $4k+8$.  

\subsection{$n = 12k +8, \ k>0$}
\subsubsection{Ap{\'e}ry set}
We have the following:
\begin{align}
    (7k+6)(12k+8)^2 &= 0(12k+9)^2 + (3k-1)(12k+10)^2 + (4k+4)(12k+11)^2, \label{12k+8minrelation0} \\
    (6k+8)(12k+9)^2 &= 1(12k+8)^2 + 1(12k+10)^2 + (6k+4)(12k+11)^2, \label{12k+8minrelation1} \\
    (9k+7)(12k+10)^2 &= (5k+4)(12k+8)^2 + 1(12k+9)^2 + (4k+3)(12k+11)^2, \label{12k+8minrelation2}\\
    (6k+5)(12k+11)^2 &= 0(12k+8)^2 + (6k+5)(12k+9)^2 + 2(12k+10)^2,  \label{12k+8minrelation3}
\end{align}
and by Lemmas \ref{1remlemma} and \ref{lemmaminrelation1}, constructing points from the above equations and deleting their associated regions, the Ap{\'e}ry set is contained in the set of labels of the remaining figure.
Furthermore, we have the following relations that satisfy the conditions of Theorem \ref{theoremprocedure}:
\begin{align}
(7k+5)(12k+8)^2 + 3(12k+9)^2  &= (3k+2)(12k+10)^2 + (4k+3)(12k+11)^2, \label{12k+8_eq1} \\
(12k+8)^2 + 3(12k+10)^2 &= 3(12k+9)^2 + (12k+11)^2,\label{12k+8_eq2} \\
(7k+6)(12k+8)^2 + (2k+1)(12k+11)^2 &= (6k+5)(12k+9)^2 + (3k+1)(12k+10)^2. \label{12k+8_eq3}  
\end{align}
Subtracting equation \eqref{12k+8_eq2} from \eqref{12k+8_eq1}, we obtain additional $2k+1$ relations:
\begin{equation}\label{12k+8_2k+1equations_1}
\begin{aligned}    
(7k+4)(12k+8)^2 + 6(12k+9)^2  &= (3k+5)(12k+10)^2 + (4k+2)(12k+11)^2 ,\\
(7k+3)(12k+8)^2 + 9(12k+9)^2  &= (3k+8)(12k+10)^2 + (4k+1)(12k+11)^2 ,\\
 & \hspace{5.6pt}  \vdots      \\
(5k+5)(12k+8)^2 + (6k+3)(12k+9)^2  &= (9k+2)(12k+10)^2 + (2k+3)(12k+11)^2 ,\\
(5k+4)(12k+8)^2 + (6k+6)(12k+9)^2  &= (9k+5)(12k+10)^2 + (2k+2)(12k+11)^2 .
\end{aligned}
\end{equation}
Similarly, we can subtract equation \eqref{12k+8_eq2} from \eqref{12k+8_eq3} and obtain further $2k+1$ relations:
\begin{equation}\label{12k+8_2k+1equations_2}
\begin{aligned}    
(7k+5)(12k+8)^2 + (2k+2)(12k+11)^2 &= (6k+2)(12k+9)^2 + (3k+4)(12k+10)^2 ,\\
(7k+4)(12k+8)^2 + (2k+3)(12k+11)^2 &= (6k-1)(12k+9)^2 + (3k+7)(12k+10)^2 ,\\
 & \hspace{5.6pt}  \vdots      \\
(5k+6)(12k+8)^2 + (4k+1)(12k+11)^2 &= 5(12k+9)^2 + (9k+1)(12k+10)^2 .\\
(5k+5)(12k+8)^2 + (4k+2)(12k+11)^2 &= 2(12k+9)^2 + (9k+4)(12k+10)^2 .
\end{aligned}
\end{equation}

Analogously as before, one can check that the figure $T$, carved out using all the above equations, has the volume $(12k+8)^2$, which implies that it is an $L$-shape.

\begin{figure}[h]
    \caption{the figure $T$ for $n=12k+8$ and $k=2$}
    \centering
    \includegraphics[width = 0.5\textwidth]{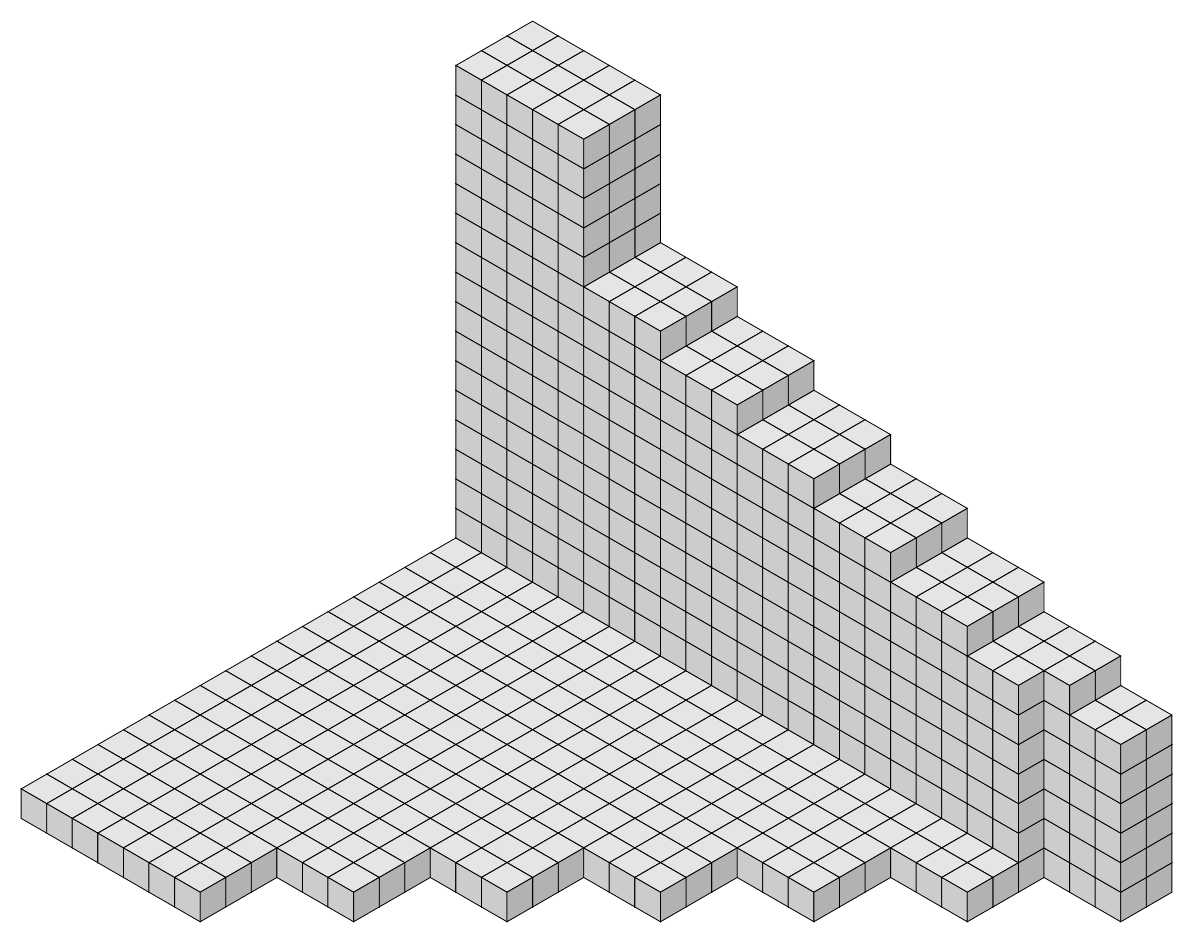}
    \label{img3}
\end{figure}

\subsubsection{Frobenius number and genus}

The largest label in $T$ is
$$ (12k+9)^2 + (9k+6)(12k+10)^2 + (2k+1)(12k+11)^2 =1584 k^3 + 3840 k^2 + 3062 k + 802,$$
which gives
$$F(S) = 1584 k^3 + 3696 k^2 + 2870 k + 738.$$
One can find that the sum of the labels of $T$ is 
$$ 120960 k^5 + 457632 k^4 + 686640 k^3 + 511128 k^2 + 188896 k + 27744,$$
which gives
$$ G(S)  = 840 k^3 + 1986 k^2 + 1555 k + 402.$$
\subsubsection{Catenary degree}
Using Proposition \ref{propminimalrelations} and Remark \ref{remark_a00_procedure}, we find that equations \eqref{12k+8minrelation0}-\eqref{12k+8minrelation3} are in fact the minimal relations.

The common value of equation \eqref{12k+8_eq3} is not a Betti element, by Lemma \ref{lemma_catenarycase1}. 
By Lemmas \ref{lemma_catenary2} and \ref{lemma_catenary1}, the common values of each of equations \eqref{12k+8minrelation0}-\eqref{12k+8_eq2}, \eqref{12k+8_2k+1equations_1}, \eqref{12k+8_2k+1equations_2} are Betti elements. 
By Theorem \ref{theorembettifind}, they are all the Betti elements, and there are $4k+8$ of them.

Similarly as before, using Lemma \ref{lemma_catenary1}, one can check that the largest catenary degree among the common values of equations \eqref{12k+8_eq1}, \eqref{12k+8_eq2}, \eqref{12k+8_2k+1equations_1}, \eqref{12k+8_2k+1equations_2} is $\textup{c}((5k+4)(12k+8)^2 + (6k+6)(12k+9)^2 ) = 11k+10$ and that the catenary degrees of the common values of equations \eqref{12k+8minrelation0}-\eqref{12k+8minrelation3} are smaller than that, which gives
$$ \textup{c}(S) = 11k+10.$$

Lastly, we apply Theorem \ref{theorem_minimalpresentation} and construct a minimal presentation, analogously as before. The cardinality of the minimal presentation is $4k+8$.

\subsection{$n = 12k + 1, \ k>1$}
\subsubsection{Ap{\'e}ry set}
We have the following:
\begin{align}
    (6k+2)(12k+1)^2 &= 1(12k+2)^2 + (6k-2)(12k+3)^2 + 1(12k+4)^2, \label{12k+1minrelation0} \\
    (9k+5)(12k+2)^2 &= (4k+2)(12k+1)^2 + 2(12k+3)^2 + 5k(12k+4)^2, \label{12k+1minrelation1} \\
    (6k+1)(12k+3)^2 &= (6k+1)(12k+1)^2 + 2(12k+2)^2 + 0(12k+4)^2, \label{12k+1minrelation2}\\
    (7k+1)(12k+4)^2 &= 4k(12k+1)^2 + (3k+4)(12k+2)^2 + 0(12k+3)^2,  \label{12k+1minrelation3}
\end{align}
and by Lemmas \ref{1remlemma} and \ref{lemmaminrelation1}, constructing points from the above equations and deleting their associated regions, the Ap{\'e}ry set is contained in the set of labels of the remaining figure.
Furthermore, we have the following relations that satisfy the conditions of Theorem \ref{theoremprocedure}:
\begin{align}
  (4k-1)(12k+1)^2 + (3k+7)(12k+2)^2  &= 3(12k+3)^2 + (7k)(12k+4)^2, \label{12k+1_eq1} \\
    (12k+1)^2 + 3(12k+3)^2 &= 3(12k+2)^2 + (12k+4)^2, \label{12k+1_eq2} \\
    (4k+1)(12k+1)^2 + (5k+1)(12k+4)^2  &= (9k+2)(12k+2)^2 + (12k+3)^2. \label{12k+1_eq3}
\end{align}
Subtracting equation \eqref{12k+1_eq2} from \eqref{12k+1_eq1}, we obtain additional $2k-1$ relations:
\begin{equation}\label{12k+1_2k-1equations_1}
\begin{aligned}    
(4k-2)(12k+1)^2 + (3k+10)(12k+2)^2  &= 6(12k+3)^2 + (7k-1)(12k+4)^2,\\
(4k-3)(12k+1)^2 + (3k+13)(12k+2)^2  &= 9(12k+3)^2 + (7k-2)(12k+4)^2,\\
 & \hspace{5.6pt}  \vdots      \\
(2k+1)(12k+1)^2 + (9k+1)(12k+2)^2  &= (6k-3)(12k+3)^2 + (5k+2)(12k+4)^2,\\
(2k)(12k+1)^2 + (9k+4)(12k+2)^2  &= (6k)(12k+3)^2 + (5k+1)(12k+4)^2.
\end{aligned}
\end{equation}
Similarly, we can subtract equation \eqref{12k+1_eq2} from \eqref{12k+1_eq3} and obtain further $2k-1$ relations:
\begin{equation}\label{12k+1_2k-1equations_2}
\begin{aligned}    
 (4k)(12k+1)^2 + (5k+2)(12k+4)^2  &= (9k-1)(12k+2)^2 + 4(12k+3)^2,\\
(4k-1)(12k+1)^2 + (5k+3)(12k+4)^2  &= (9k-4)(12k+2)^2 + 7(12k+3)^2,\\
 & \hspace{5.6pt}  \vdots      \\
(2k+3)(12k+1)^2 + (7k-1)(12k+4)^2  &= (3k+8)(12k+2)^2 + (6k-5)(12k+3)^2,\\
(2k+2)(12k+1)^2 + (7k)(12k+4)^2  &= (3k+5)(12k+2)^2 + (6k-2)(12k+3)^2,
\end{aligned}
\end{equation}

Analogously as before, one can check that the figure $T$, carved out using all the above equations, has the volume $(12k+1)^2$, which implies that it is an $L$-shape.

\begin{figure}[h]
    \caption{the figure $T$ for $n=12k+1$ and $k=2$}
    \centering
    \includegraphics[width = 0.5\textwidth]{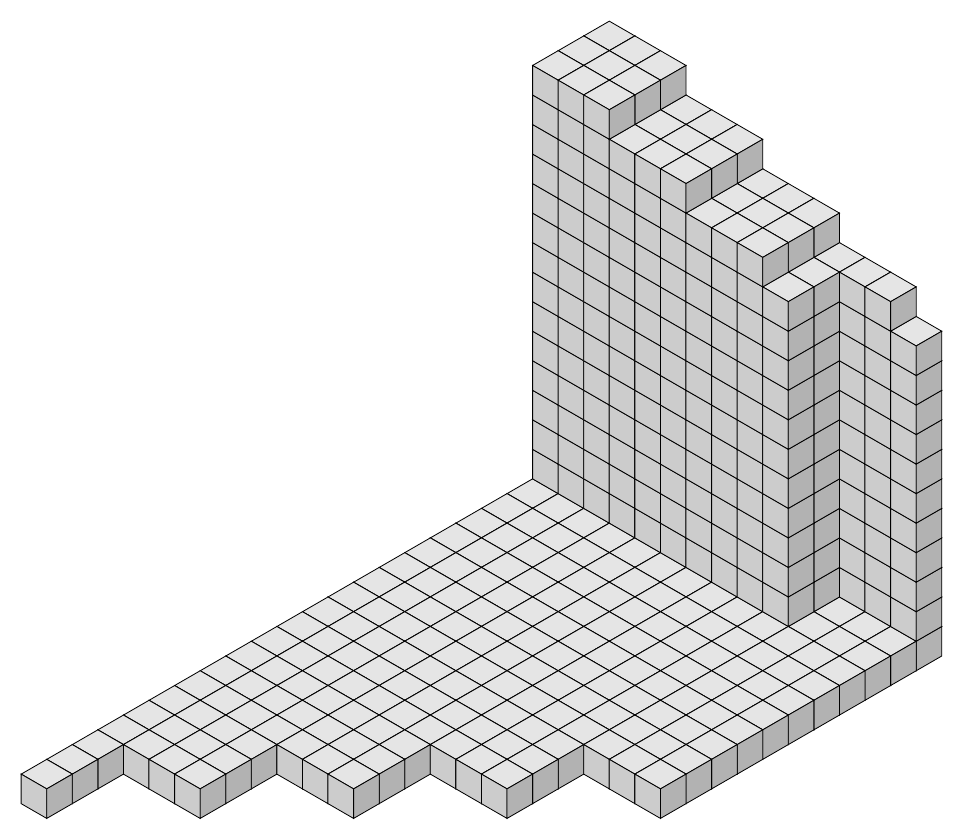}
    \label{img4}
\end{figure}

\subsubsection{Frobenius number and genus}

The largest label in $T$ is
$$ (6k-1)(12k+3)^2 + (5k+1)(12k+4)^2 =1584 k^3 + 912 k^2 + 158 k + 7,$$
which gives
$$F(S) = 1584 k^3 + 768 k^2 + 134 k + 6.$$
One can find that the sum of the labels of $T$ is 
$$ 120960 k^5 + 107424 k^4 + 31944 k^3 + 4230 k^2 + 259 k + 6,$$
which gives
$$ G(S)  = 840 k^3 + 534 k^2 + 103 k + 6.$$

\subsubsection{Catenary degree}
Using Proposition \ref{propminimalrelations} and Remark \ref{remark_a00_procedure}, we find that equations \eqref{12k+1minrelation0}-\eqref{12k+1minrelation3} are in fact the minimal relations.

By Lemmas \ref{lemma_catenary2} and \ref{lemma_catenary1}, the common values of each of equations \eqref{12k+1minrelation0}-\eqref{12k+1_2k-1equations_2} are Betti elements. 
By Theorem \ref{theorembettifind}, they are all the Betti elements, and there are $4k+5$ of them.

Similarly as before, using Lemma \ref{lemma_catenary1}, one can check that the largest catenary degree among the common values of equations \eqref{12k+1_eq1}-\eqref{12k+1_2k-1equations_2} is $\textup{c}((2k)(12k+1)^2 + (9k+4)(12k+2)^2 ) = 11k+4$ and that the catenary degrees of the common values of equations \eqref{12k+1minrelation0}-\eqref{12k+1minrelation3} are smaller than that, which gives
$$ \textup{c}(S) = 11k+4.$$

Lastly, we apply Theorem \ref{theorem_minimalpresentation} and construct a minimal presentation, analogously as before. The cardinality of the minimal presentation is $4k+5$.

\subsection{$n = 12k + 5, \ k>0$}
\subsubsection{Ap{\'e}ry set}
We have the following:
\begin{align}
    (6k+4)(12k+5)^2 &= 1(12k+6)^2 + (6k)(12k+7)^2 + 1(12k+8)^2, \label{12k+5minrelation0} \\
    (9k+7)(12k+6)^2 &= (4k+3)(12k+5)^2 + 1(12k+7)^2 + (5k+2)(12k+8)^2, \label{12k+5minrelation1} \\
    (6k+3)(12k+7)^2 &= (6k+3)(12k+5)^2 + 2(12k+6)^2 + 0(12k+8)^2, \label{12k+5minrelation2}\\
    (7k+4)(12k+8)^2 &= (4k+2)(12k+5)^2 + (3k+3)(12k+6)^2 + 2(12k+7)^2,  \label{12k+5minrelation3}
\end{align}
and by Lemmas \ref{1remlemma} and \ref{lemmaminrelation1}, constructing points from the above equations and deleting their associated regions, the Ap{\'e}ry set is contained in the set of labels of the remaining figure.
Furthermore, we have the following relations that satisfy the conditions of Theorem \ref{theoremprocedure}:
\begin{align}
  (4k+1)(12k+5)^2 + (3k+6)(12k+6)^2  &= (12k+7)^2 + (7k+3)(12k+8)^2, \label{12k+5_eq1}\\
    (12k+5)^2 + 3(12k+7)^2 &= 3(12k+6)^2 + (12k+8)^2,\label{12k+5_eq2}\\
    (4k+2)(12k+5)^2 + (5k+3)(12k+8)^2  &= (9k+4)(12k+6)^2 + 2(12k+7)^2.\label{12k+5_eq3}
\end{align}
Subtracting equation \eqref{12k+5_eq2} from \eqref{12k+5_eq1}, we obtain additional $2k$ relations:
\begin{equation}\label{12k+5_2kequations_1}
\begin{aligned}    
(4k)(12k+5)^2 + (3k+9)(12k+6)^2  &= 4(12k+7)^2 + (7k+2)(12k+8)^2,\\
(4k-1)(12k+5)^2 + (3k+12)(12k+6)^2  &= 7(12k+7)^2 + (7k+1)(12k+8)^2,\\
 & \hspace{5.6pt}  \vdots      \\
(2k+2)(12k+5)^2 + (9k+3)(12k+6)^2  &= (6k-2)(12k+7)^2 + (5k+4)(12k+8)^2,\\
(2k+1)(12k+5)^2 + (9k+6)(12k+6)^2  &= (6k+1)(12k+7)^2 + (5k+3)(12k+8)^2.
\end{aligned}
\end{equation}
Similarly, we can subtract equation \eqref{12k+5_eq2} from \eqref{12k+5_eq3} and obtain further $2k$ relations:
\begin{equation}\label{12k+5_2kequations_2}
\begin{aligned}    
(4k+1)(12k+5)^2 + (5k+4)(12k+8)^2  = (9k+1)(12k+6)^2 + 5(12k+7)^2,\\
(4k)(12k+5)^2 + (5k+5)(12k+8)^2  = (9k-2)(12k+6)^2 + 8(12k+7)^2,\\
 & \hspace{5.6pt}  \vdots      \\
(2k+3)(12k+5)^2 + (7k+2)(12k+8)^2  = (3k+7)(12k+6)^2 + (6k-1)(12k+7)^2,\\
(2k+2)(12k+5)^2 + (7k+3)(12k+8)^2  = (3k+4)(12k+6)^2 + (6k+2)(12k+7)^2.
\end{aligned}
\end{equation}

Analogously as before, one can check that the figure $T$, carved out using all the above equations, has the volume $(12k+5)^2$, which implies that it is an $L$-shape.

\begin{figure}[h]
    \caption{the figure $T$ for $n=12k+5$ and $k=2$}
    \centering
    \includegraphics[width = 0.5\textwidth]{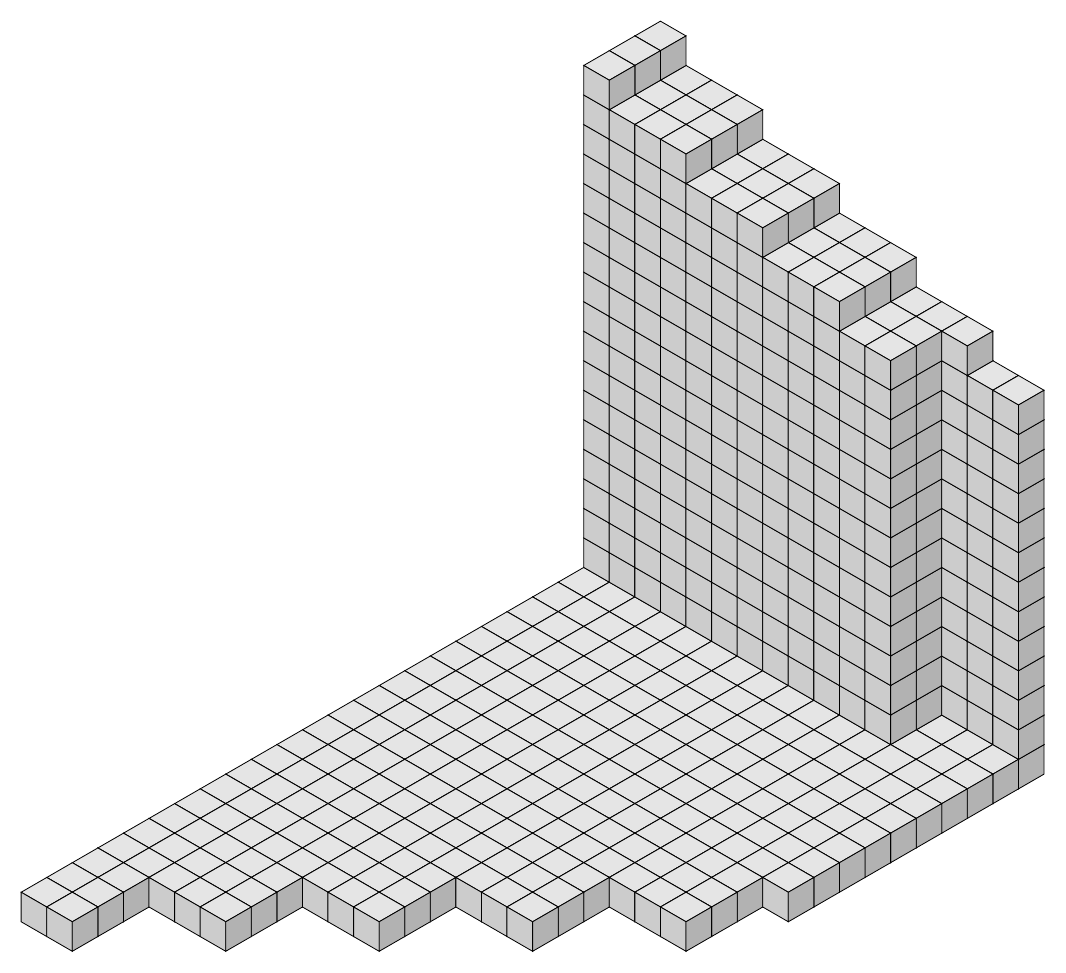}
    \label{img5}
\end{figure}

\subsubsection{Frobenius number and genus}

The largest label in $T$ is 
$$ (6k+2)(12k+7)^2 + (5k+2)(12k+8)^2 =1584 k^3 + 2544 k^2 + 1334 k + 226,$$
which gives
$$F(S) = 1584 k^3 + 2400 k^2 + 1214 k + 201.$$
One can find that the sum of the labels of $T$ is 
$$ 120960 k^5 + 309024 k^4 + 309288 k^3 + 152094 k^2 + 36845 k + 3525,$$
which gives
$$ G(S)  = 840 k^3 + 1374 k^2 + 737 k + 129.$$

\subsubsection{Catenary degree}
Using Proposition \ref{propminimalrelations} and Remark \ref{remark_a00_procedure}, we find that equations \eqref{12k+5minrelation0}-\eqref{12k+5minrelation3} are in fact the minimal relations.

By Lemmas \ref{lemma_catenary2} and \ref{lemma_catenary1}, the common values of each of equations \eqref{12k+5minrelation0}-\eqref{12k+5_2kequations_2} are Betti elements. 
By Theorem \ref{theorembettifind}, they are all the Betti elements, and there are $4k+7$ of them.

Similarly as before, using Lemma \ref{lemma_catenary1}, one can check that the largest catenary degree among the common values of equations \eqref{12k+5_eq1}-\eqref{12k+5_2kequations_2} is $\textup{c}((2k+1)(12k+5)^2 + (9k+6)(12k+6)^2 ) = 11k+7$ and that the catenary degrees of the common values of equations \eqref{12k+5minrelation0}-\eqref{12k+5minrelation3} are smaller than that, which gives
$$ \textup{c}(S) = 11k+7.$$

Lastly, we apply Theorem \ref{theorem_minimalpresentation} and construct a minimal presentation, analogously as before. The cardinality of the minimal presentation is $4k+7$.

\subsection{$n = 12k +9, \ k \geq 0$}
\subsubsection{Ap{\'e}ry set}
We have the following:
\begin{align}
    (6k+6)(12k+9)^2 &= 1(12k+10)^2 + (6k+2)(12k+11)^2 + 1(12k+12)^2, \label{12k+9minrelation0} \\
    (9k+9)(12k+10)^2 &= (4k+4)(12k+9)^2 + 0(12k+11)^2 + (5k+4)(12k+12)^2, \label{12k+9minrelation1} \\
    (6k+5)(12k+11)^2 &= (6k+5)(12k+9)^2 + 2(12k+10)^2 + 0(12k+12)^2, \label{12k+9minrelation2}\\
    (7k+6)(12k+12)^2 &= (4k+3)(12k+9)^2 + (3k+5)(12k+10)^2 + 1(12k+11)^2,  \label{12k+9minrelation3}
\end{align}
and by Lemmas \ref{1remlemma} and \ref{lemmaminrelation1}, constructing points from the above equations and deleting their associated regions, the Ap{\'e}ry set is contained in the set of labels of the remaining figure.
Furthermore, we have the following relations that satisfy the conditions of Theorem \ref{theoremprocedure}:
\begin{align}
  (4k+2)(12k+9)^2 + (3k+8)(12k+10)^2  &= 2(12k+11)^2 + (7k+5)(12k+12)^2,\label{12k+9_eq1}\\
    (12k+9)^2 + 3(12k+11)^2 &= 3(12k+10)^2 + (12k+12)^2,\label{12k+9_eq2}\\
    (4k+3)(12k+9)^2 + (5k+5)(12k+12)^2  &= (9k+6)(12k+10)^2 + 3(12k+11)^2.\label{12k+9_eq3}
\end{align}
Subtracting equation \eqref{12k+9_eq2} from \eqref{12k+9_eq1}, we obtain additional $2k$ relations:
\begin{equation}\label{12k+9_2kequations_1}
\begin{aligned}    
(4k+1)(12k+9)^2 + (3k+11)(12k+10)^2  &= 5(12k+11)^2 + (7k+4)(12k+12)^2,\\
(4k)(12k+9)^2 + (3k+14)(12k+10)^2  &= 8(12k+11)^2 + (7k+3)(12k+12)^2,\\
 & \hspace{5.6pt}  \vdots      \\
(2k+3)(12k+9)^2 + (9k+5)(12k+10)^2  &= (6k-1)(12k+11)^2 + (5k+6)(12k+12)^2,\\
(2k+2)(12k+9)^2 + (9k+8)(12k+10)^2  &= (6k+2)(12k+11)^2 + (5k+5)(12k+12)^2.
\end{aligned}
\end{equation}
Similarly, we can subtract equation \eqref{12k+9_eq2} from \eqref{12k+9_eq3} and obtain further $2k$ relations:
\begin{equation}\label{12k+9_2kequations_2}
\begin{aligned}    
 (4k+2)(12k+9)^2 + (5k+6)(12k+12)^2  &= (9k+3)(12k+10)^2 + 6(12k+11)^2,\\
(4k+1)(12k+9)^2 + (5k+7)(12k+12)^2  &= (9k)(12k+10)^2 + 9(12k+11)^2,\\
 & \hspace{5.6pt}  \vdots      \\
(2k+4)(12k+9)^2 + (7k+4)(12k+12)^2  &= (3k+9)(12k+10)^2 + (6k)(12k+11)^2,\\
(2k+3)(12k+9)^2 + (7k+5)(12k+12)^2  &= (3k+6)(12k+10)^2 + (6k+3)(12k+11)^2.
\end{aligned}
\end{equation}

Now analogously as before, one can check that the figure $T$, carved out using all the above equations, has the volume $(12k+9)^2$, which implies that it is an $L$-shape.

\begin{figure}[h]
    \caption{the figure $T$ for $n=12k+9$ and $k=2$}
    \centering
    \includegraphics[width = 0.5\textwidth]{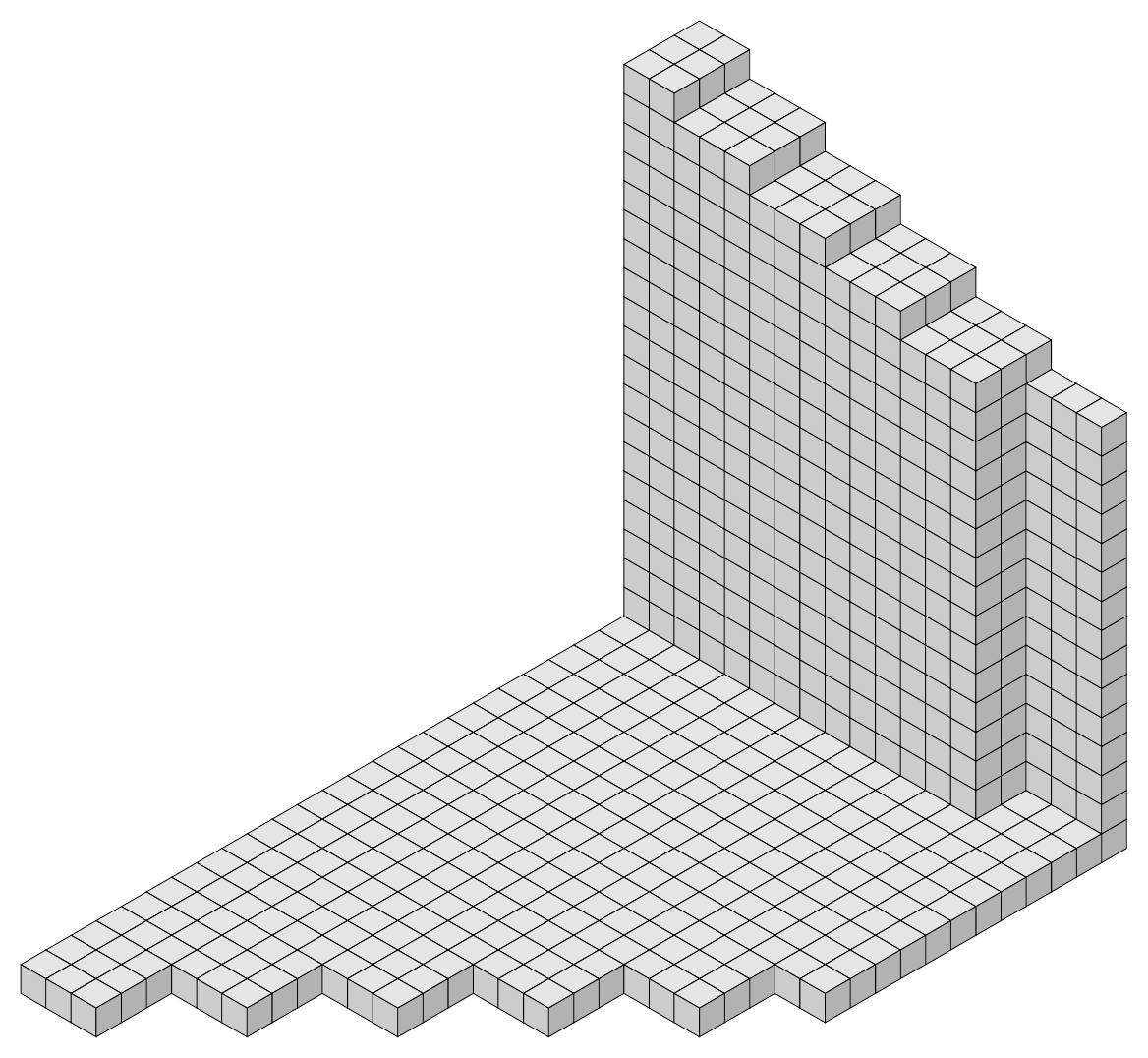}
    \label{img6}
\end{figure}

\subsubsection{Frobenius number and genus}

The largest label in $T$ is 
$$ (6k+4)(12k+11)^2 +(5k+4)(12k+12)^2 =1584 k^3 + 4176 k^2 + 3654 k + 1060,$$ 
which gives
$$F(S) = 1584 k^3 + 4032 k^2 + 3438 k + 979.$$
One can find that the sum of the labels of $T$ is 
$$ 120960 k^5 + 510624 k^4 + 856008 k^3 + 712854 k^2 + 295083 k + 48600,$$
which gives
$$ G(S)  = 840 k^3 + 2214 k^2 + 1935 k + 560.$$

\subsubsection{Catenary degree}
Using Proposition \ref{propminimalrelations} and Remark \ref{remark_a00_procedure}, we find that equations \eqref{12k+9minrelation0}-\eqref{12k+9minrelation3} are in fact the minimal relations.

By Lemmas \ref{lemma_catenary2} and \ref{lemma_catenary1}, the common values of each of equations \eqref{12k+9minrelation0}-\eqref{12k+9_2kequations_2} are Betti elements. 
By Theorem \ref{theorembettifind}, they are all the Betti elements, and there are $4k+7$ of them.

Similarly as before, using Lemma \ref{lemma_catenary1}, one can check that the largest catenary degree among the common values of equations \eqref{12k+9_eq1}-\eqref{12k+9_2kequations_2} is $\textup{c}((2k+2)(12k+9)^2 + (9k+8)(12k+10)^2 ) = 11k+10$ and that the catenary degrees of the common values of equations \eqref{12k+9minrelation0}-\eqref{12k+9minrelation3} are smaller than that, which gives
$$ \textup{c}(S) = 11k+10.$$

Lastly, we apply Theorem \ref{theorem_minimalpresentation} and construct a minimal presentation, analogously as before. The cardinality of the minimal presentation is $4k+7$.

\subsection{$n = 4m+2, \ m>1$}
\subsubsection{Ap{\'e}ry set}
We have the following:
\begin{align}
    (4m+5)(4m+2)^2 &= 2(4m+3)^2 + (4m-3)(4m+4)^2 + 2(4m+5)^2, \label{4m+2minrelation0} \\
    (m+4)(4m+3)^2 &= 1(4m+2)^2 + 2(4m+4)^2 + m(4m+5)^2, \label{4m+2minrelation1} \\
    (4m+3)(4m+4)^2 &= 4(4m+2)^2 + (4m+3)(4m+3)^2 + 0(4m+5)^2, \label{4m+2minrelation2}\\
    (m+1)(4m+5)^2 &= 0(4m+2)^2 + (m+1)(4m+3)^2 + 1(4m+4)^2,  \label{4m+2minrelation3}
\end{align}
and by Lemmas \ref{1remlemma} and \ref{lemmaminrelation1}, constructing points from the above equations and deleting their associated regions, the Ap{\'e}ry set is contained in the set of labels of the remaining figure.
Furthermore, we have the following relations that satisfy the conditions of Theorem \ref{theoremprocedure}:
\begin{align}
    (4m+4)(4m+2)^2 + (4m+3)^2 &= (4m)(4m+4)^2 + (4m+5)^2,\label{4m+2eq1}\\
    (4m+2)^2+3(4m+4)^2 &= 3(4m+3)^2 + (4m+5)^2,\label{4m+2eq2}\\
    (4m+5)(4m+2)^2 + (m-1)(4m+5)^2 &= (m+3)(4m+3)^2 + (4m-2)(4m+4)^2.\label{4m+2eq3}
\end{align}
Subtract equation \eqref{4m+2eq2} from \eqref{4m+2eq3} to get
\begin{equation}\label{4m+2eq4}
     (4m+4)(4m+2)^2 + (m)(4m+5)^2 = (m)(4m+3)^2 + (4m+1)(4m+4)^2.
\end{equation}

Now analogously as before, one can check that the figure $T$, carved out using all the above equations, has the volume $(4m+2)^2$, which implies that it is an $L$-shape.

\begin{figure}[h]
    \caption{the figure $T$ for $n =4m+2$ and $m=7$}
    \centering
    \includegraphics[width = 0.5\textwidth]{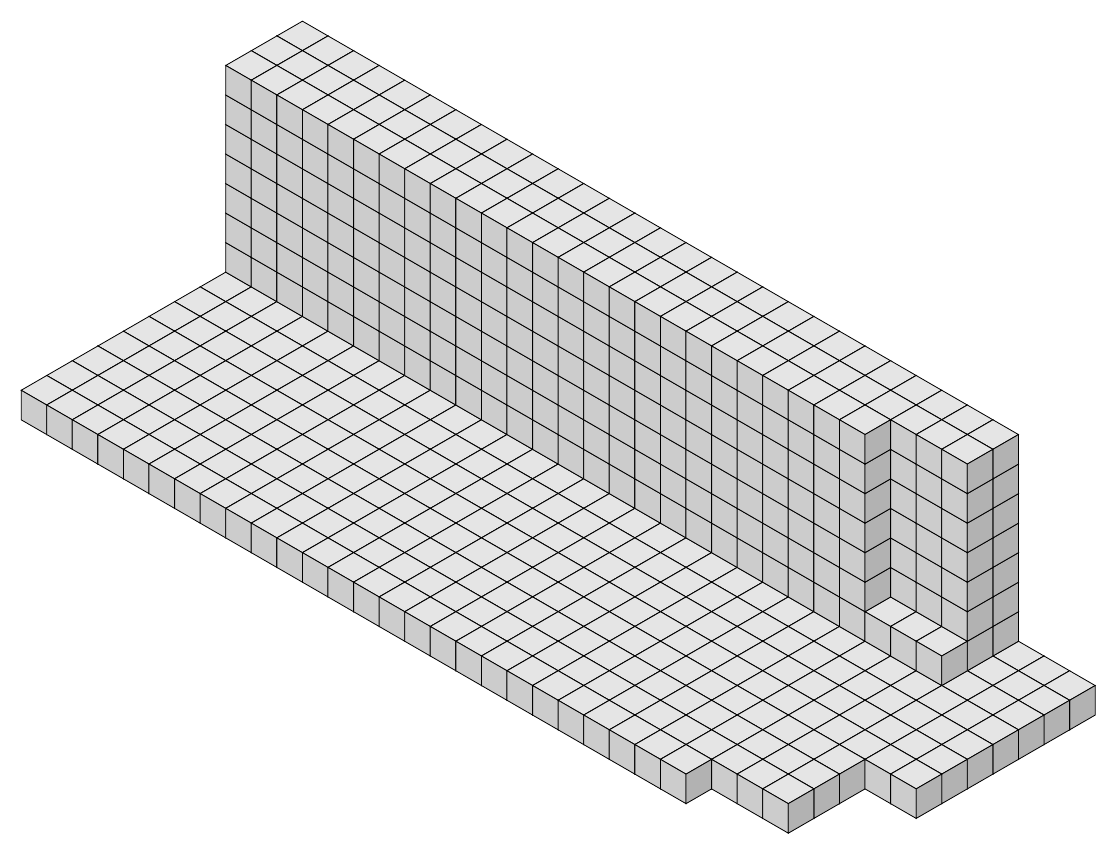}
    \label{img7}
\end{figure}

\subsubsection{Frobenius number and genus}

The largest label in $T$ is 
$$(m+2)(4m+3)^2 + (4m)(4m+4)^2 =80 m^3 + 184 m^2 + 121 m + 18,$$
which gives
$$ F(S) = 80 m^3 + 168 m^2 + 105 m + 14. $$
One can find that the sum of the labels of $T$ is 
$$ 640 m^5 + 2208 m^4 + 2816 m^3 + 1696 m^2 + 488 m + 54,$$
which gives
$$ G(S)  = 40 m^3 + 90 m^2 + 60 m + 12.$$

\subsubsection{Catenary degree}
Using Proposition \ref{propminimalrelations} and Remark \ref{remark_a00_procedure}, we find that equations \eqref{4m+2minrelation0}-\eqref{4m+2minrelation3} are in fact the minimal relations.

The common value of equation \eqref{4m+2eq3} is not a Betti element, by Lemma \ref{lemma_catenarycase1}.
By Lemmas \ref{lemma_catenary2} and \ref{lemma_catenary1}, the common values of each of equations \eqref{4m+2minrelation0}-\eqref{4m+2eq2}, \eqref{4m+2eq4} are Betti elements. 
By Theorem \ref{theorembettifind}, they are all the Betti elements, and there are $7$ of them.

Similarly as before, using Lemma \ref{lemma_catenary1}, one can check that the largest catenary degree among the common values of equations \eqref{4m+2eq1}, \eqref{4m+2eq2}, \eqref{4m+2eq4} is $\textup{c}((4m+4)(4m+2)^2 + (m)(4m+5)^2) = 5m+4$ and that the catenary degrees of the common values of equations \eqref{4m+2minrelation0}-\eqref{4m+2minrelation3} are smaller than that, which gives
$$ \textup{c}(S) = 5m+4.$$

Lastly, we apply Theorem \ref{theorem_minimalpresentation} and construct a minimal presentation, analogously as before. The cardinality of the minimal presentation is $7$.

\subsection{$n = 4m+3, \ m > 1$}
\subsubsection{Ap{\'e}ry set}
We have the following:
\begin{align}
    (m+2)(4m+3)^2 &= 2(4m+4)^2 + (m-2)(4m+5)^2 + 1(4m+6)^2, \label{4m+3minrelation0} \\
    (4m+11)(4m+4)^2 &= 2(4m+3)^2 + 2(4m+5)^2 + (4m+3)(4m+6)^2, \label{4m+3minrelation1} \\
    (m+1)(4m+5)^2 &= (m+1)(4m+3)^2 + 1(4m+4)^2 + 0(4m+6)^2, \label{4m+3minrelation2}\\
    (4m+5)(4m+6)^2 &= 0(4m+3)^2 + (4m+5)(4m+4)^2 + 4(4m+5)^2,  \label{4m+3minrelation3}
\end{align}
and by Lemmas \ref{1remlemma} and \ref{lemmaminrelation1}, constructing points from the above equations and deleting their associated regions, the Ap{\'e}ry set is contained in the set of labels of the remaining figure.
Furthermore, we have the following relations that satisfy the conditions of Theorem \ref{theoremprocedure}:
\begin{align}
    (m)(4m+3)^2 + (4m+9)(4m+4)^2 &= (m)(4m+5)^2 + (4m+4)(4m+6)^2, \label{4m+3eq1}\\
    (4m+3)^2+3(4m+5)^2 &= 3(4m+4)^2 + (4m+6)^2,\label{4m+3eq2}\\
    (4m+3)^2+(4m+4)(4m+6)^2 &= (4m+8)(4m+4)^2 + (4m+5)^2,\label{4m+3eq3}
\end{align}

Now analogously as before, one can check that the figure $T$, carved out using all the above equations, has the volume $(4m+3)^2$, which implies that it is an $L$-shape.

\begin{figure}[h]
    \caption{the figure $T$ for $n =4m+3$ and $m=6$}
    \centering
    \includegraphics[width = 0.45\textwidth]{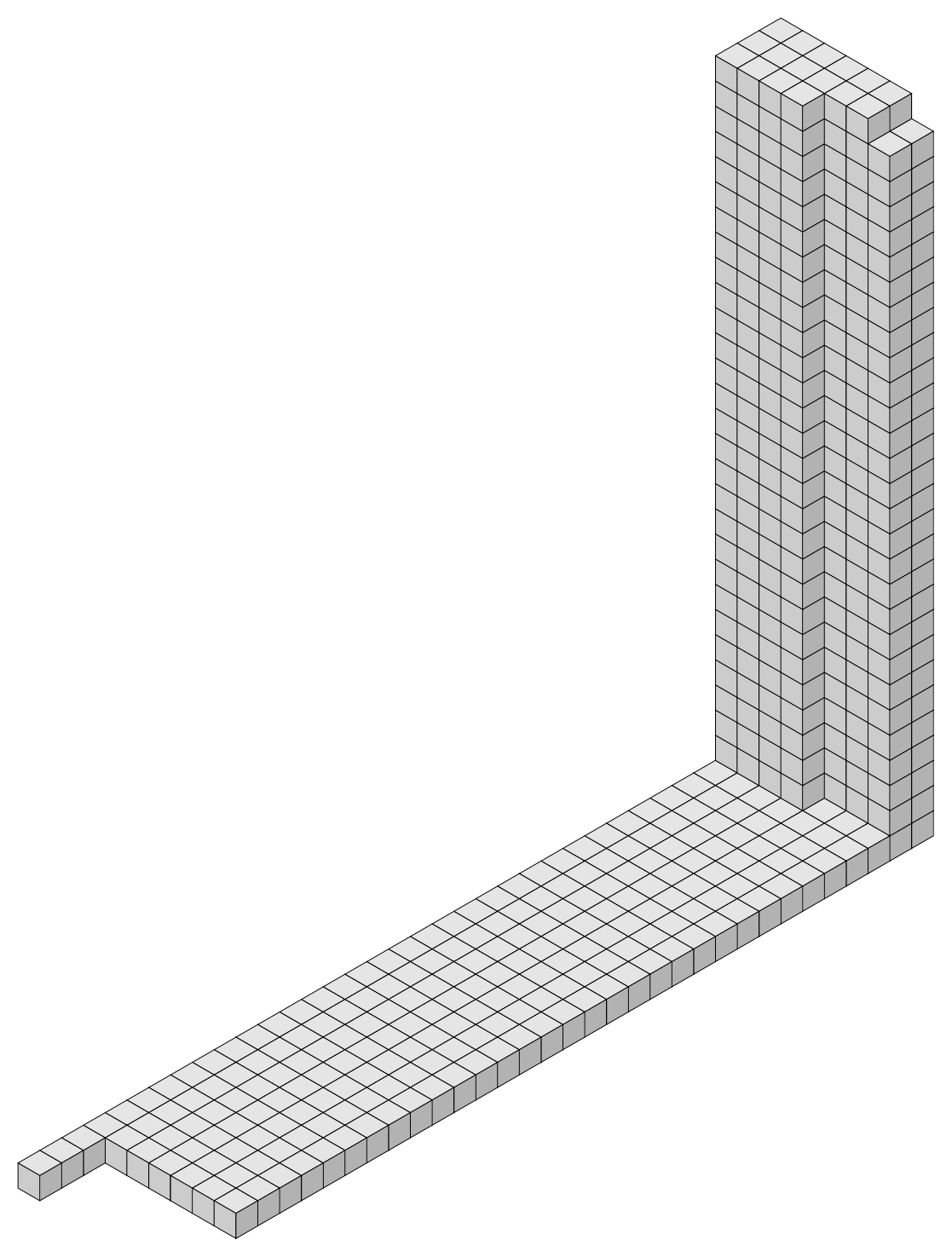}
    \label{img8}
\end{figure}

\subsubsection{Frobenius number and genus}

The largest label in $T$ is 
$$ (4m+4)^2 + (m-1)(4m+5)^2 + (4m+4)(4m+6)^2 =80 m^3 + 296 m^2 + 353 m + 135,$$
which gives
$$ F(S) = 80 m^3 + 280 m^2 + 329 m + 126.$$
One can find that the sum of the labels in $T$ is 
$$ 640 m^5 + 3488 m^4 + 7224 m^3 + 7246 m^2 + 3552 m + 684,$$
which gives
$$ G(S)  = 40 m^3 + 150 m^2 + 180 m + 72.$$

\subsubsection{Catenary degree}
Using Proposition \ref{propminimalrelations} and Remark \ref{remark_a00_procedure}, we find that equations \eqref{4m+3minrelation0}-\eqref{4m+3minrelation3} are in fact the minimal relations.

By Lemmas \ref{lemma_catenary2} and \ref{lemma_catenary1}, the common values of each of equations \eqref{4m+3minrelation0}-\eqref{4m+3eq3} are Betti elements. 
By Theorem \ref{theorembettifind}, they are all the Betti elements, and there are $7$ of them.

Similarly as before, using Lemma \ref{lemma_catenary1}, one can check that the largest catenary degree among the common values of equations \eqref{4m+3eq1}-\eqref{4m+3eq3} is $\textup{c}((m)(4m+3)^2 + (4m+9)(4m+4)^2) = 5m+9$ and that the catenary degrees of the common values of equations \eqref{4m+3minrelation0}-\eqref{4m+3minrelation3} are smaller than that, which gives
$$ \textup{c}(S) = 5m+9.$$

Lastly, we apply Theorem \ref{theorem_minimalpresentation} and construct a minimal presentation, analogously as before. The cardinality of the minimal presentation is $7$.

\subsection{Summary}

We can use \texttt{numericalsgps} \cite{NumericalSgps} \texttt{GAP} \cite{GAP4} package to supplement the above results, with all small cases of $n$, which were not covered so far.  

\begin{cor}\label{corfrobsquares}
The Frobenius number of the numerical semigroup generated by the four consecutive squares $n^2$, $(n+1)^2$, $(n+2)^2$, $(n+3)^2$ is as follows:
$$
\begin{cases}
\begin{aligned}
&\frac{11}{12} n^3 + 4 n^2 + \frac{9}{2} n - 2
    &&\text{if} \ \ n \equiv 0 \pmod{12},\ \text{and } n \ge 24,\\
&\frac{11}{12} n^3 + \frac{10}{3} n^2 + \frac{9}{2} n - 2
    &&\text{if} \ \ n \equiv 4 \pmod{12},\ \text{and } n \ge 16,\\
&\frac{11}{12} n^3 + \frac{11}{3} n^2 + \frac{9}{2} n - 2
    &&\text{if} \ \ n \equiv 8 \pmod{12},\ \text{and } n \ge 20,\\
&\frac{11}{12} n^3 + \frac{31}{12} n^2 + \frac{13}{4} n - \frac{3}{4}
    &&\text{if} \ \ n \equiv 1 \pmod{12},\ \text{and } n \ge 25,\\
&\frac{11}{12} n^3 + \frac{35}{12} n^2 + \frac{13}{4} n - \frac{11}{4}
    &&\text{if} \ \ n \equiv 5 \pmod{12},\ \text{and } n \ge 17,\\
&\frac{11}{12} n^3 + \frac{13}{4} n^2 + \frac{21}{4} n + \frac{1}{4}
    &&\text{if} \ \ n \equiv 9 \pmod{12},\ \text{and } n \ge 9,\\
&\frac{5}{4} n^3 + 3 n^2 - \frac{3}{4} n - \frac{13}{2}
    &&\text{if} \ \ n \equiv 2 \pmod{4},\ \text{and } n \ge 6,\\
&\frac{5}{4} n^3 + \frac{25}{4} n^2 + 11 n + 3
    &&\text{if} \ \ n \equiv 3 \pmod{4},\ \text{and } n \ge 7.
\end{aligned}
\end{cases}
$$
Moreover: $F(\langle 2^2, 3^2, 4^2, 5^2 \rangle) = 23$, 
$F(\langle 3^2, 4^2, 5^2, 6^2 \rangle) = 119$,
$F(\langle 4^2, 5^2, 6^2, 7^2 \rangle) = 119$,
$F(\langle 5^2, 6^2, 7^2, 8^2 \rangle) \\ = 240$,
$F(\langle 8^2, 9^2, 10^2, 11^2 \rangle) = 659$,
$ F(\langle12^2,13^2,14^2,15^2 \rangle) = 2045$, 
$F(\langle 13^2, 14^2, 15^2, 16^2 \rangle) = 2553$.

\end{cor}

\begin{cor}\label{corgenussquares}
The genus of the numerical semigroup generated by the four consecutive squares $n^2$, $(n+1)^2$, $(n+2)^2$, $(n+3)^2$ is as follows:
$$
\begin{cases}
\begin{aligned}
&\frac{35}{72} n^3 + \frac{17}{8} n^2 + \frac{9}{4} n
    &&\text{if} \ \ n \equiv 0 \pmod{12},\ \text{and } n \ge 24,\\
&\frac{35}{72} n^3 + \frac{17}{8} n^2 + \frac{25}{12} n - \frac{4}{9}
    &&\text{if} \ \ n \equiv 4 \pmod{12},\ \text{and } n \ge 16,\\
&\frac{35}{72} n^3 + \frac{17}{8} n^2 + \frac{9}{4} n - \frac{8}{9}
    &&\text{if} \ \ n \equiv 8 \pmod{12},\ \text{and } n \ge 20,\\
&\frac{35}{72} n^3 + \frac{9}{4} n^2 + \frac{21}{8} n + \frac{23}{36}
    &&\text{if} \ \ n \equiv 1 \pmod{12},\ \text{and } n \ge 13,\\
&\frac{35}{72} n^3 + \frac{9}{4} n^2 + \frac{59}{24} n - \frac{11}{36}
    &&\text{if} \ \ n \equiv 5 \pmod{12},\ \text{and } n \ge 5,\\
&\frac{35}{72} n^3 + \frac{9}{4} n^2 + \frac{21}{8} n - \frac{1}{4}
    &&\text{if} \ \ n \equiv 9 \pmod{12},\ \text{and } n \ge 9,\\
&\frac{5}{8} n^3 + \frac{15}{8} n^2 - \frac{1}{2}
    &&\text{if} \ \ n \equiv 2 \pmod{4},\ \text{and } n \ge 2,\\
&\frac{5}{8} n^3 + \frac{15}{4} n^2 + \frac{45}{8} n + \frac{9}{2}
    &&\text{if} \ \ n \equiv 3 \pmod{4},\ \text{and } n \ge 11.
\end{aligned}
\end{cases}
    $$
Moreover:  $G(\langle 3^2,4^2,5^2,6^2 \rangle) = 60$,
$ G(\langle 4^2,5^2,6^2,7^2 \rangle) = 66$, 
$G(\langle 7^2,8^2,9^2,10^2 \rangle) = 427$, \\ 
$G(\langle 8^2,9^2,10^2,11^2 \rangle) =  389$,
$ G(\langle12^2,13^2,14^2,15^2 \rangle) = 1161$.
\end{cor}

\begin{cor}\label{corcatenarysquares}
    The catenary degree of the numerical semigroup generated by the four consecutive squares $n^2$, $(n+1)^2$, $(n+2)^2$, $(n+3)^2$ is as follows:
$$
\begin{cases}
\begin{aligned}
&\frac{11}{12}n + 2
    &&\text{if} \ \ n \equiv 0 \pmod{12},\ \text{and } n \ge 24,\\
&\frac{11}{12}n + \frac{10}{3}
    &&\text{if} \ \ n \equiv 4 \pmod{12},\ \text{and } n \ge 16,\\
&\frac{11}{12}n + \frac{8}{3}
    &&\text{if} \ \ n \equiv 8 \pmod{12},\ \text{and } n \ge 20,\\
&\frac{11}{12}n + \frac{37}{12}
    &&\text{if} \ \ n \equiv 1 \pmod{12},\ \text{and } n \ge 13,\\
&\frac{11}{12}n + \frac{29}{12}
    &&\text{if} \ \ n \equiv 5 \pmod{12},\ \text{and } n \ge 5,\\
&\frac{11}{12}n + \frac{7}{4}
    &&\text{if} \ \ n \equiv 9 \pmod{12},\ \text{and } n \ge 9,\\
&\frac{5}{4}n + \frac{3}{2}
    &&\text{if} \ \ n \equiv 2 \pmod{4},\ \text{and } n \ge 14,\\
&\frac{5}{4}n + \frac{21}{4}
    &&\text{if} \ \ n \equiv 3 \pmod{4},\ \text{and } n \ge 7.
\end{aligned}
\end{cases}
    $$
Moreover:   
$\textup{c}(\langle 2^2,3^2,4^2,5^2 \rangle) = 9$, 
$\textup{c}(\langle 3^2,4^2,5^2,6^2 \rangle) = 16$,
 $ \textup{c}(\langle 4^2,5^2,6^2,7^2 \rangle) = 9$,
 $\textup{c}(\langle 6^2,7^2,8^2,9^2 \rangle) = 11$,
 $\textup{c}(\langle 8^2,9^2,10^2,11^2 \rangle) = 11$, 
$\textup{c}(\langle 10^2,11^2,12^2,13^2 \rangle) = 15$,
 $\textup{c}(\langle12^2,13^2,14^2,15^2 \rangle) = 15$.
\end{cor}

\begin{cor}\label{corbettisquares}
    The cardinality of the minimal presentation of the numerical semigroup generated by the four consecutive squares $n^2$, $(n+1)^2$, $(n+2)^2$, $(n+3)^2$ is as follows:
$$
\begin{cases}
\begin{aligned}
&\frac{1}{3}n+6
    &&\text{if} \ \ n \equiv 0 \pmod{12},\ \text{and} \ n \ge 24,\\
&\frac{1}{3}n+\frac{20}{3}
    &&\text{if} \ \ n \equiv 4 \pmod{12},\ \text{and } n \ge 16,\\
&\frac{1}{3}n + \frac{16}{3}
    &&\text{if} \ \ n \equiv 8 \pmod{12},\ \text{and } n \ge 20,\\
&\frac{1}{3}n + \frac{14}{3}
    &&\text{if} \ \ n \equiv 1 \pmod{12},\ \text{and } n \ge 13,\\
&\frac{1}{3}n + \frac{16}{3}
    &&\text{if} \ \ n \equiv 5 \pmod{12},\ \text{and } n \ge 17,\\
&\frac{1}{3}n + 4
    &&\text{if} \ \ n \equiv 9 \pmod{12},\ \text{and } n \ge 9,\\
&7
    &&\text{if} \ \ n \equiv 2 \pmod{4},\ \text{and } n \ge 6, \\
&7
    &&\text{if} \ \ n \equiv 3 \pmod{4},\ \text{and } n \ge 15.
\end{aligned}
\end{cases}
    $$
Moreover:
$\# \rho(\langle 2^2,3^2,4^2,5^2 \rangle) = 1$,
$\# \rho(\langle 3^2,4^2,5^2,6^2 \rangle) = 1$,
$\# \rho(\langle 4^2,5^2,6^2,7^2 \rangle) = 6$, 
$\# \rho(\langle 5^2,6^2,7^2,8^2 \rangle) = 6$,
$\# \rho(\langle 7^2,8^2,9^2,10^2 \rangle) = 7$,
$\# \rho(\langle 8^2,9^2,10^2,11^2 \rangle) = 10$, 
$\# \rho(\langle 11^2,12^2,13^2,14^2 \rangle) = 6$,\\
$\# \rho(\langle 12^2,13^2,14^2,15^2 \rangle) = 11$.
\end{cor}

\section{Four consecutive triangular numbers}\label{triangularnumbers}

In this section, we will find the Frobenius number, the genus, the set of Betti elements, the catenary degree, and a minimal presentation of the numerical semigroup generated by four consecutive triangular numbers; $S = \langle \binom{n+1}{2},\binom{n+2}{2}, \binom{n+3}{2}, \binom{n+4}{2} \rangle $. We will use the method proposed in Sections \ref{geometricprocedure} and \ref{bettielements}, where we put (for now) $(d_0,d_1,d_2,d_3) = \left(\binom{n+1}{2},\binom{n+2}{2}, \binom{n+3}{2}, \binom{n+4}{2} \right)$. 

We will call $T$ the figure that will be obtained by removing regions from the initial collection of cubes in each case of $n \pmod{6}$. The cubes of $T$ will be labeled (as always), where
the cube $[[a,b,c]]$ will be labeled with $ad_1 + bd_2 + cd_3$. We will know that the set of labels of $T$ is the Ap{\'e}ry set with respect to $d_0$ after we calculate that $T$ has exactly $d_0$ cubes.

We sum up the results of this section in Corollaries \ref{corfrobtriangular}-\ref{corbettitriangular}.

\subsection{$n = 6k, k>1$}
\subsubsection{Ap{\'e}ry set}
We have the following:
\begin{align}
    (3k+1)\binom{6k+1}{2} &= 3k \binom{6k+2}{2} + 0 \binom{6k+3}{2} + 0\binom{6k+4}{2}, \label{6kminrelation0}\\
     3k \binom{6k+2}{2} &= (3k+1)\binom{6k+1}{2}+ 0 \binom{6k+3}{2} + 0\binom{6k+4}{2}, \label{6kminrelation1}\\
     3k \binom{6k+3}{2} &= (2k+1)\binom{6k+1}{2} + 0\binom{6k+2}{2} + k\binom{6k+4}{2}, \label{6kminrelation2}\\
     (2k+1)\binom{6k+4}{2} &= (2k+1)\binom{6k+1}{2} + 0\binom{6k+2}{2} + 2\binom{6k+3}{2}.\label{6kminrelation3}
\end{align}
and by Lemmas \ref{1remlemma} and \ref{lemmaminrelation1}, constructing points from the above equations and deleting their associated regions, the Ap{\'e}ry set is contained in the set of labels of the remaining figure.
Furthermore, we have the following relations that satisfy the conditions of Theorem \ref{theoremprocedure}:
\begin{align}
    2k \binom{6k+1}{2} + 3 \binom{6k+2}{2} &= \binom{6k+3}{2} + 2k\binom{6k+4}{2},\label{6keq1}\\
   \binom{6k+1}{2} + 3 \binom{6k+3}{2} &= 3\binom{6k+2}{2} + \binom{6k+4}{2},\label{6keq2}\\
   3k \binom{6k+1}{2} +  \binom{6k+4}{2} &= (3k-3)\binom{6k+2}{2} + 3\binom{6k+3}{2} .\label{6keq3}
\end{align}
Subtracting equation \eqref{6keq2} from \eqref{6keq1} we obtain additional $k-1$ relations:
\begin{equation}\label{6k_k-1equations_1}
\begin{aligned}  
 (2k-1) \binom{6k+1}{2} + 6 \binom{6k+2}{2} &= 4\binom{6k+3}{2} + (2k-1)\binom{6k+4}{2},\\
 (2k-2) \binom{6k+1}{2} + 9 \binom{6k+2}{2} &= 7\binom{6k+3}{2} + (2k-2)\binom{6k+4}{2},\\
 & \hspace{5.6pt} \vdots\\
  (k+2)\binom{6k+1}{2} + (3k-3) \binom{6k+2}{2} &= (3k-5)\binom{6k+3}{2} + (k+2)\binom{6k+4}{2},\\
  (k+1)\binom{6k+1}{2} + 3k \binom{6k+2}{2} &= (3k-2)\binom{6k+3}{2} + (k+1)\binom{6k+4}{2}.
\end{aligned}
\end{equation}
Similarly, we can subtract equation \eqref{6keq2} from \eqref{6keq3} and obtain further $k-2$ relations:
\begin{equation}\label{6k_k-2equations_2}
\begin{aligned} 
 (3k-1) \binom{6k+1}{2} +  2\binom{6k+4}{2} &= (3k-6)\binom{6k+2}{2} + 6\binom{6k+3}{2},\\
 (3k-2) \binom{6k+1}{2} +  3\binom{6k+4}{2} &= (3k-9)\binom{6k+2}{2} + 9\binom{6k+3}{2},\\
 & \hspace{5.6pt} \vdots\\
 (2k+3) \binom{6k+1}{2} +  (k-2)\binom{6k+4}{2} &= 6\binom{6k+1}{2} + (3k-6)\binom{6k+2}{2},\\
 (2k+2) \binom{6k+1}{2} +  (k-1)\binom{6k+4}{2} &= 3\binom{6k+2}{2} + (3k-3)\binom{6k+3}{2}.\\
\end{aligned}
\end{equation}

Analogously as always, one can check that the figure $T$, carved out using all the above equations, has the volume $\binom{6k+1}{2}$, which implies that it is an $L$-shape.

\begin{figure}[h]
    \caption{the figure $T$ for $n = 6k$ and $k=6$}
    \centering
    \includegraphics[width = 0.5\textwidth]{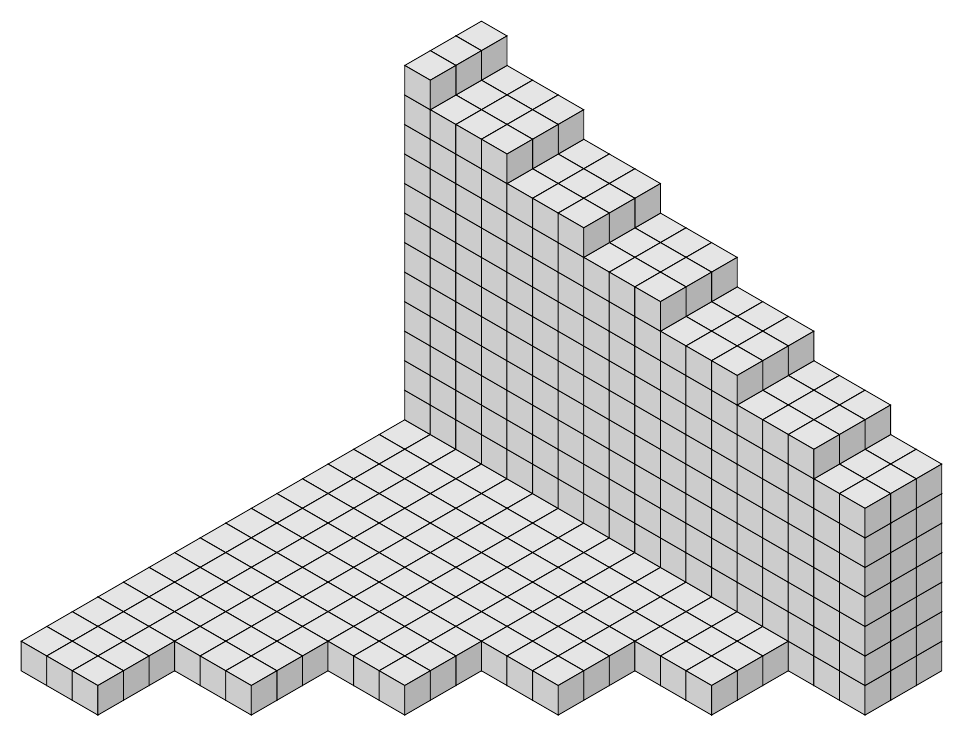}
    \label{fig6k}
\end{figure}

\subsubsection{Frobenius number and genus}

The largest label in $T$ is  
$$  2\binom{6k+2}{2} + (3k-1)\binom{6k+3}{2} + k\binom{6k+4}{2} = 72 k^3 + 84 k^2 + 18 k - 1,$$ 
which gives
$$ F(S) = 72 k^3 + 66 k^2 + 15k - 1.$$
One can also find that the sum of the labels in $T$ is 
$$ 675 k^5 + \frac{1845}{2} k^4 + 297 k^3 + 18 k^2 - \frac{3}{2} k,$$
which gives
$$ G(S)  = \frac{75}{2}k^3 + 36k^2 + \frac{15}{2}k.$$

\subsubsection{Catenary degree}
Using Proposition \ref{propminimalrelations} and Remark \ref{remark_a00_procedure}, we find that \eqref{6kminrelation0}-\eqref{6kminrelation3} are in fact the minimal relations.

By Lemmas \ref{lemma_catenary2} and \ref{lemma_catenary3}, the common values of each of equations \eqref{6kminrelation0}-\eqref{6k_k-2equations_2} are Betti elements. 
By Theorem \ref{theorembettifind}, they are all the Betti elements, and there are $2k+3$ of them (as the common values of \eqref{6kminrelation0} and \eqref{6kminrelation1} are equal).

From Lemmas \ref{lemma_catenary1} and \ref{lemma_catenarycase1}, we obtain the sets of factorizations of the common values of each of equations \eqref{6keq1}-\eqref{6k_k-2equations_2}.
One can easily check that the largest catenary degree among these values is \\ $\textup{c} \left((k+1)\binom{6k+1}{2} + 3k \binom{6k+2}{2} \right) = 4k+1$ and that the catenary degrees of the common values of equations \eqref{6kminrelation0}-\eqref{6kminrelation3} are smaller than that, which implies
$$ \textup{c}(S) = 4k+1.$$

Lastly, to obtain a minimal presentation, we apply Theorem \ref{theorem_minimalpresentation}. Similarly as always, we do not have the sets of factorizations of the common values of the minimal relations. However, we can see that $(3k)\binom{6k+3}{2}, (2k+1)\binom{6k+4}{2} \notin \langle \binom{6k+2}{2} \rangle$ and that 
$ (3k+1)\binom{6k+1}{2} = 3k \binom{6k+2}{2} \notin \langle \binom{6k+3}{2}, \binom{6k+4}{2} \rangle, $
as $\gcd \left( \binom{6k+3}{2}, \binom{6k+4}{2} \right) = 6k+3 \nmid (3k+1)\binom{6k+1}{2}$. This means that we can construct the minimal presentation as always, and it has cardinality $2k+3$.

\subsection{$n = 6k+2, k>1$}
\subsubsection{Ap{\'e}ry set}
We have the following:
\begin{align}
    (3k+2)\binom{6k+3}{2} &= (3k+1) \binom{6k+4}{2} + 0 \binom{6k+5}{2} + 0\binom{6k+6}{2}, \label{6k+2minrelation0}\\
     (3k+1) \binom{6k+4}{2} &= (3k+2)\binom{6k+3}{2}+ 0 \binom{6k+5}{2} + 0\binom{6k+6}{2}, \label{6k+2minrelation1}\\
     (3k+3) \binom{6k+5}{2} &= (2k+1)\binom{6k+3}{2} + 2\binom{6k+4}{2} + (k+1)\binom{6k+6}{2}, \label{6k+2minrelation2}\\
     (2k+1)\binom{6k+6}{2} &= (2k+1)\binom{6k+3}{2} + 2\binom{6k+4}{2} + 0\binom{6k+5}{2}.\label{6k+2minrelation3}
\end{align}
and by Lemmas \ref{1remlemma} and \ref{lemmaminrelation1}, constructing points from the above equations and deleting their associated regions, the Ap{\'e}ry set is contained in the set of labels of the remaining figure.
Furthermore, we have the following relations that satisfy the conditions of Theorem \ref{theoremprocedure}:
\begin{align}
    2k \binom{6k+3}{2} + 5 \binom{6k+4}{2} &= 3\binom{6k+5}{2} + 2k\binom{6k+6}{2},\label{6k+2eq1}\\
   \binom{6k+3}{2} + 3 \binom{6k+5}{2} &= 3\binom{6k+4}{2} + \binom{6k+6}{2},\label{6k+2eq2}\\
   (3k+1)\binom{6k+3}{2} +  \binom{6k+6}{2} &= (3k-2)\binom{6k+4}{2} + 3\binom{6k+5}{2} .\label{6k+2eq3}
\end{align}
Subtracting equation \eqref{6k+2eq2} from \eqref{6k+2eq1}, we obtain additional $k-1$ relations:
\begin{equation}\label{6k+2_k-1equations_1}
\begin{aligned}  
 (2k-1) \binom{6k+3}{2} + 8 \binom{6k+4}{2} &= 6\binom{6k+5}{2} + (2k-1)\binom{6k+6}{2},\\
 (2k-2) \binom{6k+3}{2} + 11 \binom{6k+4}{2} &= 9\binom{6k+5}{2} + (2k-2)\binom{6k+6}{2},\\
 & \hspace{5.6pt} \vdots\\
 (k+2) \binom{6k+3}{2} + (3k-1) \binom{6k+4}{2} &=  (3k-3)\binom{6k+5}{2} + (k+2)\binom{6k+6}{2},\\
 (k+1) \binom{6k+3}{2} + (3k+2) \binom{6k+4}{2} &=  (3k)\binom{6k+5}{2} + (k+1)\binom{6k+6}{2}.
\end{aligned}
\end{equation}
Similarly, we can subtract equation \eqref{6k+2eq2} from \eqref{6k+2eq3} and obtain further $k-1$ relations:
\begin{equation}\label{6k+2_k-1equations_2}
\begin{aligned} 
 (3k)\binom{6k+3}{2} +  2\binom{6k+6}{2} &= (3k-5)\binom{6k+4}{2} + 6\binom{6k+5}{2},\\
 (3k-1)\binom{6k+3}{2} +  3\binom{6k+6}{2} &= (3k-8)\binom{6k+4}{2} + 9\binom{6k+5}{2},\\
 & \hspace{5.6pt} \vdots\\
 (2k+3)\binom{6k+3}{2} +  (k-1)\binom{6k+6}{2} &= 4\binom{6k+4}{2} + (3k-3)\binom{6k+5}{2},\\
 (2k+2)\binom{6k+3}{2} +  (k)\binom{6k+6}{2} &= 1\binom{6k+4}{2} + (3k)\binom{6k+5}{2}.
\end{aligned}
\end{equation}

Analogously as always, one can check that the figure $T$, carved out using all the above equations, has the volume $\binom{6k+3}{2}$, which implies that it is an $L$-shape.

\begin{figure}[h]
    \caption{the figure $T$ for $n = 6k+2$ and $k=6$}
    \centering
    \includegraphics[width = 0.5\textwidth]{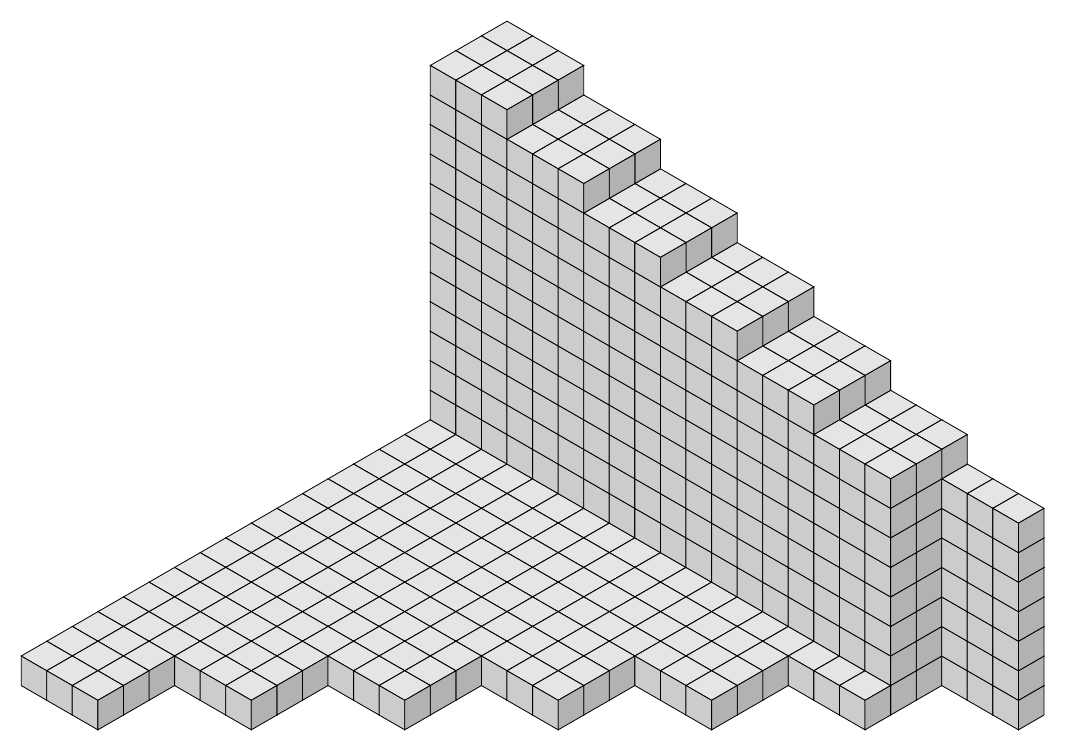}
    \label{fig6k+2}
\end{figure}
\subsubsection{Frobenius number and genus}
The largest label in $T$ is  
$$   (3k+2)\binom{6k+5}{2} + k\binom{6k+6}{2} = 72 k^3 + 150 k^2 + 99 k + 20,$$ 
which gives
$$ F(S) =72 k^3 + 132 k^2 + 84 k + 17.$$
One can also find that the sum of the labels in $T$ is 
$$ 675 k^5 + \frac{4095}{2} k^4 + 2295 k^3 + 1215 k^2 + \frac{615}{2}k + 30 ,$$
which gives
$$ G(S)  = \frac{75}{2} k^3 + \frac{147}{2} k^2 + 45 k + 9.$$

\subsubsection{Catenary degree}
Using Proposition \ref{propminimalrelations} and Remark \ref{remark_a00_procedure}, we find that \eqref{6k+2minrelation0}-\eqref{6k+2minrelation3} are in fact the minimal relations.

By Lemmas \ref{lemma_catenary2} and \ref{lemma_catenary3}, the common values of each of equations \eqref{6k+2minrelation0}-\eqref{6k+2_k-1equations_2} are Betti elements. 
By Theorem \ref{theorembettifind}, they are all the Betti elements, and there are $2k+4$ of them (as the common values of \eqref{6k+2minrelation0} and \eqref{6k+2minrelation1} are equal).

From Lemmas \ref{lemma_catenary1} and \ref{lemma_catenarycase1}, we obtain the sets of factorizations of the common values of each of equations \eqref{6k+2eq1}-\eqref{6k+2_k-1equations_2}.
One can check that the largest catenary degree among these values is $\textup{c} \left((k+1) \binom{6k+3}{2} + (3k+2) \binom{6k+4}{2} \right) = 4k+3$ and that the catenary degrees of the common values of equations \eqref{6k+2minrelation0}-\eqref{6k+2minrelation3} is smaller than that, which gives
$$ \textup{c}(S) = 4k+3.$$

Lastly, we apply Theorem \ref{theorem_minimalpresentation} and construct a minimal presentation, analogously as before. The cardinality of the minimal presentation is $2k+4$.

\subsection{$n = 6k+4, k>1$}
\subsubsection{Ap{\'e}ry set}
We have the following:
\begin{align}
    (3k+3)\binom{6k+5}{2} &= (3k+2)\binom{6k+6}{2} + 0\binom{6k+7}{2} + 0\binom{6k+8}{2}, \label{6k+4minrelation0} \\
    (3k+2)\binom{6k+6}{2} &= (3k+3)\binom{6k+5}{2} + 0\binom{6k+7}{2} + 0\binom{6k+8}{2}, \label{6k+4minrelation1} \\
    (3k+3)\binom{6k+7}{2} &= (2k+2)\binom{6k+5}{2} + 1\binom{6k+6}{2} + (k+1)\binom{6k+8}{2}, \label{6k+4minrelation2} \\
    (2k+2)\binom{6k+8}{2} &= (2k+2)\binom{6k+5}{2} + 1\binom{6k+6}{2} + 1\binom{6k+7}{2}, \label{6k+4minrelation3} 
\end{align}
and by Lemmas \ref{1remlemma} and \ref{lemmaminrelation1}, constructing points from the above equations and deleting their associated regions, the Ap{\'e}ry set is contained in the set of labels of the remaining figure.
Furthermore, we have the following relations that satisfy the conditions of Theorem \ref{theoremprocedure}:
\begin{align}
    (2k+1) \binom{6k+5}{2} + 4 \binom{6k+6}{2} &= 2\binom{6k+7}{2} + (2k+1)\binom{6k+8}{2}, \label{6k+4eq1}\\
   \binom{6k+5}{2} + 3 \binom{6k+7}{2} &= 3\binom{6k+6}{2} + \binom{6k+8}{2}, \label{6k+4eq2}\\
   (3k+2)\binom{6k+5}{2} +  \binom{6k+8}{2} &= (3k-1)\binom{6k+6}{2} + 3\binom{6k+7}{2} . \label{6k+4eq3}
\end{align}
Subtracting equation \eqref{6k+4eq2} from \eqref{6k+4eq1}, we obtain additional $k$ relations:
\begin{equation}\label{6k+4_kequations_1}
\begin{aligned}    
(2k) \binom{6k+5}{2} + 7 \binom{6k+6}{2} &= 5\binom{6k+7}{2} + (2k)\binom{6k+8}{2},\\
(2k-1) \binom{6k+5}{2} + 10 \binom{6k+6}{2} &= 8\binom{6k+7}{2} + (2k-1)\binom{6k+8}{2},\\
 & \hspace{5.6pt}  \vdots      \\
(k+2) \binom{6k+5}{2} + (3k+1) \binom{6k+6}{2} &= (3k-1)\binom{6k+7}{2} + (k+2)\binom{6k+8}{2}, \\
(k+1) \binom{6k+5}{2} + (3k+4) \binom{6k+6}{2} &= (3k+2)\binom{6k+7}{2} + (k+1)\binom{6k+8}{2}.
\end{aligned}
\end{equation}
Similarly, we can subtract equation \eqref{6k+4eq2} from \eqref{6k+4eq3} and obtain further $k-1$ relations:
\begin{equation}\label{6k+4_kequations_2}
\begin{aligned}    
(3k+1)\binom{6k+5}{2} +  2\binom{6k+8}{2} &= (3k-4)\binom{6k+6}{2} + 6\binom{6k+7}{2},\\
(3k)\binom{6k+5}{2} +  3\binom{6k+8}{2} &= (3k-7)\binom{6k+6}{2} + 9
\binom{6k+7}{2},\\
 & \hspace{5.6pt}  \vdots      \\
(2k+4)\binom{6k+5}{2} +  (k-1)\binom{6k+8}{2} &= 5\binom{6k+6}{2} + (3k-3)\binom{6k+7}{2}, \\
(2k+3)\binom{6k+5}{2} +  (k)\binom{6k+8}{2} &= 2\binom{6k+6}{2} + (3k)\binom{6k+7}{2}.
\end{aligned}
\end{equation}

Analogously as always, one can check that the figure $T$, carved out using all the above equations, has the volume $\binom{6k+5}{2}$, which implies that it is an $L$-shape.

\begin{figure}[h]
    \caption{the figure $T$ for $n = 6k+4$ and $k=6$}
    \centering
    \includegraphics[width = 0.5\textwidth]{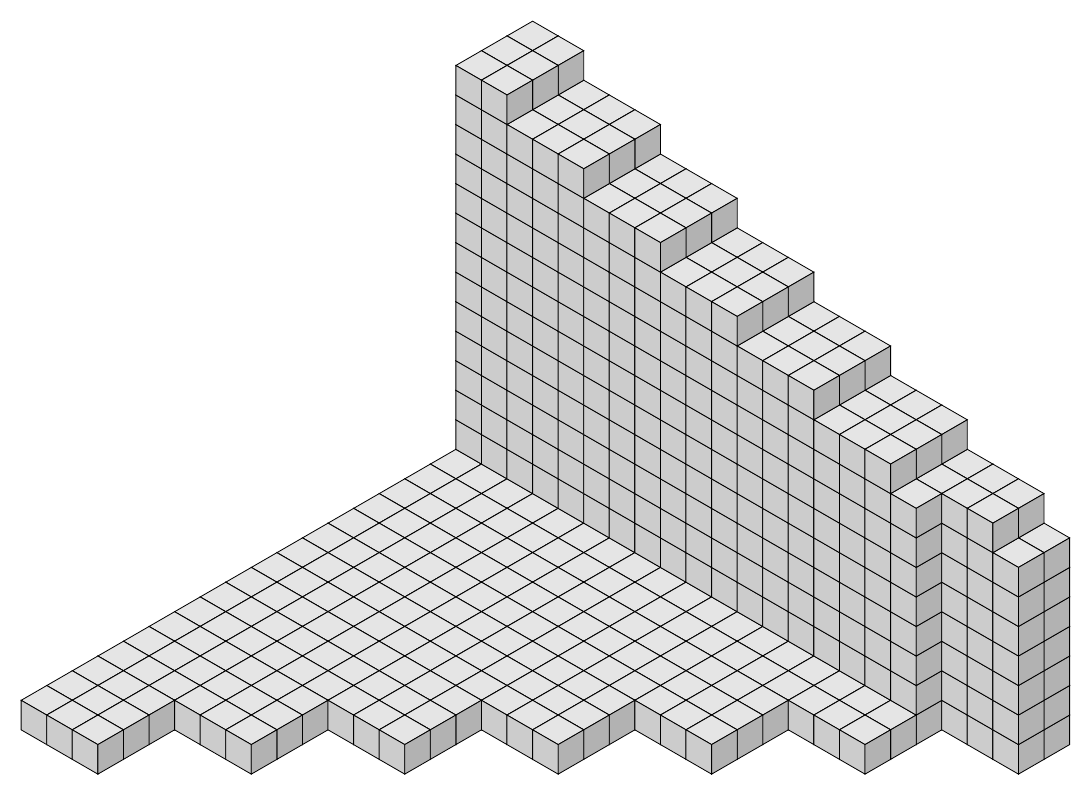}
    \label{fig6k+4}
\end{figure}

\subsubsection{Frobenius number and genus}
The largest label in $T$ is  
$$ \binom{6k+6}{2} +  (3k+1)\binom{6k+7}{2} + (k+1)\binom{6k+8}{2} =  72 k^3 + 216 k^2 + 208 k + 64,$$ 
which gives
$$ F(S) = 72 k^3 + 198 k^2 + 181 k + 54.$$
One can also find that the sum of the labels in $T$ is 
$$ 675 k^5 + \frac{6345}{2} k^4 + 5739 k^3 + 5025 k^2 + \frac{4277}{2}k + 355 ,$$
which gives
$$ G(S)  = \frac{75}{2} k^3 + 111 k^2 + \frac{209}{2} k + 31.$$

\subsubsection{Catenary degree}
Using Proposition \ref{propminimalrelations} and Remark \ref{remark_a00_procedure}, we find that \eqref{6k+4minrelation0}-\eqref{6k+4minrelation3} are in fact the minimal relations.

By Lemmas \ref{lemma_catenary2} and \ref{lemma_catenary3}, the common values of each of equations \eqref{6k+4minrelation0}-\eqref{6k+4_kequations_2} are Betti elements. 
By Theorem \ref{theorembettifind}, they are all the Betti elements, and there are $2k+5$ of them (as the common values of \eqref{6k+4minrelation0} and \eqref{6k+4minrelation1} are equal).

From Lemmas \ref{lemma_catenary1} and \ref{lemma_catenarycase1}, we obtain the sets of factorizations of the common values of each of equations \eqref{6k+4eq1}-\eqref{6k+4_kequations_2}.
One can easily check that the largest catenary degree among these values is $\textup{c} \left((k+1) \binom{6k+5}{2} + (3k+4) \binom{6k+6}{2} \right) = 4k+5$ and that the catenary degrees of the common values of equations \eqref{6k+4minrelation0}-\eqref{6k+4minrelation3} is smaller than that, which gives
$$ \textup{c}(S) = 4k+5.$$

Lastly, we apply Theorem \ref{theorem_minimalpresentation} and construct a minimal presentation, analogously as before. The cardinality of the minimal presentation is $2k+5$.

\subsection{$n = 6k+1, k>1$}
For the rest of this section, we put $(d_0,d_1,d_2,d_3) = \left(\binom{n+2}{2},\binom{n+1}{2},\binom{n+3}{2},\binom{n+4}{2} \right)$, and we will look for the Ap{\'e}ry set with respect to $\binom{n+2}{2}$ instead.
\subsubsection{Ap{\'e}ry set}
We have the following:
\begin{align}
    (6k+3)\binom{6k+2}{2} &= (3k-1)\binom{6k+3}{2} + (3k+1)\binom{6k+4}{2} + 0\binom{6k+5}{2}, \label{6k+1minrelation0} \\
    (3k+2)\binom{6k+3}{2} &= 0\binom{6k+2}{2} + (3k+1)\binom{6k+4}{2} + 0\binom{6k+5}{2}, \label{6k+1minrelation1} \\
    (3k+1)\binom{6k+4}{2} &= 0\binom{6k+2}{2} + (3k+2)\binom{6k+3}{2} + 0\binom{6k+5}{2}, \label{6k+1minrelation2} \\
    (2k+1)\binom{6k+5}{2} &= (2k+1)\binom{6k+2}{2} + 1\binom{6k+3}{2} + 1\binom{6k+4}{2}, \label{6k+1minrelation3} 
\end{align}
and by Lemmas \ref{1remlemma} and \ref{lemmaminrelation1}, constructing points from the above equations and deleting their associated regions, the Ap{\'e}ry set is contained in the set of labels of the remaining figure.
Furthermore, we have the following relations that satisfy the conditions of Theorem \ref{theoremprocedure}:
\begin{align}
      2 \binom{6k+3}{2} + k \binom{6k+2}{2} &= \binom{6k+4}{2} + k\binom{6k+5}{2}, \label{6k+1eq1}\\
  (6k-2)\binom{6k+3}{2} + 3\binom{6k+4}{2}  &= (6k+2)\binom{6k+2}{2} +  \binom{6k+5}{2} , \label{6k+1eq2}\\
     3\binom{6k+3}{2} + \binom{6k+5}{2} &= \binom{6k+2}{2} + 3 \binom{6k+4}{2}, \label{6k+1eq3} \\
      \binom{6k+3}{2} + (k+1)\binom{6k+5}{2} &= (k+1)\binom{6k+2}{2} + 2 \binom{6k+4}{2}. \label{6k+1eq31}
\end{align}
Adding equation \eqref{6k+1eq3} to \eqref{6k+1eq1}, we obtain additional $k-1$ relations:
\begin{equation}\label{6k+1_k-1equations_1}
\begin{aligned}    
 5 \binom{6k+3}{2} + (k-1) \binom{6k+2}{2} &= 4\binom{6k+4}{2} + (k-1)\binom{6k+5}{2},\\
8 \binom{6k+3}{2} + (k-2) \binom{6k+2}{2} &= 7\binom{6k+4}{2} + (k-2)\binom{6k+5}{2},\\
 & \hspace{5.6pt}  \vdots      \\
(3k-4) \binom{6k+3}{2} + 2 \binom{6k+2}{2} &= (3k-5)\binom{6k+4}{2} + 2\binom{6k+5}{2}, \\
(3k-1) \binom{6k+3}{2} + 1 \binom{6k+2}{2} &= (3k-2)\binom{6k+4}{2} + 1\binom{6k+5}{2}.
\end{aligned}
\end{equation}
Subtracting equation \eqref{6k+1eq3} from \eqref{6k+1eq2}, we obtain further $2k-1$ relations:
\begin{equation}\label{6k+1_2k-1equations_2}
\begin{aligned}    
(6k-5)\binom{6k+3}{2} + 6\binom{6k+4}{2}  &= (6k+1)\binom{6k+2}{2} +  2\binom{6k+5}{2},\\
(6k-8)\binom{6k+3}{2} + 9\binom{6k+4}{2}  &= (6k)\binom{6k+2}{2} +  3\binom{6k+5}{2},\\
 & \hspace{5.6pt}  \vdots      \\
4\binom{6k+3}{2} + (6k-3)\binom{6k+4}{2}  &= (4k+4)\binom{6k+2}{2} +  (2k-1)\binom{6k+5}{2}, \\
1\binom{6k+3}{2} + (6k)\binom{6k+4}{2}  &= (4k+3)\binom{6k+2}{2} +  (2k)\binom{6k+5}{2}.
\end{aligned}
\end{equation}

One can check that the figure $T$, carved out using all the above equations, has the volume $\binom{6k+3}{2}$, which implies that it is an $L$-shape.

\begin{figure}[h]
    \caption{the figure $T$ for $n = 6k+1$ and $k=5$}
    \centering
    \includegraphics[width = 0.6\textwidth]{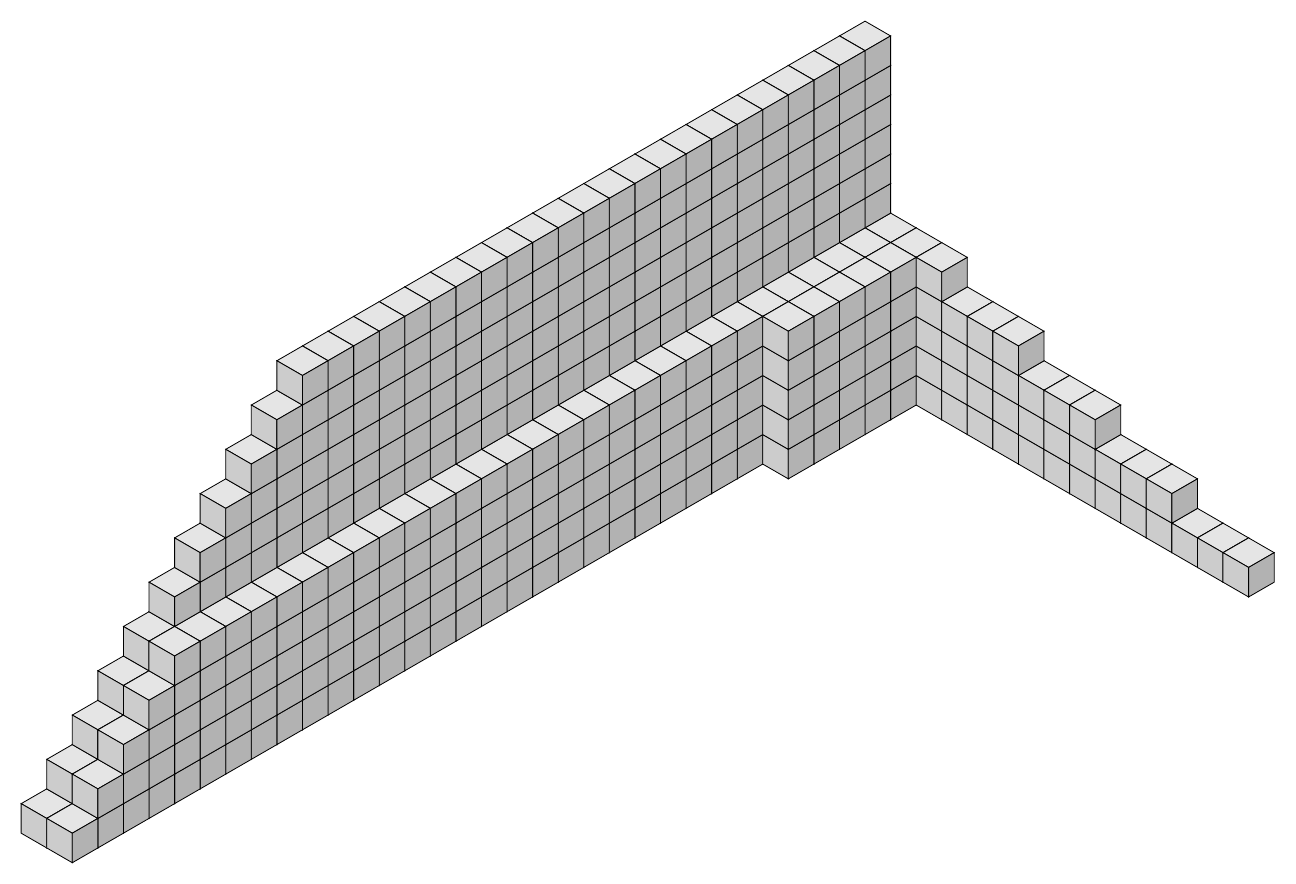}
    \label{fig6k+1}
\end{figure}
\subsubsection{Frobenius number and genus}
The largest label in $T$ is  
$$ (4k+2)\binom{6k+2}{2}  + 2k\binom{6k+5}{2} = 108 k^3 + 126 k^2 + 42 k + 2 ,$$ 
which gives
$$ F(S) = 108 k^3 + 108 k^2 + 27 k - 1.$$
One can also find that the sum of the labels in $T$ is 
$$ 999 k^5 + \frac{4149}{2} k^4 + \frac{3231}{2} k^3 + 570 k^2 + 84 k + 3  ,$$
which gives
$$ G(S)  = \frac{111}{2}k^3 + 60k^2 + \frac{31}{2}k.$$

\subsubsection{Catenary degree}

Using Proposition \ref{propminimalrelations} and Remark \ref{remark_a00_procedure}, we find that \eqref{6k+1minrelation0}-\eqref{6k+1minrelation3} are in fact the minimal relations.

By Lemmas \ref{lemma_catenary2}, \ref{lemma_catenary1}, and \ref{lemma_catenarycase1}, the common values of each of equations \eqref{6k+1minrelation0}-\eqref{6k+1_2k-1equations_2} are Betti elements. 
By Theorem \ref{theorembettifind}, they are all the Betti elements, and there are $3k+5$ of them (as the common values of \eqref{6k+1minrelation1} and \eqref{6k+1minrelation2} are equal). 

From Lemmas \ref{lemma_catenary1} and \ref{lemma_catenarycase1}, we obtain the sets of factorizations of the common values of each of equations \eqref{6k+1eq1}-\eqref{6k+1_2k-1equations_2}.
One can easily check that the largest catenary degree among these values is $6k+3$, attained for \eqref{6k+1minrelation0}, \eqref{6k+1eq2}, and all of \eqref{6k+1_2k-1equations_2}, hence
$$ \textup{c}(S) = 6k+3.$$

Lastly, to obtain a minimal presentation, we apply Theorem \ref{theorem_minimalpresentation}. Similarly as always, we do not have the sets of factorizations of the common values of the minimal relations. However, we can see that $(6k+3)\binom{6k+2}{2} \notin \langle \binom{6k+5}{2} \rangle$ and that 
$ (3k+2)\binom{6k+3}{2} = (3k+1) \binom{6k+4}{2} \notin \langle \binom{6k+2}{2}, \binom{6k+5}{2} \rangle, $
as 
$$(3k+2)\binom{6k+3}{2} = (9k^2 + 15k + 6)\binom{6k+2}{2} - (9k^2 + 3k)\binom{6k+5}{2},$$ 
where $(9k^2 + 15k + 6)\binom{6k+2}{2} \in \textup{Ap}(\langle \binom{6k+2}{2}, \binom{6k+5}{2} \rangle,\binom{6k+5}{2})$ (because $9k^2 + 15k + 6 < \binom{6k+5}{2}$) . This means that we can construct the minimal presentation as always, and it has cardinality $3k+5$.

\subsection{$n = 6k+3, k>1$}
\subsubsection{Ap{\'e}ry set}
We have the following:
\begin{align}
    (6k+5)\binom{6k+4}{2} &= 3k\binom{6k+5}{2} + (3k+2)\binom{6k+6}{2} + 0\binom{6k+7}{2}, \label{6k+3minrelation0} \\
    (3k+3)\binom{6k+5}{2} &= 0\binom{6k+4}{2} + (3k+2)\binom{6k+6}{2} + 0\binom{6k+7}{2}, \label{6k+3minrelation1} \\
    (3k+2)\binom{6k+6}{2} &= 0\binom{6k+4}{2} + (3k+3)\binom{6k+5}{2} + 0\binom{6k+7}{2}, \label{6k+3minrelation2} \\
    (k+1)\binom{6k+7}{2} &= (k+1)\binom{6k+4}{2} + 0\binom{6k+5}{2} + 1\binom{6k+6}{2}, \label{6k+3minrelation3} 
\end{align}
and by Lemmas \ref{1remlemma} and \ref{lemmaminrelation1}, constructing points from the above equations and deleting their associated regions, the Ap{\'e}ry set is contained in the set of labels of the remaining figure.
Furthermore, we have the following relations that satisfy the conditions of Theorem \ref{theoremprocedure}:
\begin{align}
      3 \binom{6k+5}{2} + k \binom{6k+4}{2} &= 2\binom{6k+6}{2} + k\binom{6k+7}{2}, \label{6k+3eq1}\\
  6k\binom{6k+5}{2} + 3\binom{6k+6}{2}  &= (6k+4)\binom{6k+4}{2} +  \binom{6k+7}{2} , \label{6k+3eq2}\\
     3\binom{6k+5}{2} + \binom{6k+7}{2} &= \binom{6k+4}{2} + 3 \binom{6k+6}{2}. \label{6k+3eq3} 
\end{align}
Adding equation \eqref{6k+3eq3} to \eqref{6k+3eq1}, we obtain additional $k-1$ relations:
\begin{equation}\label{6k+3_k-1equations_1}
\begin{aligned}    
6 \binom{6k+5}{2} + (k-1) \binom{6k+4}{2} &= 5\binom{6k+6}{2} + (k-1)\binom{6k+7}{2},\\
9 \binom{6k+5}{2} + (k-2) \binom{6k+4}{2} &= 8\binom{6k+6}{2} + (k-2)\binom{6k+7}{2},\\
 & \hspace{5.6pt}  \vdots      \\
(3k-3) \binom{6k+5}{2} + 2 \binom{6k+4}{2} &= (3k-4)\binom{6k+6}{2} + 2\binom{6k+7}{2}, \\
(3k) \binom{6k+5}{2} + 1 \binom{6k+4}{2} &= (3k-1)\binom{6k+6}{2} + 1\binom{6k+7}{2}.
\end{aligned}
\end{equation}
Subtracting equation \eqref{6k+3eq3} from \eqref{6k+3eq2}, we obtain further $k-1$ relations:
\begin{equation}\label{6k+3_k-1equations_2}
\begin{aligned}    
(6k-3)\binom{6k+5}{2} + 6\binom{6k+6}{2}  &= (6k+3)\binom{6k+4}{2} +  2\binom{6k+7}{2},\\
(6k-6)\binom{6k+5}{2} + 9\binom{6k+6}{2}  &= (6k+2)\binom{6k+4}{2} +  3\binom{6k+7}{2},\\
 & \hspace{5.6pt}  \vdots      \\
(3k+6)\binom{6k+5}{2} + (3k-3)\binom{6k+6}{2}  &= (5k+6)\binom{6k+4}{2} +  (k-1)\binom{6k+7}{2}, \\
(3k+3)\binom{6k+5}{2} + 3k\binom{6k+6}{2}  &= (5k+5)\binom{6k+4}{2} +  k\binom{6k+7}{2}.
\end{aligned}
\end{equation}

One can check that the figure $T$, carved out using all the above equations, has the volume $\binom{6k+5}{2}$, which implies that it is an $L$-shape.

\begin{figure}[h]
    \caption{the figure $T$ for $n = 6k+3$ and $k=5$}
    \centering
     \includegraphics[width = 0.6\textwidth]{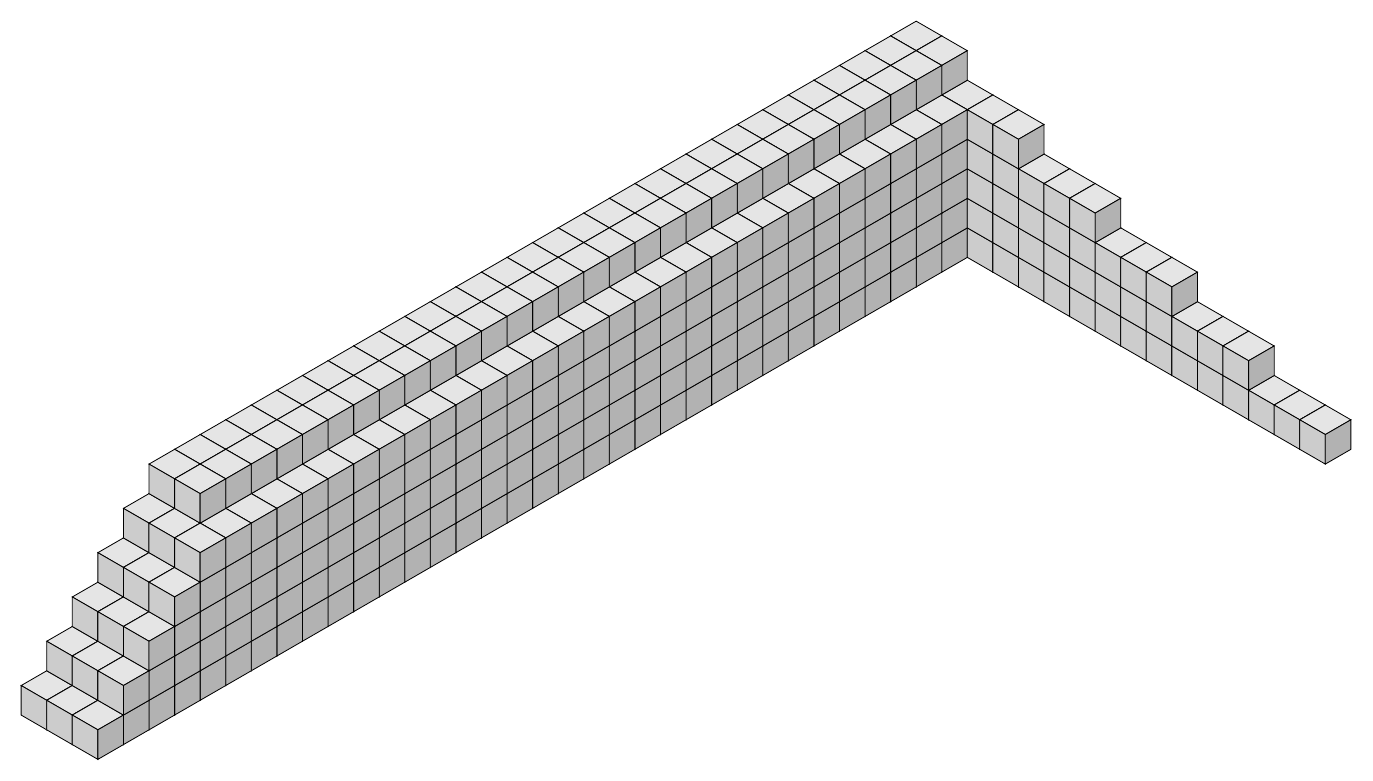}
    \label{fig6k+3}
\end{figure}
\subsubsection{Frobenius number and genus}
The largest label in $T$ is  
$$ (5k+5)\binom{6k+4}{2}  + 2\binom{6k+6}{2} + (k-1)\binom{6k+7}{2} = 108 k^3 + 252 k^2 + 183 k + 39 ,$$ 
which gives
$$ F(S) = 108 k^3 + 234 k^2 + 156 k + 29.$$
One can also find that the sum of the labels in $T$ is 
$$ 999k^5 + \frac{7641}{2}k^4 + 5712k^3 + 4152k^2 + \frac{2913}{2}k + 195 ,$$
which gives
$$ G(S)  = \frac{111}{2}k^3 + 120k^2 + \frac{159}{2}k + 15 .$$

\subsubsection{Catenary degree}

Using Proposition \ref{propminimalrelations} and Remark \ref{remark_a00_procedure}, equations \eqref{6k+3minrelation0}-\eqref{6k+3minrelation3} are in fact the minimal relations.

By Lemmas \ref{lemma_catenary2}, \ref{lemma_catenary1}, and \ref{lemma_catenarycase1}, the common values of each of we find that \eqref{6k+3minrelation0}-\eqref{6k+3_k-1equations_2} are Betti elements. 
By Theorem \ref{theorembettifind}, they are all the Betti elements, and there are $2k+4$ of them (as the common values of \eqref{6k+3minrelation1} and \eqref{6k+3minrelation2} are equal). 

From Lemmas \ref{lemma_catenary1} and \ref{lemma_catenarycase1}, we obtain the sets of factorizations of the common values of each of equations \eqref{6k+1eq1}-\eqref{6k+3_k-1equations_2}.
One can easily check that the largest catenary degree among these values is $6k+5$, attained for \eqref{6k+3minrelation0}, \eqref{6k+3eq2}, and \eqref{6k+3_k-1equations_2}, hence
$$ \textup{c}(S) = 6k+5.$$

Lastly, to obtain a minimal presentation, we apply Theorem \ref{theorem_minimalpresentation}. Similarly as always, we do not have the sets of factorizations of the common values of the minimal relations. However, we can see that $(6k+5)\binom{6k+4}{2} \notin \langle \binom{6k+7}{2} \rangle$, $(k+1)\binom{6k+7}{2} \notin \langle \binom{6k+5}{2} \rangle$, and that 
$ (3k+3)\binom{6k+5}{2} = (3k+2) \binom{6k+6}{2} \notin \langle \binom{6k+4}{2}, \binom{6k+7}{2} \rangle, $
as 
$$(3k+3)\binom{6k+5}{2} = (3k^2 + 8k + 5)\binom{6k+4}{2} - (3k^2+2k)\binom{6k+7}{2},$$ 
where $(3k^2 + 8k + 5)\binom{6k+4}{2} \in \textup{Ap}(\langle \binom{6k+4}{2}, \binom{6k+7}{2} \rangle,\binom{6k+7}{2})$ (because $3k^2 + 8k + 5 < \binom{6k+7}{2}$) . This means that we can construct the minimal presentation as always, and it has cardinality $2k+4$.

\subsection{$n = 6k+5, k>1$}
\subsubsection{Ap{\'e}ry set}
We have the following:
\begin{align}
    (6k+7)\binom{6k+6}{2} &= (3k+1)\binom{6k+7}{2} + (3k+3)\binom{6k+8}{2} + 0\binom{6k+9}{2}, \label{6k+5minrelation0} \\
    (3k+4)\binom{6k+7}{2} &= 0\binom{6k+6}{2} + (3k+3)\binom{6k+8}{2} + 0\binom{6k+9}{2}, \label{6k+5minrelation1} \\
    (3k+3)\binom{6k+8}{2} &= 0\binom{6k+6}{2} + (3k+4)\binom{6k+7}{2} + 0\binom{6k+9}{2}, \label{6k+5minrelation2} \\
    (k+1)\binom{6k+9}{2} &= (k+1)\binom{6k+6}{2} + 1\binom{6k+7}{2} + 0\binom{6k+8}{2}, \label{6k+5minrelation3} 
\end{align}
and by Lemmas \ref{1remlemma} and \ref{lemmaminrelation1}, constructing points from the above equations and deleting their associated regions, the Ap{\'e}ry set is contained in the set of labels of the remaining figure.
Furthermore, we have the following relations that satisfy the conditions of Theorem \ref{theoremprocedure}:
\begin{align}
      4 \binom{6k+7}{2} + k \binom{6k+6}{2} &= 3\binom{6k+8}{2} + k\binom{6k+9}{2}, \label{6k+5eq1}\\
  (6k+2)\binom{6k+7}{2} + 3\binom{6k+8}{2}  &= (6k+6)\binom{6k+6}{2} +  \binom{6k+9}{2}, \label{6k+5eq2}\\
     3\binom{6k+7}{2} + \binom{6k+9}{2} &= \binom{6k+6}{2} + 3 \binom{6k+8}{2}. \label{6k+5eq3} 
\end{align}
Adding equation \eqref{6k+5eq3} to \eqref{6k+5eq1}, we obtain additional $k-1$ relations:
\begin{equation}\label{6k+5_k-1equations_1}
\begin{aligned}    
 7 \binom{6k+7}{2} + (k-1) \binom{6k+6}{2} &= 6\binom{6k+8}{2} + (k-1)\binom{6k+9}{2},\\
10 \binom{6k+7}{2} + (k-2) \binom{6k+6}{2} &= 9\binom{6k+8}{2} + (k-2)\binom{6k+9}{2},\\
 & \hspace{5.6pt}  \vdots      \\
(3k-2) \binom{6k+7}{2} + 2 \binom{6k+6}{2} &= (3k-3)\binom{6k+8}{2} + 2\binom{6k+9}{2}, \\
(3k+1) \binom{6k+7}{2} + 1 \binom{6k+6}{2} &= (3k)\binom{6k+8}{2} + 1\binom{6k+9}{2}.
\end{aligned}
\end{equation}
Subtracting equation \eqref{6k+5eq3} from \eqref{6k+5eq2}, we obtain further $k-1$ relations:
\begin{equation}\label{6k+5_k-1equations_2}
\begin{aligned}    
(6k-1)\binom{6k+7}{2} + 6\binom{6k+8}{2}  &= (6k+5)\binom{6k+6}{2} +  2\binom{6k+9}{2},\\
(6k-4)\binom{6k+7}{2} + 9\binom{6k+8}{2}  &= (6k+4)\binom{6k+6}{2} +  3\binom{6k+9}{2},\\
 & \hspace{5.6pt}  \vdots      \\
(3k+8)\binom{6k+7}{2} + (3k-3)\binom{6k+8}{2}  &= (5k+8)\binom{6k+6}{2} +  (k-1)\binom{6k+9}{2}, \\
(3k+5)\binom{6k+7}{2} + (3k)\binom{6k+8}{2}  &= (5k+7)\binom{6k+6}{2} +  (k)\binom{6k+9}{2}.
\end{aligned}
\end{equation}

One can check that the figure $T$, carved out using all the above equations, has the volume $\binom{6k+7}{2}$, which implies that it is an $L$-shape.

\begin{figure}[h]
    \caption{the figure $T$ for $n = 6k+5$ and $k=5$}
    \centering
   \includegraphics[width = 0.6\textwidth]{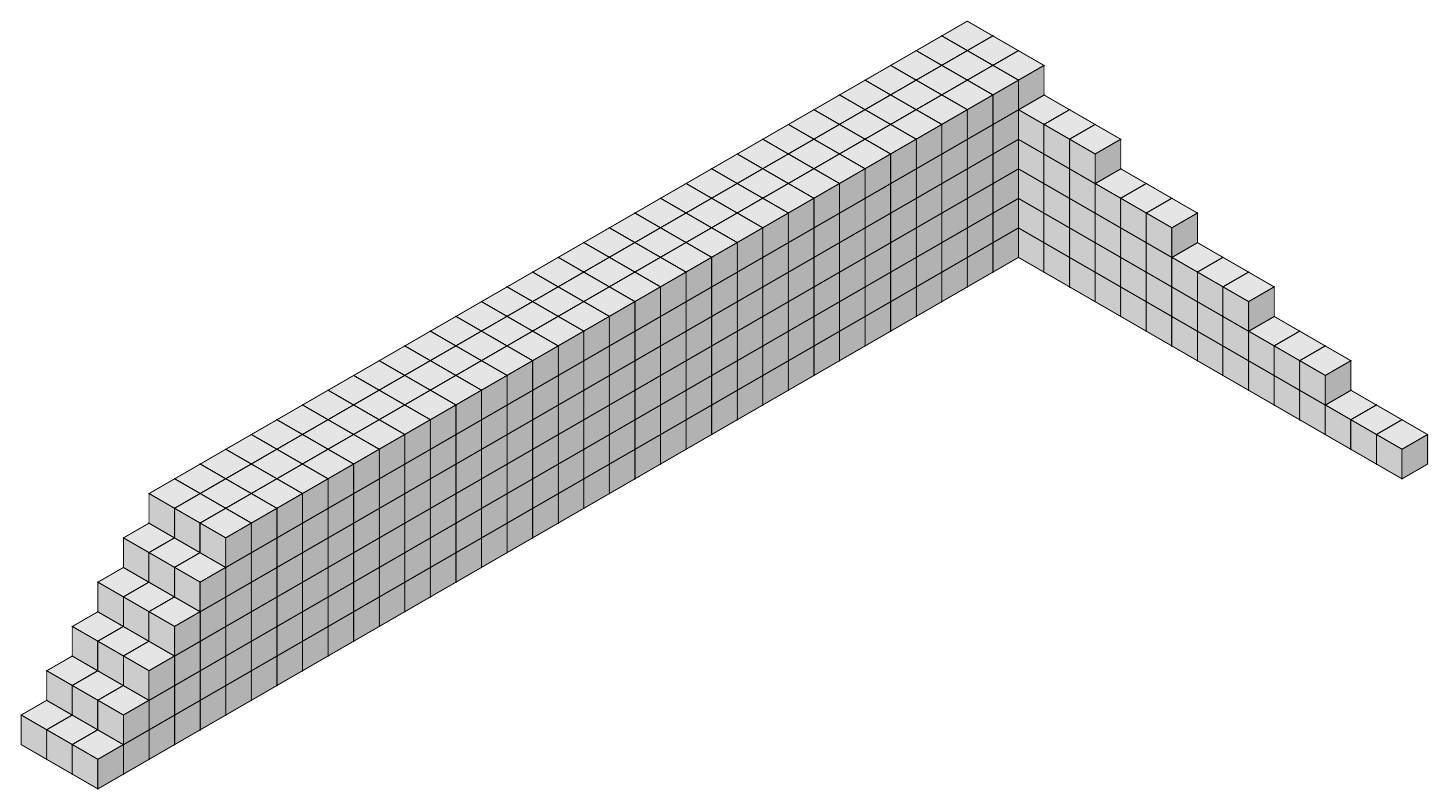}
    \label{fig6k+5}
\end{figure}
\subsubsection{Frobenius number and genus}
One can check that the largest label in $T$ is  
$$ (5k+6)\binom{6k+6}{2}  + 2\binom{6k+8}{2} + k\binom{6k+9}{2} =108 k^3 + 360 k^2 + 399 k + 146  ,$$ 
which gives
$$ F(S) = 108 k^3 + 342 k^2 + 360 k + 125.$$
One can also find that the sum of the labels in $T$ is 
$$ 999 k^5 + \frac{10971}{2}k^4 + 12006k^3 + 13086k^2 + \frac{14199}{2}k + 1533  ,$$
which gives
$$ G(S)  = \frac{111}{2}k^3 + \frac{351}{2}k^2 + 183k + 63.$$

\subsubsection{Catenary degree}

Using Proposition \ref{propminimalrelations} and Remark \ref{remark_a00_procedure}, we find that \eqref{6k+5minrelation0}-\eqref{6k+5minrelation3} are in fact the minimal relations.

By Lemmas \ref{lemma_catenary2}, \ref{lemma_catenary1}, and \ref{lemma_catenarycase1}, the common values of each of equations \eqref{6k+5minrelation0}-\eqref{6k+5_k-1equations_2} are Betti elements. 
By Theorem \ref{theorembettifind}, they are all the Betti elements, and there are $2k+4$ of them (as the common values of \eqref{6k+5minrelation1} and \eqref{6k+5minrelation2} are equal). 

From Lemmas \ref{lemma_catenary1} and \ref{lemma_catenarycase1}, we obtain the sets of factorizations of the common values of each of equations \eqref{6k+5eq1}-\eqref{6k+5_k-1equations_2}.
One can easily check that the largest catenary degree among these values is $6k+7$, attained for \eqref{6k+5minrelation0}, \eqref{6k+5eq2}, and \eqref{6k+5_k-1equations_2}, hence
$$ \textup{c}(S) = 6k+7.$$

Lastly, to obtain a minimal presentation, we apply Theorem \ref{theorem_minimalpresentation}. Similarly as always, we do not have the sets of factorizations of the common values of the minimal relations. However, we can see that $(6k+7)\binom{6k+6}{2} \notin \langle \binom{6k+9}{2} \rangle$, $(k+1)\binom{6k+9}{2} \notin \langle \binom{6k+8}{2} \rangle$, and that 
$ (3k+4)\binom{6k+7}{2} = (3k+3) \binom{6k+8}{2} \notin \langle \binom{6k+6}{2}, \binom{6k+9}{2} \rangle, $
as 
$$\frac{1}{3}(3k+4)\binom{6k+7}{2} = \frac{1}{3}(3k^2 + 10k + 8)\binom{6k+6}{2} - \frac{1}{3}(3k^2 + 4k +1)\binom{6k+9}{2},$$ where $\frac{1}{3}(3k^2 + 10k + 8)\binom{6k+6}{2} \in \textup{Ap}(\langle \frac{1}{3}\binom{6k+6}{2}, \frac{1}{3}\binom{6k+9}{2} \rangle,\frac{1}{3}\binom{6k+7}{2})$ (because $3k^2 + 10k + 8 < \frac{1}{3}\binom{6k+9}{2}$) . This means that we can construct the minimal presentation as always, and it has cardinality $2k+4$.

\subsection{Summary}

\begin{cor}\label{corfrobtriangular}
The Frobenius number of the numerical semigroup generated by the four consecutive triangular numbers $\binom{n+1}{2},\binom{n+2}{2}, \binom{n+3}{2}, \binom{n+4}{2}$ is as follows:
$$
\begin{cases}
\begin{aligned}
&\frac{1}{3} n^3 + \frac{11}{6} n^2 + \frac{5}{2} n - 1
    &&\text{if} \ \ n \equiv 0 \pmod{6},\ n \ge 6,\\
&\frac{1}{3} n^3 + \frac{5}{3} n^2 + \frac{10}{3} n + 1
    &&\text{if} \ \ n \equiv 2 \pmod{6},\ n \ge 2,\\
&\frac{1}{3} n^3 + \frac{3}{2} n^2 + \frac{13}{6} n
    &&\text{if} \ \ n \equiv 4 \pmod{6},\ n \ge 4,\\
&\frac{1}{2} n^3 + \frac{3}{2} n^2 - 3
    &&\text{if} \ \ n \equiv 1 \pmod{6},\ n \ge 1,\\
&\frac{1}{2} n^3 + 2 n^2 + \frac{1}{2} n - 4
    &&\text{if} \ \ n \equiv 3 \pmod{6},\ n \ge 3,\\
&\frac{1}{2} n^3 + 2 n^2 + \frac{5}{2} n
    &&\text{if} \ \ n \equiv 5 \pmod{6},\ n \ge 5.
\end{aligned}
\end{cases}
$$
 
\end{cor}

\begin{cor}\label{corgenustriangular}
The genus of the numerical semigroup generated by the four consecutive triangular numbers $\binom{n+1}{2},\binom{n+2}{2}, \binom{n+3}{2}, \binom{n+4}{2}$ is as follows:
$$
\begin{cases}
\begin{aligned}
&\frac{25}{144} n^3 +  n^2 + \frac{5}{4} n 
    &&\text{if} \ \ n \equiv 0 \pmod{6},\ n \ge 6,\\
&\frac{25}{144} n^3 +  n^2 + \frac{17}{12} n + \frac{7}{9}
    &&\text{if} \ \ n \equiv 2 \pmod{6},\ n \ge 2,\\
&\frac{25}{144} n^3 +  n^2 + \frac{13}{12} n - \frac{4}{9}
    &&\text{if} \ \ n \equiv 4 \pmod{6},\ n \ge 4,\\
&\frac{37}{144} n^3 + \frac{43}{48} n^2 + \frac{1}{48}n - \frac{169}{144}
    &&\text{if} \ \ n \equiv 1 \pmod{6},\ n \ge 1,\\
&\frac{37}{144} n^3 + \frac{49}{48} n^2 + \frac{3}{16}n - \frac{27}{16}
    &&\text{if} \ \ n \equiv 3 \pmod{6},\ n \ge 3,\\
&\frac{37}{144} n^3 + \frac{49}{48} n^2 + \frac{49}{48}n + \frac{37}{144}
    &&\text{if} \ \ n \equiv 5 \pmod{6},\ n \ge 5.
\end{aligned}
\end{cases}
$$

\end{cor}

\begin{cor}\label{corcatenarytriangular}
The catenary degree of the numerical semigroup generated by the four consecutive triangular numbers $\binom{n+1}{2},\binom{n+2}{2}, \binom{n+3}{2}, \binom{n+4}{2}$ is as follows:
$$
\begin{cases}
\begin{aligned}
&\frac{2}{3}n + 1 
    &&\text{if} \ \ n \equiv 0 \pmod{6},\ n \ge 6,\\
&\frac{2}{3}n + \frac{5}{3}
    &&\text{if} \ \ n \equiv 2 \pmod{6},\ n \ge 8,\\
&\frac{2}{3}n + \frac{7}{3}
    &&\text{if} \ \ n \equiv 4 \pmod{6},\ n \ge 4,\\
&n+2
    &&\text{if} \ \ n \equiv 1 \pmod{2},\ n \ge 3.\\
\end{aligned}
\end{cases}
$$
Moreover: $\textup{c} \left( \langle \binom{3}{2},\binom{4}{2},\binom{5}{2},\binom{6}{2} \rangle \right) = 10$.
\end{cor}

\begin{cor}\label{corbettitriangular}
The cardinality of the minimal presentation of the numerical semigroup generated by the four consecutive triangular numbers $\binom{n+1}{2},\binom{n+2}{2}, \binom{n+3}{2}, \binom{n+4}{2}$ is as follows:
$$
\begin{cases}
\begin{aligned}
&\frac{1}{3}n + 3
    &&\text{if} \ \ n \equiv 0 \pmod{6},\ n \ge 12,\\
&\frac{1}{3}n + \frac{10}{3}
    &&\text{if} \ \ n \equiv 2 \pmod{6},\ n \ge 8,\\
&\frac{1}{3}n + \frac{11}{3}
    &&\text{if} \ \ n \equiv 4 \pmod{6},\ n \ge 10,\\
&\frac{1}{2}n + \frac{9}{2}
    &&\text{if} \ \ n \equiv 1 \pmod{6},\ n \ge 7,\\
&\frac{1}{3}n + 3
    &&\text{if} \ \ n \equiv 3 \pmod{6},\ n \ge 9,\\
&\frac{1}{3}n + \frac{7}{3}
    &&\text{if} \ \ n \equiv 5 \pmod{6},\ n \ge 11.
\end{aligned}
\end{cases}
$$
Moreover: $\# \rho \left( \langle \binom{3}{2},\binom{4}{2},\binom{5}{2},\binom{6}{2} \rangle \right) = 1$, $\# \rho \left( \langle \binom{4}{2},\binom{6}{2},\binom{7}{2},\binom{8}{2} \rangle \right) = 1$, $\# \rho \left( \langle \binom{5}{2},\binom{6}{2},\binom{7}{2},\binom{8}{2} \rangle \right)  = 4$, $\#\rho \left( \langle \binom{6}{2},\binom{7}{2},\binom{8}{2},\binom{9}{2} \rangle \right) = 2$, $\# \rho \left( \langle \binom{7}{2},\binom{8}{2},\binom{9}{2},\binom{10}{2} \rangle \right) = 4$.    
 
\end{cor}

\section{Closing remarks}\label{closingremarks}

In \cite{Moscariello2015}, Moscariello demonstrated that the Frobenius number of numerical semigroups generated by the infinite sequences $n^2,(n+1)^2, (n+2)^2, \dots$ is asymptotically $O(n^2)$. This result was used in \cite{Arias2025} to show that the embedding dimension of such semigroups is asymptotically $O(n)$. In \cite{Lepilov2015}, the Frobenius number of the numerical semigroup generated by $n^2,(n+1)^2,(n+2)^2$ was determined, and it turned out to be given by a finite set of degree-three polynomials in $n$. This matches the results of this paper on four consecutive squares, as in this case we also obtain polynomials of degree three in $n$. Since for infinite sequences of squares the answer is asymptotically $O(n^2)$, the question arises whether there exists a positive integer $k$ such that $F(\langle n^2,(n+1)^2, \dots, (n+k)^2\rangle)$ is found by a finite set of quadratic polynomials in $n$. Using the following lower bound on the Frobenius number, proved in \cite{Killingbergtro2000} (see also \cite[Chapter 3]{RamirezAlfonsin2005}), the answer is negative.

\begin{thm}\label{lowerboundF}
    Let $a_1,\dots,a_m$ be positive integers such that $\gcd(a_1,\dots,a_m) = 1$. Then
    $$ F(\langle a_1, \dots,a_m \rangle) \geq ( (m-1)!a_1a_2 \cdots a_m)^{\frac{1}{m-1}} -\sum_{i=1}^m a_i .$$
\end{thm}

Substituting into Theorem \ref{lowerboundF}, we obtain 
$$ F(\langle n^2,(n+1)^2, \dots, (n+k)^2\rangle) \geq (k! (n(n+1) \cdots (n+k))^2)^{\frac{1}{k}} - \sum_{i = 0}^k (n+i)^2 \geq (k! (n^{k+1})^2)^{\frac{1}{k}} - \sum_{i = 0}^k (n+i)^2. $$
Hence, for $n$ large enough, the right-hand side is $\geq C \cdot (n^2)^{ \frac{k+1}{k}}$ for some positive constant $C$, which shows that the Frobenius number cannot be a quadratic polynomial in $n$. 

In view of the results of Sections \ref{geometricprocedure} and  \ref{bettielements}, we propose the following:

\begin{prob}
   Apply the methods from Sections \ref{geometricprocedure} and \ref{bettielements} to other families of numerical semigroups with embedding dimension four.  
\end{prob}

In Section \ref{bettielements}, we restricted ourselves to work under the hypotheses of Proposition \ref{proptrick}, since the families of numerical semigroups examined in Sections \ref{findingtheFnumber} and \ref{triangularnumbers} satisfy these assumptions. The author did not attempt to extend the procedure developed in Section \ref{bettielements} to all numerical semigroups with embedding dimension four. By Lemma \ref{lemma2cases}, any numerical semigroup with embedding dimension four satisfies the assumptions of Proposition \ref{proptrick} (for some arrangement of its generators) or $a_{ii}d_i = a_{jj}d_j$ and $a_{kk}d_k = a_{ll}d_l$ where $\{i,j,k,l\} = \{0,1,2,3\}$. As a next problem, we propose the following: 

\begin{prob}\label{problemcatenary}
    Formulate and prove a version of Proposition \ref{propminimalrelations} and of the procedure from Section \ref{bettielements} that applies to numerical semigroups with embedding dimension four for which $a_{ii}d_i = a_{jj}d_j$ and $a_{kk}d_k = a_{ll}d_l$ where $\{i,j,k,l\} = \{0,1,2,3\}$. 
\end{prob}

The tame degree is a central invariant in factorization theory (see \cite[Chapter 6]{AssiDAnnaGarciaSanchez2020} and \cite{Geroldinger2025}). Catenary and tame degrees are often considered together; however, the tame degree is harder to obtain. In particular, obtaining it requires the sets $\textup{Minimals}_{\leq} \varphi^{-1}(d_i + S)$. See \cite[Theorem 16]{AssiDAnnaGarciaSanchez2020}. Also, compare \cite[Proposition 70]{AssiDAnnaGarciaSanchez2020} with Remark \ref{remarkcatenary}.

\begin{prob}\label{problemtame}
    Develop a procedure for finding the tame degree of any numerical semigroup with embedding dimension four. 
\end{prob}

\section{Acknowledgments}
At the time of writing this paper, the author was a high school student. The author thanks his tutor, Tomasz Kowalczyk, for suggesting the topic of numerical semigroups. The joint work of the author and Tomasz Kowalczyk is a part of the individual care of highly mathematically talented highschool students held by Jagiellonian University. 

The author would like to also thank Pedro A. Garc{\'i}a{-}S{\'a}nchez and Alfred Geroldinger for their valuable comments and suggestions offered during our correspondence, throughout the development of this paper.

\bibliographystyle{plain}
\bibliography{refs.bib}

\vspace{5pt}

\noindent

\vspace{10pt}

\begin{small}

\noindent
Kazimierz Chomicz

\noindent
Institute of Mathematics

\noindent
Faculty of Mathematics and Computer Science

\noindent
Jagiellonian University

\noindent
ul. Łojasiewicza 6, 30-348 Kraków, Poland

\noindent
e-mail: kazikchomicz@gmail.com

\end{small}

\end{document}